\documentclass[letterpaper,10pt]{amsart}
\usepackage{tikz, tikz-cd, stmaryrd}
\usetikzlibrary{arrows}
\usepackage[font=footnotesize,labelfont=bf]{caption}
\usepackage{subcaption} 
\usepackage{blkarray}
\usepackage{enumitem}
\usepackage{bbold}
\usepackage{comment} 
\usepackage{latexsym,array,delarray,epsfig,setspace,mathtools,amssymb,mathrsfs}
\usepackage[colorlinks,backref=page]{hyperref}
\hypersetup{citecolor={blue}}
\usepackage{cleveref}
\usepackage{xcolor}

\definecolor{light}{gray}{.75}
\definecolor{med}{gray}{.5}
\definecolor{dark}{gray}{.25}

\numberwithin{equation}{section}

\newtheorem{theorem}{Theorem}
\numberwithin{theorem}{section}
\newtheorem{proposition}[theorem]{Proposition}

\newtheorem{lemma}[theorem]{Lemma}

\newtheorem{question}[theorem]{Question}
\theoremstyle{definition}
\newtheorem{definition}[theorem]{Definition}
\newtheorem{remark}[theorem]{Remark}

\newtheorem{example}[theorem]{Example}

\newcommand{\C}{{\mathbb C}}

\newcommand{\Q}{{\mathbb Q}}
\newcommand{\R}{{\mathbb R}}
\newcommand{\Z}{{\mathbb Z}}

\newcommand{\A}{{\mathbb A}}
\newcommand{\Aa}{\mathcal{A}}
\newcommand{\K}{{\mathbb K}}
\newcommand{\B}{\mathbb{B}}
\newcommand{\bb}{\mathbb{b}}
\newcommand{\E}{\mathcal{E}}

\newcommand{\HH}{\mathcal{H}}
\newcommand{\PP}{\mathcal{P}}

\newcommand{\Sp}{\textup{Sp}}

\newcommand{\hh}{\mathbb{h}}

\newcommand{\w}{{\sf w}}
\renewcommand{\v}{{\sf v}}

\newcommand{\fv}{\mathfrak{v}}

\newcommand{\supp}{\textup{supp}}
\newcommand{\ptconv}{\textup{pt-conv}}

\newcommand{\Spec}{\textup{Spec}}

\newcommand{\Hom}{\textup{Hom}}

\newcommand{\NN}{\mathcal{N}}
\newcommand{\MM}{\mathcal{M}}

\newcommand{\relint}{{\textup{rel-int}}}

\title{A construction of polyptych lattices via a pair of Gorenstein PL cones} 

\author{Laura Escobar}
\address{Mathematics Department, University of California Santa Cruz, USA} 
\email{lauraescobar@ucsc.edu}

\author{Megumi Harada}
\address{Dept.\ of Mathematics and Statistics, McMaster University, 
1280 Main Street West, 
Hamilton, Ontario L8S 4K1, Canada}
\email{haradam@mcmaster.ca}

\author{Christopher Manon} 
\address{Dept.\ of Mathematics, University of Kentucky, USA} 
\email{chris.manon@gmail.com}

\date{\today}


\begin{document}

\maketitle

\begin{abstract} 
The theory of polyptych lattices seeks to incorporate, into a single combinatorial framework, the piecewise-linear bijections (mutations) that appear in many theories that generalize toric geometry, such as Newton-Okounkov bodies, toric degenerations, cluster varieties, and other areas. The link to algebraic geometry comes from the accompanying notion of detropicalizations of polyptych lattices, and their associated compactifications. The main result of this manuscript gives a concrete construction of a strict dual pair $(\mathcal{E}, \mathcal{F})$ of polyptych lattices, which we call a \textbf{Gorenstein PL cone extension}.  To define our construction, we introduce the notion of Gorenstein PL cones, which are a polyptych analogue of a Gorenstein cone in the classical setting. If the original polyptych lattices are detropicalizable, then the new polyptych lattices $\mathcal{E},\mathcal{F}$ are also detropicalizable, via a simple explicit formula. Our Gorenstein PL cone extensions give a rich source of examples of strict dual pairs of polyptych lattices. In particular, any pair of Gorenstein-Fano polytopes $\Delta,\Delta'$ of dimension $n,n'$ respectively, gives rise to a strict dual pair $(\mathcal{E}(\Delta,\Delta'), \mathcal{F}(\Delta,\Delta'))$ of detropicalizable polyptych lattices of rank $n+n'+1$. Other examples can be built from cluster data. 
\end{abstract}

{
  \hypersetup{linkcolor=black}
  \tableofcontents
}

\section{Introduction}

The notion of polyptych lattices was first introduced in \cite{EscobarHaradaManon-PL}. The authors were motivated by the appearance of mutations (piecewise-linear bijections) in many theories that generalize toric geometry, such as Newton-Okounkov bodies \cite{KavehManon-Siaga, EscobarHarada}, toric degenerations and toric flat families \cite{KavehManon-PL, KavehManon-PL-part2, ChristophersenIlten2016}, $T$-varieties \cite{Ilten2012, IltenVollmert}, cluster varieties \cite{RW, BCMNC, FriasMedina-Magee}, among other areas.  In \cite{EscobarHaradaManon-PL} we connected our definition of polyptych lattices to algebraic geometry through our notion of a detropicalization $\Aa_\MM$ of a polyptych lattice $\MM$ \cite[Definition 6.3]{EscobarHaradaManon-PL} and an accompanying theory of their associated compactifications $X_{\Aa_\MM}(\PP)$ with respect to a PL polytope $\PP$ \cite[Section 7.2]{EscobarHaradaManon-PL}. Furthermore, we established some geometric properties of these compactifications; in particular, \cite[Theorem 7.19]{EscobarHaradaManon-PL} shows that if $\Aa_\MM$ is a UFD, then $X_{\Aa_\MM}(\PP)$ has finitely generated Cox ring.

We have already seen that the theory of polyptych lattices yields interesting examples that are connected to other research areas.  Detropicalizations of polyptych lattices are a rich source of algebras which admit a combinatorial description in a similar fashion to cluster algebras.  Indeed, in forthcoming work \cite{FMEHMM} we will show that, in certain cases, a cluster algebra gives rise to a polyptych lattice which is self-dual. The construction which we introduce in the present paper provides an additional way to produce many new explicit examples of detropicalizations (see Example \ref{example: classical GF pairs} and Example \ref{example: cluster}). Moreover, examples that have already been studied are connected to other research areas. 
In \cite{CookEscobarHaradaManon2024}
we studied a specific family of rank $2$ examples; these are self-dual in an appropriate sense, and we saw that their detropicalizations can be UFDs and also be non-UFDs. 
More recently, Fujita and Higashitani introduced and studied a polyptych lattice coming from  marked chain-order polytopes; these are related to the Gelfand–Tsetlin and FFLV polytopes arising in representation theory \cite{FujitaHigashitani}. In addition, Oda studied in \cite{Oda2025} the rank-$2$ polyptych lattices from \cite{CookEscobarHaradaManon2024} in the context of log-Calabi-Yau geometry.

The theory of polyptych lattices is still in early stages, and the main result in this manuscript describes a concrete construction, which we call a \textbf{Gorenstein PL cone extension}, which gives a rich source of new examples of polyptych lattices.  As part of our construction, we introduce the definition of a \textbf{Gorenstein PL cone} (Definition~\ref{definition: Gorenstein PL cone}), which is a polyptych generalization of the notion of a Gorenstein cone in the classical setting. We give the rough statement below; the precise statement is Theorem~\ref{theorem: main}.

\medskip
\noindent \textbf{Main Theorem.} (Theorem~\ref{theorem: main}) 
Let $(\MM_1,\NN_1),(\MM_2,\NN_2)$ be strict dual pairs of finite polyptych lattices, of ranks $r_1$ and $r_2$ respectively with $r_1,r_2\ge 2$. Let $\PP_1,\PP_2$ be strongly convex Gorenstein PL cones in $\MM_1,\MM_2$ respectively in the sense of Definition~\ref{definition: Gorenstein PL cone}. 
Then: 
\begin{itemize} 
\item there exists a strict dual pair $(\mathcal{E},\mathcal{F}, \v,\w)$ of polyptych lattices of rank $r := r_1+r_2-1$ constructed from this data, 
\item if $\MM_i,\NN_i$ for $i=1,2$ are detropicalizable, with detropicalizations that are equipped with convex adapted bases, then there is an explicit construction for  detropicalizations of $\mathcal{E}$ and $\mathcal{F}$ in terms of those of $\MM_i, \NN_i$, and there exists a construction of convex adapted bases for the detropicalizations of $\mathcal{E}$ and $\mathcal{F}$. 
\end{itemize}

We emphasize that our formula for the detropicalizations of $\mathcal{E}, \mathcal{F}$ are very concrete and computable; indeed, if $\Aa_{\MM_1}, \Aa_{\NN_2}$ are detropicalizations of $\MM_1, \NN_2$ respectively, then our detropicalization $\Aa_{\mathcal{E}}$ of $\mathcal{E}$ is given as $\Aa_{\mathcal{E}} := \Aa_{\PP_1}\otimes \Aa_{\NN_2}/\langle \mathbb{h}\rangle$ where $\Aa_{\PP_1}$ is a subalgebra of $\Aa_{\MM_1}$ defined in terms of the PL cone $\PP_1$, and $\mathbb{h}$ is an element of $\Aa_{\PP_1}\otimes \Aa_{\NN_2}$ also concretely constructed from the PL cone data. The formula for $\Aa_{\mathcal{F}}$ is similar. Both Example \ref{example: classical GF pairs} and Example \ref{example: cluster} illustrate the explicit nature of our computations.

We proceed to show in Section~\ref{section: examples} that the polyptych lattices $\MM_{d,r}$ studied in \cite[Section 8]{EscobarHaradaManon-PL} can be constructed by our Gorenstein PL cone extension, in a special case where the initial data consist of trivial polyptych lattices and the PL cones are positive orthants.  
The fact that such a simple choice of initial data yields a non-trivial strict dual pair of polyptych lattices suggests that our Gorenstein PL cone extension construction can yield a rich class of as-yet unexplored examples. 
Indeed, we give several potentially rich classes of examples in Section~\ref{section: examples}. In fact, our construction yields a new example of a polyptych lattice for any pair $\Delta_1, \Delta_2$ of Gorenstein-Fano polytopes. As mentioned above, we also have examples arising from cluster algebras (see Example~\ref{example: cluster}).

The structure of the paper is as follows. In Section~\ref{sec: background} we prepare some preliminaries, and in particular, define a Gorenstein PL cone. The construction of the Gorenstein PL cone extension, the result of which is a strict dual pair $(\mathcal{E},\mathcal{F},\v,\w)$ of polyptych lattices, is given in Section~\ref{sec: cone extensions}.  It is shown that (under the hypothesis that the original data satisfies detropicalizability assumptions)  both $\mathcal{E}$ and $\mathcal{F}$ are detropicalizable, and explicit formulas are given for a choice of such. We also show that, if the original Gorenstein PL cones have the property that their duals are also Gorenstein, then the cone extensions also come equipped with Gorenstein PL cones. This opens the possibility of iterating our Gorenstein PL cone construction. Finally, in Section~\ref{section: examples} we suggest families of examples worth further study, and record some questions which we intend to pursue in future work.

\subsection*{Statement on AI Use.} 
The authors acknowledge the use of AI tools to assist with proofreading this manuscript and making improvements to exposition. Generative AI was not used to formulate the theorems or to (initially) construct the proofs.

\subsection*{Acknowledgements}
This material is partly based upon work supported by the National Science Foundation under Grant No. DMS-1928930, while the authors were in residence at the Simons Laufer Mathematical Sciences Institute in Berkeley, California, during the summer of 2025. We thank the Institute for its support and hospitality. LE was supported in part by a Fields Research Fellowship from the Fields Institute for Research in the Mathematical Sciences. LE is also supported by an NSF CAREER grant DMS-2142656 and DMS-2521270. MH was supported by a Canada Research Chair Award (Tier 2) and NSERC Discovery Grant 2019-06567. CM is supported by NSF DMS grant 2101911 and 2501468.

\section{Preliminaries and first results}\label{sec: background}

In this section, we collect some definitions and preliminaries. For details we refer to \cite{EscobarHaradaManon-PL}. 
We begin with some key definitions. 

\begin{definition}\cite[Definition 2.1]{EscobarHaradaManon-PL}\label{definition of polyptich lattice}
Let $r$ be a positive integer and let $\Z \subseteq F \subseteq \R$ be a subring of $\R$. A \textbf{polyptych lattice (PL) of rank $r$ over $F$} is a pair $\mathcal{M} := (\{M_\alpha\}_{\alpha \in \mathcal{I}}, \{\mu_{\alpha,\beta}: M_\alpha \to M_\beta\}_{\alpha,\beta \in \mathcal{I}})$ consisting of a collection $\{M_\alpha\}_{\alpha \in \mathcal{I}}$ of free $F$-modules, each of rank $r$ and indexed by a set $\mathcal{I}$, and a collection of piecewise-linear maps $\mu_{\alpha,\beta}: M_\alpha \to M_\beta$ for every pair $(\alpha,\beta)$ of indices, satisfying the following conditions:
\begin{enumerate} 
\item $\mu_{\alpha,\alpha} = \mathrm{Id}_{M_\alpha}$ is the identity map for all $\alpha \in \mathcal{I}$, 
\item $\mu_{\alpha,\beta} = \mu_{\beta,\alpha}^{-1}$ for all pairs $\alpha,\beta \in \mathcal{I}$, and 
\item $\mu_{\beta,\gamma} \circ \mu_{\alpha,\beta} = \mu_{\alpha,\gamma}$ for all triples $\alpha,\beta,\gamma \in \mathcal{I}$. 
\end{enumerate} 
Note in particular that the requirement (2) above implies that all the maps $\mu_{\alpha,\beta}$ are invertible. We call the maps $\mu_{\alpha,\beta}$ \textbf{mutations}, and we call  $M_\alpha$ a \textbf{chart} of $\mathcal{M}$.
The \textbf{$\alpha$-th chart map} is the association 
\begin{equation*}\label{eq: def alpha coordinate} 
\pi_\alpha: \MM \to M_\alpha, \quad m \mapsto m_\alpha,
\end{equation*}
taking an equivalence class $m \in \MM$ to its unique representative $m_\alpha$ in $M_\alpha$.

 When $\mathcal{I}$ is finite, we say $\mathcal{M}$ is a \textbf{finite} polyptych lattice. 
\end{definition} 

We sometimes denote the indexing set $\mathcal{I}$ associated to a polyptych lattice $\MM = (\{M_\alpha\}_{\alpha \in \mathcal{I}}, \{\mu_{\alpha,\beta}\}_{\alpha,\beta \in \mathcal{I}})$ by $\pi(\MM) := \mathcal{I}$. By slight abuse of notation we also use $\pi(\MM)$ to denote the set of charts $\{M_\alpha\mid \alpha\in\mathcal{I}\}$.

We also need the PL (polyptych lattice) analogue of a dual lattice. 

\begin{definition}\cite[Definition 3.1]{EscobarHaradaManon-PL}\label{definition: space of points}
Let $\MM = (\{M_\alpha\}_{\alpha \in \mathcal{I}}, \{\mu_{\alpha,\beta}:M_\alpha \to M_\beta\}_{\alpha,\beta \in \mathcal{I}}\})$ be a polyptych lattice over $F$. A \textbf{point of $\MM$} is a  function $p: \MM \to F$ such that 
\begin{equation}\label{eq: def point min} 
p(m) + p(m') = \min\{p(m +_\alpha m') \,\mid\, \alpha \in \mathcal{I} \} \, \, \textup{ for all } \, m, m' \in \MM
\end{equation} 
and
\begin{equation}\label{eq: def point F homog}
p(\lambda m) = \lambda p(m) \, \, \textup{ for all } \, m \in \MM, \lambda \in F_{\geq 0}. 
\end{equation} 
The set of all such $p: \MM \to F$ is called \textbf{the space of points of $\MM$} and denoted $\Sp(\MM)$. For $F \subset F'$, an \textbf{$F'$-point of $\MM$} is a point of $\MM \otimes_F F'$. We denote the set of $F'$-points as $\Sp_{F'}(\MM)$. (For the base ring $F$, we often drop the subscript.) 
\end{definition} 

\begin{definition}\cite[Definition 3.9]{EscobarHaradaManon-PL}
For $\alpha \in \mathcal{I}$, we write $p_\alpha:=p\circ\pi_\alpha^{-1}$ and
we let $\Sp(\MM,\alpha) \subset \Sp(\MM)$ denote the subset of points of $\MM$ such that $p_\alpha: M_\alpha \to F$ is $F$-linear. 
\end{definition}

The following is the PL analogue of the classical pairing between a lattice and its dual.

\begin{definition}\cite[Definition 4.1]{EscobarHaradaManon-PL}
\label{def_dual}
Let $\MM, \NN$ be finite polyptych lattices of rank $r$ over $F$ with associated PL fans $\Sigma(\MM)$ and $\Sigma(\NN)$ respectively. Let $F \subseteq F' \subseteq \R$. We say that a pair of continuous maps $\v: \MM_\R \to \Sp_\R(\NN)$ and $\w: \NN_\R \to \Sp_\R(\MM)$ is a \textbf{strict dual $F'$-pairing} if: 
\begin{enumerate} 
\item[(1)] $\v$ and $\w$ restrict respectively to maps $\v: \MM_{F'} \to \Sp_{F'}(\NN)$ and $\w: \NN_{F'} \to \Sp_{F'}(\MM)$ (by abuse of notation we denote the restrictions also by $\v$ and $\w$),
\item[(2)] $\v(m)(n) = \w(n)(m)$ for all $n \in \NN_{F'}, m \in \MM_{F'}$, 
\item[(3)] $\v: \MM_{F'} \to \Sp_{F'}(\NN)$ and $\w: \NN_{F'} \to \Sp_{F'}(\MM)$ are both bijections, and
\item[(4)] the preimages $\v^{-1}\Sp_{\R}(\NN,\gamma)$ (respectively $\w^{-1}\Sp_{\R}(\MM,\alpha)$) are precisely the maximal-dimensional cones of $\Sigma(\MM)$ (respectively $\Sigma(\NN)$), as $\gamma$ ranges over $\pi(\NN)$ (respectively, $\alpha$ ranges over $\pi(\MM)$), giving a bijection between the index set $\pi(\NN)=\mathcal{J}$ of charts of $\NN$ and the set of maximal-dimensional faces of $\Sigma(\MM)$ (respectively, between $\pi(\MM)$ and the maximal-dimensional faces of $\Sigma(\NN)$). 
\end{enumerate} 
In the above setting, 
we say that $(\MM, \NN, \v,\w)$ is a \textbf{strict dual ($F'$)-pair} of polyptych lattices; 
more informally, we also refer to $\NN_{F'}$ as a \textbf{strict ($F'$-)dual to $\MM$}.
If $(\MM,\NN,\v,\w)$ satisfies only the axioms (1)-(3), we call it a \textbf{dual $(F')$-pair}, and $\NN_{F'}$ a \textbf{($F'$-) dual to} $\MM$. If $\MM$ has a strict dual (respectively dual) pairing with itself, we say that $\MM$ is \textbf{strictly self-dual} (respectively \textbf{self-dual}). 
\end{definition}

\begin{remark}\label{remark: abuse of notation for strict duals} 
The Definition~\ref{def_dual} is stated in some generality; however, the terminology and notation can become cumbersome. For our discussion here and in the following, we note that: (i) In practice, we are often interested in the situation where $F = F'=\Z$. \emph{In particular, in this manuscript, our polyptych lattices are over $\Z$. In the definition of the space of points, we take $F=\Z$ as our coefficients unless otherwise specified. }
(ii) When the coefficients $F$ and/or $F'$ are understood from context, we frequently drop the reference to the coefficients. 
(iii) As noted in condition (1) of Definition~\ref{def_dual}, by slight abuse of notation we use the same $\v$ and $\w$ to refer to the restrictions $\v: \MM \to \Sp(\NN), \w: \NN \to \Sp(\MM)$. Indeed, for purposes of developing the theory, it is often the above restrictions to $\MM$ and $\NN$ respectively which take center stage. Motivated by this, and continuing with the slight abuse of notation, in this manuscript we will often refer to the data $(\MM,\NN,\v: \MM \to \Sp(\NN), \w: \NN \to \Sp(\MM))$ as a strict dual pair, with the understanding that $\v$ and $\w$ naturally extend to $\R$ coefficients.  
\end{remark}

We also recall the definition of PL cone.
\begin{definition}\cite[Definition 2.8]{EscobarHaradaManon-PL}\label{definition: PL cone} 
Let $\MM$ be a polyptych lattice over $F$. 
A \textbf{PL cone over $F$} is a subset $\mathcal{C} $ of $\MM_{\R}$ such that $\pi_\alpha(\mathcal{C}) \subseteq M_\alpha \otimes_{F} \R$ is an $F$-rational polyhedral cone for each $\alpha \in \pi(\MM)$. 
\end{definition}

  As in \cite{EscobarHaradaManon-PL}, in this manuscript we focus our attention on \textbf{finite} polyptych lattices. 
 In this case, we have the \textbf{mutation fan of $\MM$}, which is the PL fan $\Sigma(\MM)$ obtained by considering the common refinement of the regions of linearity of the $\mu_{\alpha,\beta}$ where $\alpha,\beta\in\mathcal{I}$ (cf. \cite[Lemma 2.11]{EscobarHaradaManon-PL}). 

 \begin{definition}
     \label{definition: lineality space}
     The \textbf{lineality space} of $\MM$ is the intersection of all cones in the mutation fan $\Sigma(\MM)$.
 \end{definition}

Recall that for every chart $\alpha \in \pi(\MM)$ we have a chart addition in $\MM$ with respect to the $\alpha$-th chart $M_\alpha$ (cf. \cite[Equation (2)]{EscobarHaradaManon-PL}). The following is a straightforward consequence of definitions. 

\begin{lemma}\label{lemma: properties of lineality space}
    Following the notation and terminology above, suppose $\mathbb{l}$ lies in the lineality space of $\MM$. Then: 
    \begin{enumerate} 
    \item Addition with any scalar multiple of $\mathbb{l}$, is well-defined for any $m \in \MM$. More precisely, for any $\lambda \in \Z$, $m \in \MM$, and $\alpha, \alpha' \in \pi(\MM)$, we have $m +_\alpha \lambda \mathbb{l} = m+_{\alpha'} \lambda \mathbb{l}$. Since the chart additions are all the same, we denote by $m + \lambda \mathbb{l}$ the result of addition with respect to any chart. 
    \item for any $u\in \MM$, $p \in \Sp(\MM)$ and $\lambda\in \Z$, we have $$p(u+\lambda\mathbb{l})
= 
 p(u) + \lambda p(\mathbb{l}).$$  
 \end{enumerate}
 The same statements hold when replacing $\MM$ and $\Z$ by $(\MM)_{F'}, \Sp_{F'}(\MM)$ and $F'$, for $\Z \subseteq F' \subseteq \R$.
\end{lemma}

\begin{proof}
    The claim (1) holds because addition is well-defined within any cone of $\Sigma(\MM)$, and the lineality space is contained in every cone.
    For (2), note that since $p$ is a point on $\MM$, by Definition~\ref{definition: space of points} it follows that $p(u+u') = \min_{\alpha}\{p(u+_\alpha u')\}$ for any $u,u'$ lying in a single cone of $\Sigma(\MM)$. By (1), all the chart additions agree for $u' = \lambda \mathbb{l}$. Thus, the RHS is a minimum over a singleton set and hence
\begin{equation*}
p(u+\lambda\mathbb{l})
= p(u) + p(\lambda \mathbb{l}) = 
 p(u) + \lambda p(\mathbb{l}).
\end{equation*}
The extension by tensoring with $F'$ is straightforward. 
\end{proof}

We first establish some notation and hypotheses for the PL cones we will consider. 
The \textbf{PL-half-space} associated to $p\in\Sp(\MM)$ and $a\in\Z$ is
\begin{equation*}\label{eq: def M half space} 
\HH_{p, a} := \{ m \in \MM_\R \, \mid \, p(m) \geq a\} \subset \MM_\R. 
\end{equation*} 
  Suppose 
$\PP \subset \MM_{\R}$ is the PL cone 
\begin{equation}\label{eq: def PPi}
\PP = \bigcap_{j=1}^{s} \HH_{p_j, 0} = \{m \in \MM_{\R} \, \mid \, p_j(m) \geq 0, \, \textup{ for all } \, j \in [s]\}
\end{equation}
where $p_j \in \Sp(\MM)$. In the setting where we have a strict dual pair $(\MM,\NN,\v,\w)$, when the $p_j$ may be written as $p_j=\w(n_j)$ for some $n_j \in \NN$, then the set $\{n_1,\cdots, n_{s}\} \subset \NN$ can be thought of as the PL analogue of inner normal vectors, and $s$ is a positive integer recording the number of such normal vectors.

The following lemma makes explicit in the PL setting that the interior of PL cones can be obtained by changing inequalities to strict inequalities. 

\begin{lemma}\label{lemma: interior of PL cone}
Let $(\MM,\NN,\v,\w)$ be a strict dual pair of finite polyptych lattices and let
$\PP = \bigcap_{j=1}^{s} \HH_{p_j,0} \subset \MM_\R$ be a PL cone as
in~\eqref{eq: def PPi}, with $p_j \neq 0$ for all $j \in [s]$. Then the interior of $\PP$ is
\begin{equation}\label{eq: interior of PL cone}
\mathrm{int}(\PP) = \{ m \in \MM_\R \, \mid \, p_j(m) > 0 \textup{ for all } j \in [s] \}.
\end{equation}
In particular, a lattice point $m \in \MM$ lies in $\mathrm{int}(\PP)$ if and
only if $p_j(m) \geq 1$ for all $j \in [s]$.
\end{lemma}

\begin{proof}
Since the interior of a finite intersection is the intersection of the
interiors, it suffices to show $\mathrm{int}(\HH_{p,0}) = \{p > 0\}$ for a
single $0 \neq p \in \Sp(\MM)$.  Note that for the purposes of this argument we are considering $p$ as an element of $\Sp_\R(\MM)$. From the assumptions on points, it follows that $p$ is 
continuous. Moreover, each chart $\pi_\alpha: \MM_\R \to M_\alpha \otimes \R$ is a homeomorphism. Thus it suffices to show $\mathrm{int}(\pi_\alpha(\HH_{p,0})) = \mathrm{int}(\{(\pi_\alpha^{-1})^*p \geq 0\}) = \{(\pi_\alpha^{-1})^*p > 0\}$ in $M_\alpha \otimes \R$. Since $(\pi_\alpha^{-1})^*p$ is a minimum of finitely many linear functions on $M_\alpha \otimes \R$, the result follows. 
The last statement follows because $p_j$ takes integer values on $\MM$ by definition.
\end{proof}

We will also need the following. 
\begin{definition}\label{definition: strongly convex}
Let $(\MM,\NN,\v,\w)$ be a strict dual pair of finite polyptych lattices over $\Z$ and let
$\PP \subset \MM_\R$ be a PL cone as in~\eqref{eq: def PPi}. 
We say that $\PP$ is \textbf{strongly chart-convex} if for every $\alpha \in \pi(\MM)$, the chart image $\pi_\alpha(\PP) \subset M_\alpha \otimes \R$ is strongly convex in the classical sense, i.e., $\pi_\alpha(\PP)$ contains no nonzero linear subspace.
\end{definition} 

The following will be used below. 
\begin{lemma}\label{lemma: strongly chart convex}
Following the notation of Definition~\ref{definition: strongly convex}, if $\PP$ is strongly chart-convex and $\dim(\PP) \geq 2$, then $s \geq 2$. 
\end{lemma} 

\begin{proof} 
It suffices to see that if $s=1$, then $\PP$ is not strongly chart-convex. Suppose $\PP = \HH_{\w(n),0}$ for $n \in \NN$. Since $(\MM,\NN,\v,\w)$ is a strict dual pair, by \cite[Lemma 4.2]{EscobarHaradaManon-PL} there exists some chart $M_\alpha$ on which $(\pi_\alpha^{-1})^*\w(n)$ is linear. Then $\pi_\alpha(\PP) = \pi_\alpha(\HH_{n,0}) = \{(\pi_\alpha^{-1})^*\w(n) \geq 0\}$ is a half-space, which is not strongly convex if $\dim(\PP) \geq 2$. Thus $s$ cannot be $1$, and therefore $s \geq 2$.  
\end{proof}

We now define one of the main objects in this paper. 

\begin{definition}\label{definition: Gorenstein PL cone} 
Suppose $(\MM,\NN,\v,\w)$ is a strict dual pair of finite polyptych lattices over $\Z$. 
We will say that a PL cone $\PP$ (together with the data of its defining elements $\{n_1,\cdots,n_s\} \subset \NN$) as in~\eqref{eq: def PPi} is a \textbf{Gorenstein PL cone} if the following are satisfied: 
\begin{enumerate} 
\item[(P-1)] For each $j \in [s]$, the set $\relint(F_j) = \{m\in\MM_\R: \w(n_t)(m) > 0 \textup{ for } t \neq j, \, \w(n_j)(m)=0\}$, i.e., the interior of the facet $F_j= \{m\in\PP: \w(n_j)(m)=0\}$ of $\PP$ corresponding to $n_j$, is non-empty and codim-$1$ in $\MM_\R$. 

\item[(P-2)] There exists $\mathbb{g} \in \MM$ with the properties that: 
\begin{enumerate} 
\item $\mathbb{g}$ is in the lineality space of $\MM$, and in particular, both $\mathbb{g}$ and its negative $-\mathbb{g}$ are contained in every cone of $\Sigma(\MM)$ 
\item for any $j \in [s]$, we have $\w(n_j)(\mathbb{g}) = 1$.  
\end{enumerate} 
\end{enumerate}
\end{definition} 

The following is immediate from the definition, and in particular, property (P-2)(b). 

\begin{lemma}\label{lemma: nonzero and primitive} 
Following the notation of Definition~\ref{definition: Gorenstein PL cone}, each $n_j$ is non-zero and primitive. 
\end{lemma} 

We will also use the following.

\begin{lemma}\label{proposition: P is full-dimensional}
    Following the notation of Definition~\ref{definition: Gorenstein PL cone}, suppose $\PP \subset \MM_\R$ is a Gorenstein PL cone. Then $\PP$ is full-dimensional and $\mathbb{g}\in \mathrm{int}(\PP)$.
\end{lemma}

\begin{proof}
    By Lemma~\ref{lemma: nonzero and primitive}, the interior of $\PP$ is given by Lemma~\ref{lemma: interior of PL cone}. By part (b) of (P-2) we have that $\mathbb{g}\in\mathrm{int}(\PP)$ since $\w(n_j)(\mathbb{g})>0$ for all $j\in[s]$. Thus, $\PP$ is a subset of $\MM_\R$ with non‑empty interior so it must be full‑dimensional.
\end{proof}

\begin{remark}  
 It may be useful for some readers to recall that there are multiple viewpoints on the Gorenstein condition. Our perspective stems from that of \cite{BrunsRomer}, which points out that the condition for an affine semigroup ring $\C[S]$ (where $S$ is given as the lattice points in a cone $C$ over a polytope) to be Gorenstein is equivalent to the condition that $\mathrm{int}(C) \cap S$ is of the form $y + S$ for some $y \in S$. 
In Lemma~\ref{lemma: Gorenstein def equivalence} we will see the condition just stated is equivalent to (P-2) in Definition~\ref{definition: Gorenstein PL cone} above. We have chosen to state the definition in terms of (P-2) because it is more convenient for our arguments.

\end{remark}

The classical version of the following lemma is well known. A minor adjustment to the argument is necessary in our PL setting. 

\begin{lemma}\label{lemma: Gorenstein def equivalence}
In the setting of Definition~\ref{definition: Gorenstein PL cone}, the condition (P-2) holds if and only if there exists $m_\PP \in \mathrm{int}(\PP) \cap \MM$ such that $m_\PP$ is in the lineality space of $\MM$, and, for every $m \in \mathrm{int}(\PP) \cap \MM$, there exists $m' \in \PP \cap \MM$ with $m = m_\PP + m'$. 
\end{lemma} 

\begin{proof} 
We begin with the forward direction of the equivalence. Set $m_\PP:=\mathbb{g}$. By assumption, $m_\PP$ is in the lineality space. Suppose $m \in \mathrm{int}(\PP) \cap \MM$. A lattice point lies in the interior of $\PP$ precisely if $\w(n_j)(m) \geq 1$ for all $j$. Set $m' = m - m_\PP$ where the difference is well-defined since $m_\PP$ lies in the lineality space. We compute $\w(n_j)(m')=\w(n_j)(m-m_\PP) = \w(n_j)(m) - \w(n_j)(m_\PP) = \w(n_j)(m) - 1$ using Lemma~\ref{lemma: properties of lineality space}. Since $m$ lies in the interior of $\PP$, we must have $\w(n_j)(m) \geq 1$, from which it follows $\w(n_j)(m')\geq 0$. Thus $m' \in \PP$ as required. 

For the reverse implication, suppose $m_\PP$ with the stated properties is given. Set $\mathbb{g}:=m_\PP$. Again by assumption, $\mathbb{g}$ is in the lineality space. Since $\mathbb g$ is in $\mathrm{int}(\PP)$ we know $\w(n_j)(\mathbb g)\ge1$ for all $i$. To see the reverse inequality, first fix $j$.  Since $n_j$ is primitive, there exists $u\in\MM$ with $\w(n_j)(u)=1$. Also there exists, and we may choose, a lattice point $v$ in the relative interior of the facet $F_j$, so $\w(n_i)(v)>0$ for $i\ne j$ and $\w(n_j)(v)=0$. Now choose $N>0, N \in \Z$ such that $\w(n_i)(u)+N > 0$ for all $i\neq j$. Choose a chart $\alpha \in \pi(\MM)$ for which $\w(n_j)(u +_\alpha Nv)$ achieves the minimum among $\{\w(n_j)(u +_\beta Nv): \beta \in \pi(\MM)\}$. 
Set $w:=u+_\alpha Nv$. By the definition of point (\cite[Definition 3.1]{EscobarHaradaManon-PL}) we have $\w(n_j)(w)=1$.
Note also that for $i\neq j$ we have that 
\[
\w(n_i)(w)\ge \min_{\beta\in\pi(\MM)}\{\w(n_i)(u+_\beta Nv)\}=\w(n_i)(u)+N\w(n_i)(v)\ge \w(n_i)(u)+N>0
.\]
Then $w:=u+_\alpha Nv$ is an interior lattice point. By assumption on $\mathbb{g}=m_\PP$, $w-\mathbb{g}$ lies in $\PP$, which by Lemma~\ref{lemma: properties of lineality space} means $\w(n_j)(w-\mathbb{g}) = \w(n_j)(w)-\w(n_j)(\mathbb{g})=1 - \w(n_j)(\mathbb{g}) \geq 0$. This forces $\w(n_j)(\mathbb{g}) \leq 1$. Since we have both inequalities, we conclude $\w(n_j)(\mathbb{g})=1$, as desired. 
\end{proof}

 Following the above discussion, we will say that a rational polyhedral cone $C$ in $M_\alpha \otimes \R$ is Gorenstein in the classical sense, if $\mathrm{int}(C)\cap M_\alpha$ is of the form $y + (C \cap M_\alpha)$ for some $y \in \mathrm{int}(C) \cap M_\alpha$.
The next lemma shows that our terminology in Definition~\ref{definition: Gorenstein PL cone} is compatible with the classical notion in each of its charts. 

\begin{lemma} 
Let $\PP$ be a PL cone as in~\eqref{eq: def PPi}. Suppose $\PP$ is Gorenstein in the sense of Definition~\ref{definition: Gorenstein PL cone}. Then all of its chart images are Gorenstein, i.e., for all $\alpha \in \pi(\MM)$, the image $\pi_\alpha(\PP) \subset M_\alpha \otimes \R$ is a Gorenstein cone in the classical sense. 
\end{lemma} 

\begin{proof} 
 From Lemma~\ref{lemma: Gorenstein def equivalence} we know that there exists $m_\PP$ in the lineality space of $\MM$ with the property that any $m \in \mathrm{int}(\PP) \cap \MM$ can be written as $m_\PP + m'$ for some $m' \in \PP \cap \MM$.
 First note that since 
 \[
 \w(n_j)(m_\PP+m')=1+\w(n_j)(m')>0
 \]
 for any $j\in[s]$ and $m' \in \PP \cap \MM$,
 $\pi_\alpha(m_\PP)+\pi_\alpha(\PP)\cap M_\alpha$ is contained in $\mathrm{int}(\pi_\alpha(\PP))\cap M_\alpha$.
 For the converse containment, take $\pi_\alpha$ of both sides of the equation in Lemma~\ref{lemma: Gorenstein def equivalence}. This means $\pi_\alpha(m) = \pi_\alpha(m_\PP) + \pi_\alpha(m')$. Here we use that $\pi_\alpha$ preserves addition with an element in the lineality space. The map $\pi_\alpha$ is a homeomorphism and takes $\MM$ to the lattice $M_\alpha$, so $\pi_\alpha(m_\PP)$ satisfies the condition needed to make $\pi_\alpha(\PP)$ a classical Gorenstein cone. 
\end{proof}

We will also need a PL analogue of dual cones. Unlike the classical case, in our PL setting, the definition requires a choice of strict dual. 

\begin{definition}\label{definition: dual of PL cone}
Let $(\MM,\NN,\v,\w)$ be a strict dual pair of polyptych lattices in the sense of Definition~\ref{def_dual}. 
Let $\PP \subset \MM_\R$ be a PL cone as in~\eqref{eq: def PPi}. We define the \textbf{dual PL cone to $\PP$ (with respect to $\NN$)} as 
\begin{equation}\label{eq: def dual PL cone} 
\PP^\vee := \{n \in \NN_\R\, \mid \, \w(n)(u) \geq 0, \textup{ for all} \, u \in \PP\} \subset \NN_\R.
\end{equation} 
\end{definition}

The lemma below justifies the terminology.

\begin{lemma}\label{lemma: dual of PL cone is a PL cone}
Let $(\MM,\NN,\v,\w)$ be a strict dual pair of finite polyptych lattices over $\Z$ and let
$\PP \subset \MM_\R$ be a PL cone as in~\eqref{eq: def PPi}. Then:
\begin{enumerate}
\item The dual $\PP^\vee$ of Definition~\ref{definition: dual of PL cone} is a PL cone in $\NN_\R$.
\item If the $n_j$ defining $\PP$ lie in $\NN$, then $\PP^\vee$ is cut out by PL half-spaces
associated to elements of $\MM$.
\item $(\PP^\vee)^\vee=\PP$.
\end{enumerate}
\end{lemma}

\begin{proof} 
We will show (1) and (2) together. To do so, we wish to find a finite set $S$ of elements in $\MM$ such that $\PP^\vee = \bigcap_{u' \in S} \HH_{\v(u'), 0}$. First, 
by Definition~\ref{definition: dual of PL cone} we immediately obtain 
\begin{equation}\label{eq: dual as intersection}
  \PP^\vee=\bigcap_{u\in\PP}\HH_{\v(u),0}.
\end{equation}

Choose a chart $\alpha_0\in\pi(\MM)$. The PL cone $\PP$ is defined in \eqref{eq: def PPi} as an intersection of PL half-spaces $\HH_{\w(n_j),0}$ where $n_j \in \NN$. Since $\w(n_j)$ is a point in $\Sp(\MM)$, the chart image $\pi_{\alpha_0}(\HH_{\w(n_j),0})
=\{\w(n_j)_{\alpha_0}\ge 0\}$ is an intersection of linear half-spaces in $M_{\alpha_0} \otimes \R$. By intersecting over all $j$, $1 \leq j \leq s$, it follows that 
$\pi_{\alpha_0}(\PP)$ is also an intersection of linear half-spaces and hence a rational polyhedral cone in $M_{\alpha_0}\otimes \R$. Let $g_1,\dots,g_\ell\in M_{\alpha_0}$ be a set of cone generators of $\pi_{\alpha_0}(\PP)$ (note that we may without loss of generality assume the $g_i$ are primitive lattice elements). 
Define $u'_i:=\pi_{\alpha_0}^{-1}(g_i)\in \MM$, for $1 \leq i \leq \ell$. We claim
\begin{equation}\label{eq: dual finite}
  \PP^\vee=\bigcap_{i=1}^\ell \HH_{\v(u'_i),0}.
\end{equation}
To see this, note first that the inclusion of the LHS of~\eqref{eq: dual finite} in the RHS is immediate from \eqref{eq: dual as intersection}. For the reverse inclusion, 
suppose $n$ is in the RHS of~\eqref{eq: dual finite} and let $u\in \PP$. We need to prove that $\w(n)(u) \geq 0$. Consider $\pi_{\alpha_0}(u) \in \pi_{\alpha_0}(\PP)$. Since $\{g_1,\cdots,g_\ell\}$ were chosen to be cone generators, we know there exist $\lambda_i \in \R_{\geq 0}$ such that $\pi_{\alpha_0}(u) = \sum_i \lambda_i g_i \in M_{\alpha_0} \otimes \R$. Since $n \in \mathcal{H}_{\v(u'_i),0}$ for all $i$, this implies $\v(u'_i)(n)=\w(n)(u'_i)=\w(n)(\pi_{\alpha_0}^{-1}(g_i)) = (\pi_{\alpha_0}^{-1})^*\w(n)(g_i) \geq 0$ for all $i$. Since $\w(n) \in \Sp_\R(\MM)$, the pullback $(\pi_{\alpha_0}^{-1})^*\w(n)$ is a min-combination of linear functions on $M_{\alpha_0} \otimes \R$, and it follows that if $(\pi_{\alpha_0}^{-1})^*\w(n)(g_i) \geq 0$ for all $i$, then $(\pi_{\alpha_0}^{-1})^*\w(n)(\sum_i \lambda_i g_i) = (\pi_{\alpha_0}^{-1})^*\w(n)(\pi_{\alpha_0}(u)) = \w(n)(u) \geq 0$. This proves the reverse inclusion. 
Setting $S := \{u'_1,\cdots,u'_\ell\} \subset \MM$ proves both (1) and (2).

Now we prove (3). 
By definition, $(\PP^\vee)^\vee=\{u\in\MM_\R : \w(n)(u)\ge 0\text{ for all } n\in\PP^\vee\}
$. Each $u\in\PP$ satisfies $\w(n)(u)\ge 0$ for all
$n\in \PP^\vee$ by definition of the dual, so it is immediate that $\PP\subseteq(\PP^\vee)^\vee$. Next recall that $\PP = \bigcap_{j=1}^s \mathcal{H}_{\w(n_j),0}$ by definition, and in particular it follows that $n_1,\cdots, n_s$ lie in $\PP^\vee$. Therefore, $(\PP^\vee)^\vee$ must lie in $\bigcap_j \HH_{\w(n_j),0}$, but this is $\PP$.  The result follows. 
\end{proof}

 We close the background section with some technical lemmas concerning the set of points $\Sp(\MM)$ and strict dual pairs (following notation as in Remark~\ref{remark: abuse of notation for strict duals}). Suppose that $(\MM, \NN, \v: \MM \to \Sp(\NN), \w: \NN \to \Sp(\MM))$ form a strict dual pair of finite polyptych lattices over $\Z$.

We begin with a refinement of \cite[Lemma 3.7]{EscobarHaradaManon-PL}.

\begin{lemma}\label{lemma: res cone surj}
Let $(\MM, \NN, \v: \MM \to \Sp(\NN), \w: \NN \to \Sp(\MM))$ be a pair of strict dual finite polyptych lattices over $\Z$ and $C$ a maximal cone in $\Sigma(\MM)$. The restriction map $L_C: \Sp(\MM) \to \Hom(C\cap\MM, \Z)$ is bijective. 
\end{lemma}

\begin{proof}
    By \cite[Lemma 3.7]{EscobarHaradaManon-PL} we know $L_C$ is injective so we only need to prove surjectivity. Let $\varphi \in \Hom(C\cap\MM,\Z)$ be given. To prove the lemma, we need to find $p \in \Sp(\MM)$ such that $L_C(p)=\varphi$. Let $\bar{\varphi}$ denote the unique extension of $\varphi$ to $\Hom(C, \R)$. Since $\MM,\NN$ are finite and $(\MM, \NN, \v, \w)$ is a strict dual pair, we know from \cite[Proposition 4.12]{EscobarHaradaManon-PL} that there exists $\tilde{p} \in \Sp_\R(\MM)$ with $L_{C}(\tilde{p}) = \bar{\varphi}$. To complete the argument, it would suffice to show that $\tilde{p}$ takes integral values on the lattice. We do so by finding $\alpha \in \pi(\MM)$ such that $\tilde{p}\circ \pi_\alpha^{-1}(M_\alpha)\subseteq\Z$.
    
    To see this, recall that $\MM_\R$ is full since $\MM_\R$ has a strict dual \cite[Lemma 4.2]{EscobarHaradaManon-PL},  so there exists some $\alpha \in \pi(\MM)$ such that $\tilde{p} \circ \pi_\alpha^{-1}$ is linear on $(M_\alpha)_\R$. Since $L_{C}(\tilde{p}) = \bar{\varphi}$ we know the values of $\tilde{p} \circ \pi_\alpha^{-1}$ agree with those of $\bar{\varphi} \circ \pi_\alpha^{-1}$ on $\pi_\alpha(C)$, and in particular, are integral on lattice points in $\pi_\alpha(C)$. Thus it suffices now to show that if $D$ is a full-dimensional cone in an $\R$-vector space $V$ with lattice $M$ and $\hat{\varphi} \in \Hom(V, \R)$ has the property that $\hat{\varphi}$ takes integral values on $D \cap M $, then $\hat{\varphi}$ takes integral values on $M$. Fix a basis $\{e_1,\ldots, e_n\}$ for $M$. Since $D$ is full-dimensional, there exists a lattice point $u \in M$ in the interior of $D$.
    Now, for sufficiently large $N \in \Z$, we have that both $Nu \in D \cap M$ and also $Nu + e_i \in D \cap M$ for all $i$. Then by assumption $\hat{\varphi}(Nu) \in \Z$ and $\hat{\varphi}(Nu + e_i) \in \Z$, and since $\hat{\varphi}$ is linear we conclude $\hat{\varphi}(e_i) \in \Z$ for all $i$. Since the $e_i$ were chosen to be a basis of $M$, the claim follows. 
\end{proof}

The following is simple linear algebra but it will be useful to state it explicitly. 

\begin{lemma}\label{lemma: injectivity} 
Following the notation of Lemma~\ref{lemma: res cone surj}, suppose $r$ is the rank of $\MM$ and $\{m_1,\cdots,m_r\} \subset C \cap \MM$ is a subset which maps to a $\Q$-basis of $(M_\alpha)_\Q$ for each $\alpha$. Then the evaluation map $\Hom(C\cap \MM,\Z) \to \Z^r, f \mapsto (f(m_1),\cdots,f(m_r))$ is injective. 
\end{lemma}

 For the next lemma we define 
\begin{equation*}\label{eq: def psi alpha}
\psi_\alpha: M_\alpha \to \Hom(C_\alpha \cap \NN,\Z), \quad m_\alpha \in M_\alpha \mapsto\left( y \in C_\alpha\cap\NN \mapsto (\w(y)_\alpha(m_\alpha) \right).
\end{equation*}
We have the following. 
\begin{lemma}\label{lemma: psi alpha properties}
Following the notation above, we have: 
\begin{enumerate} 
\item $\psi_\alpha$ is well-defined, i.e. the image is in $\Hom(C_\alpha \cap \NN,\Z)$,  
\item 
the diagram 
\begin{equation}\label{eq: psi alpha comm diagram}
\begin{tikzcd}
\MM \arrow[r, "\v"] \arrow[d, "\pi_\alpha"]
& \Sp(\NN) \arrow[d, "L_{C_\alpha}"] \\
M_\alpha \arrow[r, "\psi_\alpha"]
& \Hom(C_\alpha \cap \NN, \Z)
\end{tikzcd}
\end{equation}
commutes, and
\item $\psi_\alpha: M_\alpha \to \Hom(C_\alpha \cap \NN,\Z)$ is a linear isomorphism of lattices. 
\end{enumerate} 
\end{lemma} 

\begin{proof} 
For (1), let $y,y' \in C_\alpha$. Since $C_\alpha$ is a cone in $\Sigma(\NN)$, addition is well-defined on $C_\alpha$ and we may compute for any $m_\alpha \in M_\alpha$ that 
$$
\w(y+y')_\alpha(m_\alpha) = \w(y+y')(\pi_\alpha^{-1}(m_\alpha)) = \v(\pi_\alpha^{-1}(m_\alpha))(y+y') = \v(\pi_\alpha^{-1}(m_\alpha))(y) + \v(\pi_\alpha^{-1}(m_\alpha))(y')
$$
where the last equality uses that $\v(\pi_\alpha^{-1}(m_\alpha))$ is a point. Moreover, the last quantity is equal to $(\pi_\alpha^{-1})^*\w(y)(m_\alpha)+(\pi_\alpha^{-1})^*\w(y')(m_\alpha)$. This shows $\psi_\alpha(m_\alpha)$ is indeed a homomorphism, as claimed. 
 
To show (2), let $m \in \MM$. We wish to show $L_{C_\alpha} \circ \v(m) = \psi_\alpha \circ \pi_\alpha(m) \in \Hom(C_\alpha \cap \NN,\Z)$.
To see this, let $y \in C_\alpha \cap \NN$. We may compute the LHS as 
\begin{equation}
L_{C_\alpha} \circ \v(m)(y)  = \v(m) \vert_{C_\alpha} (y) 
 = \v(m)(y)  = \w(y)(m).
\end{equation}
The RHS may be computed using the definition of $\psi_\alpha$ as 
\begin{equation}
  (\psi_\alpha \circ \pi_\alpha(m))(y)  = (\pi_\alpha^{-1})^* \w(y)(\pi_\alpha(m)) 
 = \w(y)(\pi_\alpha^{-1} \circ \pi_\alpha(m)) 
 = \w(y)(m) 
\end{equation}
which shows that the RHS and LHS agree on $y \in C_\alpha \cap \NN$, as desired. 

Finally, for (3),  
we first show that $\psi_\alpha$ is linear. Recall that, from the assumption of strict duality, we have $C_\alpha = \w^{-1}(\Sp(\MM, \alpha))$ for $\alpha \in \pi(\MM)$ Definition~\ref{def_dual}.  Thus, for any $y \in C_\alpha$, we have $\w(y) \in \Sp(\MM,\alpha)$, so  $(\pi_\alpha^{-1})^* \w(y)$ is linear on $M_\alpha$. Let $m_\alpha, m'_\alpha \in M_\alpha$. Then 
\begin{equation}
    \begin{split} 
\psi_\alpha(m_\alpha + m'_\alpha)(y)
& = (\pi_\alpha^{-1})^*\w(y)(m_\alpha + m'_\alpha) \\
& = (\pi_\alpha^{-1})^*\w(y)(m_\alpha) + (\pi_\alpha^{-1})^*\w(y)(m'_\alpha) \\
& = \psi_\alpha(m_\alpha)(y) + \psi_\alpha(m'_\alpha)(y)
    \end{split} 
\end{equation}
as desired. To see that $\psi_\alpha$ is an isomorphism, notice that the diagram in \eqref{eq: psi alpha comm diagram} commutes. We know from Lemma~\ref{lemma: res cone surj} that $L_{C_\alpha}$ is a bijection, and $\v$ and $\pi_\alpha$ are also bijections. This implies $\psi_\alpha$ is also a bijection, as desired.   
\end{proof} 

The following lemma a direct consequence of  
Lemma~\ref{lemma: psi alpha properties}(3) but it is useful to state it in this form. 

\begin{lemma}\label{lemma: M alpha addition formula}
Following the notation of the previous lemmas, let $\alpha \in \pi(\MM)$ be a coordinate chart of $\MM$. Given $m, m' \in \MM$, we have 
\begin{equation}\label{eq: M alpha addition formula} 
\v(m) \vert_{C_\alpha} + \v(m') \vert_{C_\alpha} = \v(m +_\alpha m') \vert_{C_\alpha}. 
\end{equation} 
\end{lemma}

\section{Construction of a Gorenstein PL  cone extension}\label{sec: cone extensions}

The goal of this section is to construct a new strict dual pair of polyptych lattices $\mathcal{E}$ and $\mathcal{F}$ over $\Z$, using the data of \emph{two} pairs of strict dual polyptych lattices over $\Z$, and two Gorenstein PL cones. We will refer to this as a \textbf{Gorenstein PL cone extension}, or simply a \textbf{cone extension} (of the given PL data). Moreover, we will show that certain properties of the original data are inherited by $\mathcal{E}$ and $\mathcal{F}$. For instance, in Section~\ref{sec: E and F detrop}, we will show that, under some extra hypotheses on the detropicalizability of the original data, the PLs $\mathcal{E}$ and $\mathcal{F}$ are also detropicalizable, where their detropicalizations can be explicitly described in terms of the detropicalizations of the original data.

\subsection{Preliminary notation and definitions}\label{subsec: prelim notation and def} 

We first set some notation. Let $(\MM_1, \NN_1, \v_1: \MM_1 \to \Sp(\NN_1), \w_1: \NN_1 \to \Sp(\MM_1))$ and $(\MM_2, \NN_2, \v_2: \MM_2 \to \Sp(\NN_2), \w_2: \NN_2 \to \Sp(\MM_2))$ be two strict dual pairs over $\Z$ of finite polyptych lattices. In particular, we assume that $\MM_i, \NN_i, i=1,2$ are polyptych lattices defined over $\Z$. We also denote by $r_1 \in \Z_{>0}$ the rank of $\MM_1$ and $\NN_1$, and by $r_2 \in \Z_{>0}$ the rank of $\MM_2$ and $\NN_2$. We assume that $r_i \geq 2$ for both $i=1,2$.

We assume to be given a PL cone $\PP_i$ in $(\MM_i)_\R$ for both $i=1$ and $i=2$, with associated elements $\{n^i_j\}_{j\in [s_i]}$ as in~\eqref{eq: def PPi}. We assume that both $\PP_1$ and $\PP_2$ are strongly chart-convex in the sense of Definition~\ref{definition: strongly convex}. We additionally assume that both $\PP_1$ and $\PP_2$ (with their associated sets $\{n^i_j\}$ for $i=1,2$) are Gorenstein PL cones in the sense of Definition~\ref{definition: Gorenstein PL cone}. 
Associated to each cone $\PP_i$ we also define the following function: 
\begin{equation}\label{eq: def supp tildeP_i}
\Psi_{\PP_i}: \MM_i \to \Z, \quad \Psi_{\PP_i} := \min_{j \in [s_i]} \{\w_i(n^i_j)\}
\end{equation} 
for $i=1,2$. (By abuse of notation, we will also denote by $\Psi_{\PP_i}$ the natural extension to a function $(\MM_i)_\R \to \R$. The meaning should be clear from context.) With this definition, it is clear from~\eqref{eq: def PPi} that $\PP_i = \{\Psi_{\PP_i} \geq 0\}$, its interior is $\mathrm{int}(\PP_i) = \{\Psi_{\PP_i} >0\}$, and its boundary is $\partial \PP_i = \{\Psi_{\PP_i} = 0\}$, where the subsets are considered in $(\MM_i)_\R$. Moreover, being the min-combination of a finite collection of points, $\Psi_{\PP_i}$ is a piecewise-linear function on $(\MM_i)_{\R}$. 

For $i=1,2$ we fix once and for all an element $\mathbb{g}_i\in\mathrm{int}(\PP_i)\cap \MM_i$ satisfying (P-2).
We will use the following lemma repeatedly.  
Note that by Lemma~\ref{lemma: properties of lineality space}, the expression $u+\lambda\mathbb{g}_i \in \MM_i$ is well-defined for any scalar $\lambda$. 

\begin{lemma}\label{lemma: psi of a sum}
    Suppose $i=1,2$, $j\in[s_i]$,  $u\in \MM_i$, and $\lambda\in \Z$. Then $$\w_i(n^i_j)(u+\lambda\mathbb{g}_i)
= 
 \w_i(n^i_j)(u) + \lambda \quad \textup{and} \quad \Psi_{\PP_i}(u + \lambda\mathbb{g}_i) = \Psi_{\PP_i}(u) + \lambda.$$ The same statement holds when replacing $\MM_i$ and $\Z$ by $(\MM_i)_\R$ and $\R$ respectively. 
\end{lemma}

\begin{proof}
The first equation is a straightforward consequence of Lemma~\ref{lemma: properties of lineality space} using the assumption $\w_i(n^i_j)(\mathbb{g}_i)=1$ for all $j \in [s_i]$. For the second equation we have 
\begin{equation*}
\Psi_{\PP_i}(u+\lambda\mathbb{g}_i) %& = \min_{j \in [s_i]} \{ \w_i(n^i_j)(u) + \lambda \w_i(n^i_j)(\mathbb{g}_i)\}\\
 = \min_{j \in [s_i]} \{\w_i(n^i_j)(u) + \lambda\} 
 = \left(\min_{j \in [s_i]} \{\w_i(n^i_j)(u)\}\right) + \lambda 
 = \Psi_{\PP_i}(u) + \lambda
\end{equation*}
by \eqref{eq: def supp tildeP_i}. The extension by tensoring with $\R$ is straightforward. 
\end{proof}

We define 
\begin{equation}\label{eq: def X}
X := \{ (u,v) \in \MM_1 \times \NN_2 \, \mid \, \w_2(v)(\mathbb{g}_2) = \Psi_{\PP_1}(u) \} \subset \MM_1 \times \NN_2
\end{equation}
and the same equation interpreted over $\R$ yields $X_\R := \{(u,v) \in (\MM_1 \times \NN_2)_\R \, \mid \, \w_2(v)(\mathbb{g}_2)=\Psi_{\PP_1}(u)\} \subset (\MM_1 \times \NN_2)_\R$. 
We also define 
\begin{equation}\label{eq: def Y} 
Y := \{(u,v) \in \MM_1 \times \NN_2 \, \mid \, 0 = \Psi_{\PP_1}(u) \}\subset \MM_1 \times \NN_2
\end{equation} 
and it is clear that by extending to $\R$ coefficients we have $Y_\R = \partial \PP_1 \times \NN_2$.

The next lemma provides a way to take arbitrary elements in $\MM_1 \times \NN_2$ and map them to either $X$ or $Y$. Let $(u,v) \in \MM_1 \times \NN_2$. Define two maps $\MM_1 \times \NN_2 \to \MM_1 \times \NN_2$ as follows: 
\begin{equation}\label{eq: def phi} 
\phi(u,v) := (u - \Psi_{\PP_1}(u)\cdot \mathbb{g}_1, v)
\end{equation} 
and
\begin{equation}\label{eq: def phi inverse} 
\psi(u,v) := (u + (\w_2(v)(\mathbb{g}_2) - \Psi_{\PP_1}(u))\cdot \mathbb{g}_1, v).
\end{equation}

\begin{lemma}\label{lem-2sets}
Consider the maps $\phi, \psi$ defined in~\eqref{eq: def phi} and~\eqref{eq: def phi inverse}. Then
\begin{enumerate} 
\item $\phi$ maps $\MM_1\times \NN_2$ onto $Y$,
\item $\phi$, when restricted to $X$, is a bijection from $X$ to $Y$, and $\phi \vert_X$ can be expressed alternatively as 
\begin{equation}\label{eq: def phi on X}
\phi(u,v) := (u-\w_2(v)(\mathbb{g}_2)\cdot \mathbb{g}_1, v),
\end{equation}
\item $\phi$, when restricted to $Y$, is the identity on $Y$, 
\item $\psi$ maps $\MM_1 \times \NN_2$ onto $X$, 
\item $\psi$, when restricted to $Y$, is a bijection from $Y$ to $X$ and is the inverse of $\phi$ from $X$ to $Y$ (i.e. $\psi \vert_Y = (\phi \vert_X)^{-1}$), 
\item $\psi$, when restricted to $X$, is the identity on $X$.
 \end{enumerate} 
 When the definitions of $\phi$ and $\psi$ are extended to $\R$ coefficients, the analogous statements (1)-(6) hold over $\R$ for $X_\R, Y_\R, (\MM_1 \times \NN_2)_\R$. 

\end{lemma}

\begin{proof}
This is a computation using Lemma~\ref{lemma: psi of a sum}, ~\eqref{eq: def X}, and~\eqref{eq: def Y}.
\end{proof}

Swapping the role of $\PP_1$ now with $\PP_2$, and replacing $\MM_1$ with $\MM_2$ and $\NN_2$ with $\NN_1$, we define 
\begin{equation}\label{eq: def X vee}
X^\vee := \{ (y,x) \in \MM_2 \times \NN_1 \, \mid \, \w_1(x)(\mathbb{g}_1) = \Psi_{\PP_2}(y) \} \subset \MM_2 \times \NN_1
\end{equation}
and  
\begin{equation}\label{eq: def Y vee} 
Y^\vee := \{(y,x) \in \MM_2 \times \NN_1 \, \mid \, 0 = \Psi_{\PP_2}(y) \} \subset \MM_2 \times \NN_1.
\end{equation} 
As above, by extending to $\R$ coefficients we have analogous definitions of $X^\vee$ and $Y^\vee = \partial \PP_2 \times \NN_1$. 
We also define maps $\phi^\vee, \psi^\vee$ in the analogous manner as follows: 
\begin{equation}\label{eq: phi vee} 
\phi^\vee: \MM_2 \times \NN_1 \to \MM_2 \times \NN_1, \quad (y,x) \mapsto (y - \Psi_{\PP_2}(y) \cdot \mathbb{g}_2,x)
\end{equation} 
and 
\begin{equation}\label{eq: psi vee} 
\psi^\vee: \MM_2 \times \NN_1 \to \MM_2 \times \NN_1, \quad (y,x) \mapsto  (y + (\w_1(x)(\mathbb{g}_1)-\Psi_{\PP_2}(y)) \cdot \mathbb{g}_2, x).
\end{equation} 

The same proof as for Lemma~\ref{lem-2sets} gives the following. 

\begin{lemma}\label{lemma: phivee psivee bijections}
Consider the maps $\phi^\vee, \psi^\vee$ defined in~\eqref{eq: phi vee} and~\eqref{eq: psi vee}. Then
\begin{enumerate} 
\item $\phi^\vee$ maps $\MM_2\times \NN_1$ onto $Y^\vee$,
\item $\phi^\vee$, when restricted to $X^\vee$, is a bijection from $X^\vee$ to $Y^\vee$, and $\phi^\vee \vert_{X^\vee}$ can be expressed alternatively as 
\begin{equation}\label{eq: def phi on X vee}
\phi^\vee(y,x) := (y-\w_1(x)(\mathbb{g}_1)\cdot \mathbb{g}_2, x),
\end{equation}
\item $\phi^\vee$, when restricted to $Y^\vee$, is the identity on $Y^\vee$,
\item $\psi^\vee$ maps $\MM_2 \times \NN_1$ onto $X^\vee$, 
\item $\psi^\vee$, when restricted to $Y^\vee$, is a bijection from $Y^\vee$ to $X^\vee$ and is the inverse of $\phi^\vee$ from $X^\vee$ to $Y^\vee$ (i.e. $\psi^\vee \vert_{Y^\vee} = (\phi^\vee \vert_{X^\vee})^{-1}$), 
\item $\psi^\vee$, when restricted to $X^\vee$, is the identity on $X^\vee$.
 \end{enumerate} 
 When the definitions of $\phi^\vee$ and $\psi^\vee$ are extended to $\R$ coefficients, the analogous statements (1)-(6) hold over $\R$ for $X^\vee_\R, Y^\vee_\R, (\MM_2 \times \NN_1)_\R$. 
\end{lemma}

From \cite[Lemma 4.6]{EscobarHaradaManon-PL} we know that $\MM_1 \times \NN_2$ and $\MM_2 \times \NN_1$ form a strict dual pair. We use $\langle -, - \rangle$ to denote the strict pairing between $\MM_1 \times \NN_2$ and $\MM_2 \times \NN_1$; this pairing can be written explicitly in multiple ways.
Indeed, for $(u,v) \in \MM_1 \times \NN_2, (y,x) \in \MM_2 \times \NN_1$ we have 
\begin{equation}\label{eq: explicit pairing}
\langle (u,v),(y,x) \rangle  = 
\v_1(u)(x) + \v_2(y)(v) 
 = \v_1(u)(x) + \w_2(v)(y) 
 = \w_1(x)(u) + \v_2(y)(v)  
 = \w_1(x)(u) + \w_2(v)(y).
\end{equation} 
The pairing is symmetric in the indices and we sometimes write $\langle (y,x),(u,v)\rangle$ (which is equal to $\langle (u,v),(y,x)\rangle$). 
Our next lemma shows a relationship between this pairing and our bijections $\phi$ and $\phi^\vee$ discussed above.

\begin{lemma}\label{lem-pairing}
Let $\langle -, - \rangle: (\MM_1 \times \NN_2) \times (\MM_2 \times \NN_1) \to \Z$ denote the strict dual pairing defined in four equivalent ways in~\eqref{eq: explicit pairing} above. 
For any $(u,v) \in X$ and $(y,x) \in X^\vee$ we have
\begin{equation}\label{eq: pairing X Xvee}
\langle (u,v), \phi^\vee(y,x) \rangle = \langle \phi(u,v), (y,x)\rangle. 
\end{equation}
Moreover, the same statement holds when the pairing is extended to $\R$ coefficients to be defined on $X$ and $X^\vee$. 
\end{lemma}

\begin{proof} 
 This is a direct computation. 
\end{proof}

We next set some notation for the coordinate charts and the maximal cones in the mutation fans of both $\MM_1 \times \NN_2$ and $\MM_2 \times \NN_1$.  Recall first that, by the definition of strict duality Definition~\ref{def_dual}, there is a bijective correspondence between the set of coordinate charts on a polyptych lattice, and the set of maximal cones of the mutation fan of its strict dual. 
In what follows we use the notation $\tau \in \pi(\MM_1)$ and $\sigma \in \pi(\NN_2)$ for the (indices of the) coordinate charts in $\MM_1, \NN_2$ respectively. 
By the strict duality mentioned above, these charts correspond to maximal cones in the corresponding strict duals; specifically, there is a maximal cone $C_\tau := \w_1^{-1}(\Sp(\MM_1, \tau)) \subset (\NN_1)_\R$ (respectively $C_\sigma := \v_2^{-1}(\Sp(\NN_2, \sigma)) \subset (\MM_2)_\R$) in the mutation fan $\Sigma(\NN_1)$ (respectively $\Sigma(\MM_2)$) corresponding to $\tau$ (respectively $\sigma$).  In particular, this means that the pair $(\sigma,\tau)$ corresponds to the maximal cone $C_\sigma \times C_\tau$ in $\Sigma(\MM_2 \times \NN_1)$. 

We will use the following lemmas multiple times. 

\begin{lemma}\label{lemma: X and Y meet interior of cones}
Let $X^\vee_\R, Y^\vee_\R$ be as defined in~\eqref{eq: def X vee} and~\eqref{eq: def Y vee}. Let $\sigma \in \pi(\NN_2), \tau \in \pi(\MM_1)$ and $C_\sigma \times C_\tau \subset (\MM_2 \times \NN_1)_\R$  the corresponding maximal cone, as above. Then: 
\begin{enumerate} 
\item $X^\vee_\R$ meets the interior of $C_\sigma \times C_\tau$ in $(\MM_2 \times \NN_1)_\R$, 
\item $\{\Psi_{\PP_2}=0\} \subset (\MM_2)_\R$ meets the interior of $C_\sigma$ in $(\MM_2)_\R$, 
\item $Y^\vee_\R$ meets the interior of $C_\sigma \times C_\tau$ in $(\MM_2 \times \NN_1)_\R$. 
\end{enumerate} 
Moreover, by appropriately changing indices, the same statements hold for $X_\R, Y_\R$ as defined in~\eqref{eq: def X} and~\eqref{eq: def Y}, $\alpha \in \pi(\NN_1), \beta \in \pi(\MM_2)$, and $C_\alpha \times C_\beta$ in $(\MM_1 \times \NN_2)_\R$. 
\end{lemma} 

\begin{proof}  
Let $(y_0,x_0) \in \mathrm{int}(C_\sigma \times C_\tau)$. By~Lemma~\ref{lemma: psi of a sum} it follows that $(y_0 + (\w_1(x_0)(\mathbb{g}_1) - \Psi_{\PP_2}(y_0))\mathbb{g}_2, x_0)$ lies in $X^\vee$. Since $\mathbb{g}_2$ lies in the lineality space of $\MM_2$, $(y_0 + (\w_1(x_0)(\mathbb{g}_1) - \Psi_{\PP_2}(y_0))\mathbb{g}_2, x_0)$ is also in the interior of $C_\sigma\times C_\tau$. This proves (1). A similar argument shows that for $y_0 \in \mathrm{int}(C_\sigma)$, the element $y'_0 = y_0 - \Psi_{\PP_2}(y_0)\mathbb{g}_2$ also lies in the interior and satisfies $\Psi_{\PP_2}(y'_0)=0$, showing (2). The last item (3) follows immediately from (2). 
\end{proof}

\begin{lemma}
Let $X^\vee_\R, Y^\vee_\R$ be as defined in~\eqref{eq: def X vee} and~\eqref{eq: def Y vee}. Then 
$X^\vee_\R$ and $Y^\vee_\R$ are each a finite union of codimension-$1$ cones in $(\MM_2 \times \NN_1)_\R$. Moreover, by appropriately changing indices, the same statements hold for $X_\R$ and $Y_\R$ in $(\MM_1 \times \NN_2)_\R$. 
\end{lemma}

\begin{proof} 
Note that $\{\Psi_{\PP_2}=0\}$ is the boundary $\partial \PP_2$ of $\PP_2 = \{\Psi_{\PP_2} \geq 0\}$ in $(\MM_2)_\R$. From the definition~\eqref{eq: def Y vee} it follows that $Y^\vee_\R = \partial \PP_2 \times (\NN_1)_\R$. Also notice that $\PP_2 \times (\NN_1)_\R$ is a full-dimensional cone in $(\MM_2 \times \NN_1)_\R$, since $\PP_2$ is full-dimensional in $(\MM_2)_\R$ by Proposition~\ref{proposition: P is full-dimensional}. Hence its boundary $Y^\vee_\R = \partial \PP_2 \times (\NN_1)_\R$ is a union of codimension-$1$ cones. From~Lemma~\ref{lemma: phivee psivee bijections} and~\eqref{eq: psi vee} we know that $\psi^\vee$ is a piecewise-linear bijection from $Y^\vee_\R$ to $X^\vee_\R$. From this it follows that if $Y^\vee_\R$ is a finite union of codimension-$1$ cones, so is $X^\vee_\R$. 
\end{proof}

We now refine our analysis to the domains of linearity of $\Psi_{\PP_2}$, which is defined in~\eqref{eq: def supp tildeP_i} as a minimum of the set $\{\w_2(n^2_j)\}_{j \in [s_2]}$ of points defining $\PP_2$. We define $C_{\sigma, j}$ to be the subcone of $C_{\sigma}$ on which the minimum is achieved by $\w_2(n^2_j)$. More precisely we have 
\begin{equation}\label{eq: C sigma' tau' j}
C_{\sigma,j} := \{y \in C_{\sigma} \, \mid \, \w_2(n^2_j)(y) \leq \w_2(n^2_t)(y), \forall t \in [s_2]\} \subset (\MM_2)_\R.
\end{equation}
Note that it may not be true for a given index $j \in [s_2]$ that the subset $C_{\sigma,j}$ is full-dimensional within $C_{\sigma}$. To address this, we first define 
\begin{equation}\label{eq: def A sigma tau}
A_{\sigma} := \{ j \in [s_2] \, \mid \, C_{\sigma,j} \, \textup{ is full-dimensional in} \, C_\sigma \} 
\end{equation} 
and henceforth, when we write the notation $C_{\sigma,j}$, we are assuming that $j \in A_{\sigma}$. Note that $j \in A_{\sigma}$ also implies that $C_{\sigma,j}$ meets the interior of $C_\sigma$. 

The following lemma motivates the definition of the set $A_\sigma$.
\begin{lemma}\label{lemma: union of csigmaj}
   For any $\sigma\in\pi(\NN_2)$ we have that $C_\sigma=\bigcup_{j\in A_\sigma}C_{\sigma,j}$. Moreover, $A_\sigma$ is in fact the unique minimal set for which the equation holds. 
\end{lemma}
\begin{proof}
    Let $\mathfrak{A}=\bigcup_{j\in A_\sigma}C_{\sigma,j}$ and $\mathfrak{B}=\bigcup_{j\notin A_\sigma}C_{\sigma,j}$.
    By definition of $C_{\sigma,j}$ we have that $\mathfrak{A}$ and $\mathfrak{B}$ are closed sets in $C_\sigma$ and that $C_\sigma=\mathfrak{A}\cup\mathfrak{B}$.
    To see that $C_\sigma=\mathfrak{A}$, note that if $j\notin A_\sigma$ then $C_{\sigma,j}$ has empty interior since it is contained in a proper linear subspace of the span of $C_\sigma$. Since a finite union of sets with these properties has an empty interior, $\mathfrak{B}$ has empty interior. Since $\mathfrak{A}$ is closed it follows that $C_\sigma=\overline{C_\sigma\setminus\mathfrak{B}}\subseteq \overline{\mathfrak{A}}=\mathfrak{A}\subset C_\sigma$. 
Finally, we show that $A_\sigma$ is the unique minimal subset of $[s_2]$ for which the equation holds. Let $j\ne j'$ be in $A_\sigma$. By the hypothesis (P-1), we know $n^2_j\neq n^2_{j'}$, hence $\w_2(n^2_j)\ne\w_2(n^2_{j'})$. By Lemma~\ref{lemma: res cone surj}, their restrictions to $C_\sigma$ are distinct linear functionals. Thus $C_{\sigma,j}\cap C_{\sigma,j'}\subseteq \{\w_2(n^2_j)=\w_2(n^2_{j'})\}\cap C_\sigma$ is contained in a hyperplane. Since $C_{\sigma,j}$ is full-dimensional it cannot be contained in the finite union $\bigcup_{j'\ne j}C_{\sigma,j'}$. Hence, no proper subset of $A_\sigma$ covers $C_\sigma$. Finally, suppose $A''\subseteq[s_2]$ covers $C_\sigma$; by the argument above, $A''\cap A_\sigma$ covers $C_\sigma$, which implies $A''\cap A_\sigma=A_\sigma$. This shows that $A_\sigma$ is the unique minimal subset for which the equation holds. 
\end{proof}

The next results form another step toward the definition of the coordinate charts of the cone extension. 

\begin{lemma}\label{lemma: A sigma equivalence}
Let $\sigma \in \pi(\NN_2)$ and $\tau\in \pi(\MM_1)$.
For any $j\in[s_2]$, $(C_{\sigma,j} \times C_\tau) \cap X^\vee_\R$ is a codimension-$1$ rational polyhedral subcone of $C_{\sigma,j}\times C_\tau$ defined by the linear equation $\w_2(n^2_j)(y) = \v_1(\mathbb{g}_1)(x)$.
Moreover, we have that the following are equivalent. 
\begin{enumerate}
\setcounter{enumi}{0}
\item $j \in A_\sigma$, i.e., $C_{\sigma,j}$ is full-dimensional in $C_\sigma$, and
\item $(C_{\sigma,j} \times C_\tau) \cap X^\vee_\R$ is full-dimensional in $(C_\sigma \times C_\tau)\cap X^\vee_\R$. In particular, $(C_{\sigma,j} \times C_\tau) \cap X^\vee_\R$ is codimension-$1$ in $C_\sigma \times C_\tau$.  
\end{enumerate}
\end{lemma} 

\begin{proof} 

Let $j\in [s_2]$, not necessarily in $A_\sigma$.
Note that since $\w_2(n^2_t)(\mathbb{g}_2)=1$ for all $t \in [s_2]$ by hypothesis, it follows that if $y_0 \in C_{\sigma,j}$, i.e. $\w_2(n^2_j)(y_0) \leq \w_2(n^2_t)(y_0)$ for all $t\in [s_2]$, then any $y_0 + \lambda \mathbb{g}_2$ for $\lambda \in \R$ is also in $C_{\sigma,j}$. Then the proof of~Lemma~\ref{lemma: X and Y meet interior of cones} also shows that $X^\vee_\R$ meets the relative interior of $C_{\sigma,j}\times C_\tau$. On $C_{\sigma,j} \times C_\tau$ the equation defining $X^\vee_\R$ is the non-trivial equation $\w_2(n^2_j)(y)=\w_1(x)(\mathbb{g}_1)=\v_1(\mathbb{g}_1)(x)$. In fact, the equation is given by points, which are $\R_{\geq 0}$-linear on maximal cones. It follows that the subset cut out by the equation is a rational polyhedral subcone as claimed. We have also just seen that the hyperplane defining its zero set intersects the relative interior of $C_{\sigma,j} \times C_\tau$. Since the equation $\w_2(n_j^2)(y)=\v_1(\mathbb{g}_1)(x)$ is not identically zero, we have 
\begin{equation*}
\dim(X^\vee_\R \cap (C_{\sigma,j} \times C_\tau))=\dim(C_{\sigma,j})+\dim(C_\tau)-1
.\end{equation*}
A similar argument shows that
\begin{equation*}
\dim(X^\vee_\R \cap (C_{\sigma} \times C_\tau))=\dim(C_{\sigma})+\dim(C_\tau)-1
.\end{equation*}
The lemma then follows by combining the two preceding equations.
%}
\end{proof}

For $j \in [s_2]$ we may also consider $\PP_2(j) \subset \partial \PP_2$, the facet of $\PP_2$ determined by $n^2_j \in \NN_2$, i.e., 
\begin{equation*}\label{eq: def PP2 j} 
\PP_2(j) := \{ y \in (\MM_2)_\R \, \mid \, \w_2(n^2_t)(y) \geq 0, \forall t \in [s_2], \textup{ and } \w_2(n^2_j)(y) = 0\}. 
\end{equation*} 

 In the following proposition, the key players are the intersections $(C_{\sigma,j} \times C_\tau) \cap X^\vee$ appearing in~Lemma~\ref{lemma: A sigma equivalence}.

\begin{proposition}\label{prop-conescover}
Let $C_{\sigma} \times C_{\tau}$ be a maximal cone in $\Sigma(\MM_2 \times \NN_1)$ as above. Let $j \in A_{\sigma}$. Then: 
\begin{enumerate}  
\item the bijection $\phi^\vee: X^\vee_\R \to Y^\vee_\R$ restricts to a bijection between $(C_{\sigma,j} \times C_\tau) \cap X^\vee_\R$ and $(\PP_2(j) \cap C_{\sigma}) \times C_{\tau} \subset Y^\vee_\R \subset (\MM_2 \times \NN_1)_\R$,
\item the restriction of $\phi^\vee$ to $(C_{\sigma,j} \times C_\tau) \cap X^\vee_\R$ is $\R_{\geq 0}$-linear, 

\item $(\PP_2(j) \cap C_\sigma) \times C_\tau$ is a codimension-$1$ subcone of $C_\sigma \times C_\tau$, and,  
\item the collection $\{(\PP_2(j) \cap C_\sigma) \times C_\tau\mid \sigma\in\pi(\NN_2), \tau\in\pi(\MM_1), j \in A_\sigma\}$ cover $Y^\vee_\R$, and, 
\item the collection $\{(C_{\sigma,j} \times C_\tau) \cap X^\vee \, \mid \, \sigma\in\pi(\NN_2), \tau\in\pi(\MM_1), j \in A_\sigma\}$ cover $X^\vee_\R$. 
\end{enumerate}  
\end{proposition}

\begin{proof}
To prove (1), since $\phi^\vee$ is a bijection  it would suffice to show that $\phi^\vee((C_{\sigma,j} \times C_\tau)\cap X^\vee_\R) = (\PP_2(j) \cap C_{\sigma}) \times C_{\tau}$. We begin by showing one inclusion, namely $\phi^\vee((C_{\sigma,j} \times C_\tau)\cap X^\vee_\R) \subseteq (\PP_2(j) \cap C_{\sigma}) \times C_{\tau}$. 
If $(y,x) \in (C_{\sigma,j} \times C_\tau) \cap X^\vee_\R$, then we have $\w_1(x)(\mathbb{g}_1) = \w_2(n^2_j)(y)$. To prove the claim, we must first show that $y-\w_1(x)(\mathbb{g}_1)\mathbb{g}_2 \in C_\sigma$, but this follows from the fact that $\mathbb{g}_2$ is in the lineality space. Next we need to prove that $y - \w_1(x)(\mathbb{g}_1)\mathbb{g}_2$ lies in $\PP_2(j)$.  
Let $t\in[s_2]$ and compute
\begin{align*}
\w_2(n^2_t)(y - \w_1(x)(\mathbb{g}_1)\mathbb{g}_2)
    &= 
    \w_2(n^2_t)(y) - \w_1(x)(\mathbb{g}_1)\, \, \textup{ by Lemma~\ref{lemma: psi of a sum}} \\
& = \w_2(n^2_t)(y) - \w_2(n_j^2)(y)  
\,\,\textup{ since $(y,x) \in X^\vee_\R$, } 
\\&\ge 0\,\,\textup{ since $y\in C_{\sigma,j}$ }
\end{align*} 
and it is equal to $0$ if $t=j$, showing that $y-\w_1(x)(\mathbb{g}_1)\mathbb{g}_2$ lies in $\PP_2(j)$.
For the converse inclusion, let $(y,x) \in (\PP_2(j) \cap C_{\sigma}) \times C_{\tau}$ and consider the inverse $\psi^\vee$ to $\phi^\vee$ defined in~\eqref{eq: psi vee}. 
Since $\Psi_{\PP_2}(y) = 0$, $\psi^\vee(y,x) = (y + \w_1(x)(\mathbb{g}_1)\mathbb{g}_2, x)$ and this lies in $X^\vee_\R$ by Lemma~\ref{lemma: phivee psivee bijections}(4). 
Moreover, $y + \w_1(x)(\mathbb{g}_1)\mathbb{g}_2$ lies in $C_{\sigma,j}$ since by Lemma~\ref{lemma: psi of a sum} we have that for all $t \in [s_2]$
\[\w_2(n^2_j)(y + \w_1(x)(\mathbb{g}_1)\mathbb{g}_2) = \w_2(n^2_j)(y) + \w_1(x)(\mathbb{g}_1) \le \w_2(n^2_t)(y) + \w_1(x)(\mathbb{g}_1)=\w_2(n^2_t)(y+ \w_1(x)(\mathbb{g}_1)\mathbb{g}_2).\]
Thus proves (1). 

To see (2), first recall that addition is well-defined on a maximal-dimensional cone in the mutation fan of a polyptych lattice, so it makes sense to discuss $\R_{\geq 0}$-linearity on $C_{\sigma} \times C_\tau$. In addition, by Lemma~\ref{lemma: A sigma equivalence} we know that $(C_{\sigma,j} \times C_\tau) \cap X^\vee_\R$ is a subcone of $C_{\sigma,j} \times C_\tau$ (which is itself a subcone of $C_{\sigma} \times C_\tau$), hence is closed under addition and $\R_{\geq 0}$-scalar multiplication. The definition of $\phi^\vee$ gives $\phi^\vee(y,x) := (y-\w_1(x)(\mathbb{g}_1)(\mathbb{g}_2), x) = (y-\v_1(\mathbb{g}_1)(x)\mathbb{g}_2, x)$. Since $\v_1(\mathbb{g}_1)$ is a point, it is linear on $C_\tau$, and it follows that $\phi^\vee$ is also linear.

The claim (3) follows from Lemma~\ref{lemma: A sigma equivalence} together with (1) and (2).

To see (4), 
let $(y,x)\in Y^\vee_\R$. Since the maximal-dimensional cones $C_\sigma \times C_\tau$ cover all of $(\MM_2 \times \NN_1)_\R$, there exist $\sigma, \tau$ such that $(y,x)\in C_\sigma\times C_\tau$. By Lemma~\ref{lemma: union of csigmaj}, there exists $j\in A_\sigma$ such that $(y,x)\in C_{\sigma,j}\times C_\tau$. Since $(y,x)\in Y^\vee_\R$ we have that $\Psi_{\mathcal{P}_2}(y)=0$ and it follows that $\w_2(n_j^2)(y)=0$. We conclude that $(y,x)\in (\mathcal{P}_2(j)\cap C_\sigma)
\times C_\tau$.
%}

Finally, claim (5) follows from (1) and (4) using the bijection $\phi^\vee$.
\end{proof}

Proposition~\ref{prop-conescover} gives us a set of cones $(\PP_2(j) \cap C_\sigma) \times C_\tau$ which cover $Y^\vee_\R$. As we will show below, each such cone will provide us with a bijection from $X$ to a lattice. We will then interpret these bijections as coordinate-chart maps of a new polyptych lattice $\mathcal{E}$ (whose set of elements can be identified with $X$), to be defined in Section~\ref{subsec: E and F definition}.

 By swapping indices, we may also give analogous constructions as above, using $X$ and $Y$ (and $X_\R$ and $Y_\R$). These will provide a collection of bijections from $X^\vee$ to a lattice, thus allowing us to define a polyptych lattice $\mathcal{F}$ (with set of elements identified with $X^\vee$) which will be shown to be a strict dual to $\mathcal{E}$.  Let $\alpha \in \pi(\NN_1)$ and $\beta \in \pi(\MM_2)$. Then, as before, these correspond uniquely to maximal cones $C_\alpha$ in $\Sigma(\MM_1)$ and $C_\beta$ in $\Sigma(\NN_2)$ respectively, so $C_{\alpha} \times C_{\beta}$ is a maximal-dimensional cone in $\Sigma(\MM_1 \times \NN_2)$. Let $i \in [s_1]$ and let $\PP_1(i)$ denote the facet of $\PP_1$ determined by $n^1_i \in \NN_1 \cong \Sp(\MM_1)$. As in the case of $X^\vee$ we define 
 \begin{equation}\label{eq: def C sigma tau i}
C_{\alpha, i} := \{ u \in C_\alpha \, \mid \, \w_1(n^1_i)(u) \leq \w_1(n^1_t)(u), \forall t \in [s_1]\} \subset (\MM_1)_\R.
 \end{equation} 
  Let $A_{\alpha} \subset [s_1]$ denote the set of indices in $[s_1]$ for which $C_{\alpha,i}$ is full-dimensional in $C_\alpha$. Essentially the same proofs (obtained by changing indices) yield the analogues of Lemma~\ref{lemma: union of csigmaj}, Lemma~\ref{lemma: A sigma equivalence}, and Proposition~\ref{prop-conescover}. We summarize this below.

 \begin{proposition}\label{proposition: cones cover Y}
 Following the notation above, we have: 
\begin{enumerate} 
\item $(C_{\alpha,i} \times C_\beta) \cap X_\R$ is a codimension-$1$ rational polyhedral subcone of $C_{\alpha,i}\times C_\beta$ defined by the linear equation $\w_1(n^1_i)(u) = \v_2(\mathbb{g}_2)(v)$.
\item $C_\alpha = \bigcup_{i \in A_\alpha} C_{\alpha,i}$, and $A_\alpha$ is minimal with this property, and, 
\item $i \in A_\alpha$ if and only if $(C_{\alpha,i} \times C_\beta) \cap X_\R$ is full-dimensional in $(C_\alpha \times C_\beta) \cap X_\R$, and, 
\item the bijection $\phi: X_\R \to Y_\R$ restricts to a bijection between $(C_{\alpha,i} \times C_\beta) \cap X_\R$ and $(\PP_1(i) \cap C_{\alpha}) \times C_{\beta} \subset Y_\R \subset (\MM_1 \times \NN_2)_\R$,
\item the restriction of $\phi$ to $(C_{\alpha,i} \times C_\beta) \cap X_\R$ is $\R_{\geq 0}$-linear, 

\item $(\PP_1(i) \cap C_\alpha) \times C_\beta$ is a codimension-$1$ subcone of $C_\alpha \times C_\beta$, and,  
\item the collection $\{(\PP_1(i) \cap C_\alpha) \times C_\beta\mid \alpha \in\pi(\NN_1), \beta\in\pi(\MM_2), i \in A_\alpha\}$ cover $Y_\R$, and, 
\item the collection $\{(C_{\alpha,i} \times C_\beta) \cap X_\R \, \mid \, \alpha\in\pi(\NN_1), \beta\in\pi(\MM_2), i \in A_\alpha\}$ cover $X_\R$. 
\end{enumerate} 
\end{proposition}

\subsection{Definition of the polyptych lattices $\mathcal{E}$ and $\mathcal{F}$}\label{subsec: E and F definition} 

We are now in a position to construct two new polyptych lattices over $\Z$, denoted $\mathcal{E}$ and $\mathcal{F}$, which we call the ``Gorenstein cone extension(s)'' of the data above. To define the mutation maps for both $\mathcal{E}$ and $\mathcal{F}$, we follow a strategy which we also used in \cite[Section 8]{EscobarHaradaManon-PL}. Concretely, to define $\mathcal{E}$, we first define a collection of bijections from the set $X$ of \eqref{eq: def X} to multiple lattices. We then interpret the composition of these bijections with their inverses (in various combinations) as the mutation maps defining $\mathcal{E}$. In this setting, the set $X$ serves as a natural model for the set of elements of the polyptych lattice $\mathcal{E}$. We then follow a similar strategy with $X^\vee$ to define $\mathcal{F}$. 

 We start with the definition of the polyptych lattice $\mathcal{E}$ over $\Z$. 
Following the setup and notation of Section~\ref{subsec: prelim notation and def}, consider a maximal-dimensional cone $C_{\sigma} \times C_{\tau}$ in $\Sigma(\MM_2 \times \NN_1)$, let $j \in A_{\sigma}$, and let $X$ be the set defined in~\eqref{eq: def X}. 
Let $\kappa_{\sigma,\tau,j}$ denote the following composition: 
\begin{equation}\label{eq: def kappa sigma tau j}
\begin{tikzcd}
X \arrow[r,""]&\Sp(\MM_2 \times \NN_1) \arrow[r,"L_{\sigma \times \tau}"] & \Hom((C_{\sigma} \times C_{\tau}) \cap (\MM_2 \times \NN_1),\Z) \arrow[r, "\mathrm{res}_j"] &  \Hom(((\PP_2(j) \cap C_\sigma) \times C_\tau) \cap (\MM_2 \times \NN_1),\Z)
\end{tikzcd}
\end{equation}
where the first arrow is the inclusion map to $\MM_1\times \NN_2$ followed by the strict duality map $(u,v) \mapsto (\w_2(v),\v_1(u))$, the map $L_{\sigma \times \tau}$ is obtained by restricting a point to a cone in the fan of $\MM_2 \times \NN_1$. (We had previously notated this $L_{C_\sigma \times C_\tau}$ but for simplicity of notation we choose, here and below, to write $L_{\sigma \times \tau}$ instead.) and $\mathrm{res}_j$ restricts it further to the subcone $(\PP_2(j) \cap C_\sigma) \times C_\tau$.
For the last statement of the proposition below, we remind the reader that by Proposition~\ref{proposition: cones cover Y}(1), $(C_{\alpha,i}\times C_\beta)\cap X$ is a
subcone of $C_{\alpha,i}\times C_\beta$, hence is closed under addition and $\Z_{\geq0}$-scalar
multiplication. Thus, a statement about $\Z_{\geq 0}$-linearity on $(C_{\alpha,i}\times C_\beta)\cap X$ is well-defined.

\begin{proposition}\label{proposition: phi sigma tau j bijection}
Let $j\in A_\sigma$.
The map $\kappa_{\sigma,\tau,j}$ defined in \eqref{eq: def kappa sigma tau j} is a bijection. Moreover, the restriction of $\kappa_{\sigma,\tau,j}$ to $(C_{\alpha,i}\times C_\beta)\cap X$ for any maximal-dimensional $C_{\alpha} \times C_{\beta} \in \Sigma(\MM_1 \times \NN_2)$ and $i \in A_{\alpha}$ is $\Z_{\geq 0}$-linear. Moreover, the natural extension $\kappa_{\sigma,\tau,j,\R}$ restricted to $(C_{\alpha,i} \times C_\beta) \cap X_\R$ is $\R_{\geq 0}$-linear. 
\end{proposition}

\begin{proof} 
We first show that $\kappa_{\sigma,\tau,j}$ is surjective. As a first step, we claim that $\mathrm{res}_j$ is surjective. By Proposition~\ref{prop-conescover} we know that the lattice generated by $((\PP_2(j) \cap C_\sigma) \times C_\tau) \cap (\MM_2 \times \NN_1)$ is obtained from the lattice generated by $(C_\sigma\times C_\tau) \cap (\MM_2 \times \NN_1)$ by setting a (non-trivial) linear equation to $0$.  Thus, any semigroup homomorphism in $\Hom((\PP_2(j) \cap C_\sigma \times C_\tau) \cap (\MM_2 \times \NN_1), \Z)$ can be extended to $\Hom((C_\sigma \times C_\tau) \cap \MM_2 \times \NN_1,\Z)$, i.e., $\mathrm{res}_j$ is surjective.
Moreover, by Lemma~\ref{lemma: res cone surj}, the map $L_{\sigma \times \tau}$ is also surjective.
Strict duality implies that $\MM_1 \times \NN_2 \to \Sp(\MM_2 \times \NN_1)$ is a bijection. From this it follows that for any $f \in \Hom(((\PP_2(j) \cap C_\sigma) \times C_\tau) \cap \MM_2 \times \NN_1,\Z)$, there exists $(u,v) \in \MM_1 \times \NN_2$ mapping to $f$ under the composition of the strict duality map with $\mathrm{res}_j \circ L_{\sigma \times \tau}$. 
To prove the surjectivity of $\kappa_{\sigma,\tau,j}$, it therefore remains to show that we can find $(u_0,v_0) \in X$ which also maps to $f$.

In order to move an arbitrary preimage of $f$ into $X$ we will adjust the $\NN_2$-coordinate. 
We need some notation. Consider $L_{C_\sigma}(\w_2(n^2_j)) \in \Hom(C_\sigma\cap\MM_2,\Z)$. Since $\Hom(C_\sigma\cap\MM_2,\Z)$ is a $\Z$-module we also have that its negative $-L_{C_\sigma}(\w_2(n^2_j) )$ is contained in $\Hom(C_\sigma\cap\MM_2,\Z)$. Now let $\mathsf{p}^2_{j,\sigma}\in \NN_2$ be such that $\w_2(\mathsf{p}^2_{j,\sigma})$ is the lift of $-L_{C_\sigma}(\w_2(n^2_j))$ to $\Sp(\MM_2)$, which exists by Lemma~\ref{lemma: res cone surj}. Note that this lift depends on $\sigma$.  Also recall that $\mathbb{g}_2 \in C_{\sigma,j}$ because $\mathbb{g}_2$ lies in the lineality space of $\MM_2$. Moreover, $\w(n^2_t)$ evaluate to $1$ on $\mathbb{g}_2$ by condition (P-2)(b) of Definition~\ref{definition: Gorenstein PL cone}. From this it follows that 
$$
\w_2(n^2_j)(\mathbb{g}_2) = 1 
\quad 
\textup{ and } 
\quad 
\w_2(\mathsf{p}^2_{j,\sigma})(\mathbb{g}_2) = -1 
$$
for all $\sigma,j$. 

Next, we claim that if $(u,v) \in \MM_1 \times \NN_2$ maps to $f$, as is assumed above, then for $A \in \Z, A \geq 0$, the element $(u,v +_\sigma A n^2_j) \in \MM_1 \times \NN_2$ also maps to $f$. To see this, we must show that they evaluate to the same value against any $(y,x) \in (\PP_2(j) \cap C_\sigma) \times C_\tau$. For such $(y,x)$ we have 
\begin{equation}\label{eq: finding u0 v0}
    \begin{split}
    \langle (u,v+_\sigma A n^2_j),(y,x) \rangle & =     \w_2(v +_\sigma A n^2_j)(y) + \v_1(u)(x) \\
    & = \v_2(y)(v +_\sigma A n^2_j) + \v_1(u)(x) \\
        & = \v_2(y)(v) + \v_2(y)(A n^2_j) + \v_1(u)(x) \, \textup{ since $y \in C_\sigma = \v_2^{-1}(\Sp(\NN_2,\sigma))$}\\
        & = \w_2(v)(y) + A \v_2(y)(n^2_j) + \v_1(u)(x)\, \textup{ since points are $\Z_{\geq 0}$-linear} \\ 
        & = \w_2(v)(y) + A \w_2(n^2_j)(y) + \v_1(u)(x)\\
        & = \w_2(v)(y) + \v_1(u)(x) \, \textup{ since $y \in \PP_2(j)$} \\
        & = f(y,x) \, \textup{ by assumption, since $(u,v)$ maps to $f$}. 
    \end{split}
\end{equation}
Thus $(u,v +_\sigma A n^2_j) \in {\MM}_1 \times {\NN}_2$ also maps to $f$ as claimed.  A similar computation shows that for $B \in \Z, B \geq 0$, the element $(u, v +_\sigma B \mathsf{p}^2_{j,\sigma})$ also maps to $f$, where we use the fact that $y\in C_\sigma$ and $y \in \PP_2(j)$ to deduce that $\w_2(\mathsf{p}_{j,\sigma}^2)(y)=-\w_2(n_{j}^2)(y)=0$.

Now we may find $(u_0,v_0) \in X$ mapping to $f$. Given $(u,v)$ as above, we may compute $\w_2(v)(\mathbb{g}_2) - \Psi_{\PP_1}(u)$. If this quantity is $0$, then by definition of $X$, the element $(u,v)$ is in $X$, so we are done. If it is negative, then consider 
$$
(u_0,v_0) := (u, v +_\sigma (\Psi_{\PP_1}(u) - \w_2(v)(\mathbb{g}_2))n^2_j).
$$
We have already seen above that $(u_0,v_0)$ thus defined maps to $f$. A computation similar to~\eqref{eq: finding u0 v0} shows that 
\[
\w_2(v_0)(\mathbb{g}_2)= \w_2(v)(\mathbb{g}_2) + \Psi_{\PP_1}(u) \w_2(n^2_j)(\mathbb{g}_2) - \w_2(v)(\mathbb{g}_2) \w_2(n^2_j)(\mathbb{g}_2)
.\]
Since $\w_2(n^2_j)(\mathbb{g}_2)=1$ we have that $(u_0,v_0) \in X$. Finally, if $\w_2(v)(\mathbb{g}_2) - \Psi_{\PP_1}(u)$ is positive, then we take $(u_0,v_0) := (u, v+_\sigma (\w_2(v)(\mathbb{g}_2) - \Psi_{\PP_1}(u)) \mathsf{p}^2_{j,\sigma})$ and a similar argument, now using $\w_2(\mathsf{p}^2_{j,\sigma})(\mathbb{g}_2)=-1$, shows that this maps to $f$ and is contained in $X$. This completes the proof that $\kappa_{\sigma,\tau,j}$ is surjective.

We now wish to show that $\kappa_{\sigma,\tau,j}$ is injective. We observe first that $\{0_{\MM_2}\} \times C_\tau$ and $(\PP_2(j) \cap C_\sigma) \times \{0_{\NN_1}\}$ are subsets of $(\PP_2(j) \cap C_\sigma) \times C_\tau$, and any semigroup homomorphism from $((\PP_2(j) \cap C_\sigma)\times C_\tau)\cap (\MM_2\times\NN_1)$ to $\Z$ is determined by its restriction to these subsemigroups. We first examine the $\MM_2$ factor. We claim that if two homomorphisms $f,f': C_\sigma\cap\MM_2 \to \Z$ agree on $\PP_2(j) \cap C_\sigma$ and on $\mathbb{g}_2$, then $f = f'$.
To see this, consider the codimension-1 subcone $H :=\{y\in C_\sigma\mid \w_2(n_j^2)(y)=0\}$ of $C_\sigma$. By Proposition~\ref{prop-conescover}(3), $\mathcal{P}_2(j)\cap C_\sigma$ is a full-dimensional subcone of $H$ and it follows that if two semigroup homomorphisms from $H\cap\MM_2$ to $\Z$ agree on $\mathcal{P}_2(j)\cap C_\sigma\cap \MM_2$ then they agree on $H\cap\MM_2$.
To prove our claim, it suffices to show that any $y\in C_\sigma\cap\MM_2$ can be written as a sum of an element in $H\cap \MM_2$ and an integral multiple of $\mathbb{g}_2$. Let $A\in\Z$ be such that $\w_2(n^2_j)(y)=A$. Consider $y - A \mathbb{g}_2$; this is well-defined and still in $C_\sigma$ since $\mathbb{g}_2$ is in the lineality space of $\MM_2$. Then by Lemma~\ref{lemma: psi of a sum}, $\w_2(n^2_j)(y - A \mathbb{g}_2) = 0$ which means $y - A \mathbb{g}_2 \in H$ and thus $y= (y- A \mathbb{g}_2) + A \mathbb{g}_2$, showing that $y$ can be written as desired. This shows that $f = f'$ if they agree on $\PP_2(j) \cap C_\sigma$ and $\mathbb{g}_2$, as claimed.

To finish the proof that $\kappa_{\sigma,\tau,j}$ is injective, suppose $(u,v),(u',v') \in X$ and suppose they give rise to the same function $f:(({\PP}_2(j) \cap C_\sigma) \times C_\tau)\cap(\MM_2\times\NN_1) \to \Z$. In particular, $(u,v),(u',v')$ agree on $\{0_{\MM_2}\} \times C_\tau$, so from Lemma~\ref{lemma: res cone surj} it follows that $u=u'$. Moreover, since they also agree on $(\PP_2(j) \cap C_\sigma) \times \{0_{\NN_1}\}$, we have $\w_2(v) \mid_{{\PP}_2(j) \cap C_\sigma} = \w_2(v') \mid_{{\PP}_2(j) \cap C_\sigma}$. From the condition that $(u,v),(u',v')$ are elements of $X$ we can additionally conclude 
$$
{\w}_2(v)(\mathbb{g}_2) = \Psi_{\PP_1}(u) = \Psi_{\PP_1}(u') = {\w}_2(v')(\mathbb{g}_2)
$$
where the second equality is because $u=u'$.  Thus $\w_2(v),\w_2(v')$ agree on $\PP_2(j)\cap C_\sigma$ and $\mathbb{g}_2$, and thus $\w_2(v) = \w_2(v')$ by the paragraph above, interpreted as elements of $\Hom(C_\sigma\cap\MM_2,\Z)$. Again from  Lemma~\ref{lemma: res cone surj}, we conclude $\w_2(v)=\w_2(v')$ as elements in $\Sp(\MM_2)$, hence also $\NN_2$. This shows $\kappa_{\sigma,\tau,j}$ is injective.

Lastly, we show that the restriction of $\kappa_{\sigma,\tau,j}$ to $(C_{\alpha,i}\times C_\beta)\cap X$ is
$\Z_{\geq0}$-linear. As noted before the statement of the proposition, addition and scalar multiplication are well-defined on $(C_{\alpha,i}\times C_\beta)\cap X$. 
Let $(u,v),(u',v')\in(C_{\alpha,i}\times C_\beta)\cap X$. We have 
\[
\kappa_{\sigma,\tau,j}\big((u,v)+(u',v')\big)
=\mathrm{res}_j\circ L_{\sigma\times\tau}\big(\w_2(v+v'),\v_1(u+u')\big).
\]
By \cite[Lemma 4.2]{EscobarHaradaManon-PL}, $\v_1$ is linear on $C_\alpha\in\Sigma(\MM_1)$ and $\w_2$ is
linear on $C_\beta\in\Sigma(\NN_2)$, so $\v_1(u+u')=\v_1(u)+\v_1(u')$ and $\w_2(v+v')=\w_2(v)+\w_2(v')$.
Since $L_{\sigma\times\tau}$ and $\mathrm{res}_j$ are restrictions of functions, they preserve addition, so it
follows that $\kappa_{\sigma,\tau,j}((u,v)+(u',v'))=\kappa_{\sigma,\tau,j}(u,v)+\kappa_{\sigma,\tau,j}(u',v')$.
A similar argument using $\Z_{\geq0}$-homogeneity of $\v_1,\w_2$
\cite[Lemma 4.3]{EscobarHaradaManon-PL}, gives
$\kappa_{\sigma,\tau,j}(\lambda(u,v))=\lambda\,\kappa_{\sigma,\tau,j}(u,v)$ for $\lambda\in\Z_{\geq0}$. The same argument above is valid also with $\R$ coefficients. This completes the proof. 
\end{proof}

With these bijections established, we can now use them to define the mutations for a new polyptych lattice. Below, for simplicity we use the notation 
\begin{equation}\label{eq: def charts of calE}
E_{\sigma,\tau,j} := \Hom(((\PP_2(j) \cap C_\sigma) \times C_\tau) \cap (\MM_2 \times \NN_1), \Z)
\end{equation}
for $\sigma \in \pi(\NN_2), \tau \in \pi(\MM_1)$, and $j \in A_\sigma$.

\begin{definition}\label{definition: cone extension calE}
We follow the setup and notation of Section~\ref{subsec: prelim notation and def}. 
Consider the set of lattices $\{E_{\sigma,\tau,j} \, \mid \, \sigma \in \pi(\NN_2), \tau \in \pi(\MM_1), j \in A_\sigma\}$ where the $E_{\sigma,\tau,j}$ are defined in~\eqref{eq: def charts of calE}. For a pair of such indices $(\sigma_1,\tau_1,j_1)$ and $(\sigma_2,\tau_2,j_2)$, consider the map
\begin{equation*}\label{eq: mutations calE} 
\mu_{(\sigma_2,\tau_2,j_2),(\sigma_1,\tau_1,j_1)} := \kappa_{\sigma_2,\tau_2,j_2} \circ \kappa^{-1}_{\sigma_1,\tau_1,j_1}: E_{\sigma_1,\tau_1,j_1} \to E_{\sigma_2,\tau_2,j_2}
\end{equation*} 
which is a bijection and piecewise-linear by Proposition~\ref{proposition: phi sigma tau j bijection}. We define a polyptych lattice $\mathcal{E}$ by this collection $\{E_{\sigma,\tau,j} \, \mid \, \sigma \in \pi(\NN_2), \tau \in \pi(\MM_1), j \in A_\sigma\}$ of lattices with mutation maps $\{\mu_{(\sigma_2,\tau_2,j_2),(\sigma_1,\tau_1,j_1)}\}$. 
\end{definition}

Our first task is to prove that $\mathcal{E}$ is well-defined as a polyptych lattice.

\begin{proposition}\label{proposition: calE well defined}
Following the notation of Definition~\ref{definition: cone extension calE}, the data $\mathcal{E} = (\{E_{\sigma,\tau,j} \, \mid \, \sigma \in \pi(\NN_2), \tau \in \pi(\MM_1), j \in A_\sigma\}, \{\mu_{(\sigma_2,\tau_2,j_2),(\sigma_1,\tau_1,j_1)}: E_{\sigma_1,\tau_1,j_1} \to E_{\sigma_2,\tau_2,j_2}\}_{(\sigma_1,\tau_1,j_1),(\sigma_2,\tau_2,j_2)})$ in Definition~\ref{definition: cone extension calE} defines a finite polyptych lattice over $\Z$ of rank $r_1+r_2-1$.  Moreover, the set $X$ is in bijection with the set of elements of $\mathcal{E}$.
\end{proposition} 

\begin{proof} 

We can readily verify that the conditions of Definition~\ref{definition of polyptich lattice} hold.
First, by Proposition~\ref{prop-conescover}(3), $(\PP_2(j)\cap C_\sigma)\times C_\tau$ is a codimension-$1$ rational polyhedral subcone of the full-dimensional cone $C_\sigma\times C_\tau \subset (\MM_2\times\NN_1)_\R$, so the lattice it generates in $\MM_2\times\NN_1$ has rank $r_1+r_2-1$. Since every $\Z_{\geq 0}$-linear semigroup homomorphism to $\Z$ extends uniquely to that lattice, this implies that each chart $E_{\sigma,\tau,j}$ is isomorphic to $\Z^{r_1+r_2-1}$. The index set is finite because $\MM_1,\NN_2$ are finite by assumption, and $A_\sigma\subseteq[s_2]$.

By Proposition~\ref{proposition: phi sigma tau j bijection} each $\kappa_{\sigma,\tau,j}: X\to E_{\sigma,\tau,j}$ is a bijection.
This immediately implies that the elements of $\mathcal{E}$ are in bijection with the elements of $X$.
We also have that each $\mu_{(\sigma_2,\tau_2,j_2),(\sigma_1,\tau_1,j_1)}=\kappa_{\sigma_2,\tau_2,j_2}\circ\kappa^{-1}_{\sigma_1,\tau_1,j_1}$ is also a bijection. For piecewise linearity, recall from Proposition~\ref{proposition: cones cover Y} that the finitely many cones $(C_{\alpha,i}\times C_\beta)\cap X$ cover $X$, while by the last statement of Proposition~\ref{proposition: phi sigma tau j bijection} every $\kappa_{\sigma,\tau,j}$ is $\Z_{\geq0}$-linear on each such cone. Hence the images $\kappa_{\sigma_1,\tau_1,j_1}\big((C_{\alpha,i}\times C_\beta)\cap X\big)$ form a finite cone cover of $E_{\sigma_1,\tau_1,j_1}$ on each of which $\mu_{(\sigma_2,\tau_2,j_2),(\sigma_1,\tau_1,j_1)}$ is a composition of $\Z_{\geq0}$-linear maps, hence linear.

Finally, conditions (1)--(3) of Definition~\ref{definition of polyptich lattice} are immediate. 
\end{proof}

By switching roles of $X$ and $X^\vee$, we also obtain a polyptych lattice $\mathcal{F}$.  For $\alpha \in \pi(\NN_1), \beta \in \pi(\MM_2),$ and $i \in A_\alpha$ we define the map $\kappa_{\alpha,\beta,i}$ as the composition
\begin{equation}\label{eq: def kappa alpha beta i}
\begin{tikzcd}
X^\vee \arrow[r,""]& \Sp(\MM_1 \times \NN_2) \arrow[r,"L_{\alpha \times \beta}"] & \Hom((C_{\alpha} \times C_{\beta}) \cap (\MM_1 \times \NN_2),\Z) \arrow[r, "\mathrm{res}_i"] &  \Hom(((\PP_1(i) \cap C_\alpha) \times C_\beta) \cap (\MM_1 \times \NN_2),\Z)
\end{tikzcd}
\end{equation}
and also define 
\begin{equation}\label{eq: def charts of calF}
F_{\alpha,\beta,i} := \Hom(((\PP_1(i) \cap C_\alpha) \times C_\beta) \cap (\MM_1 \times \NN_2), \Z).
\end{equation}

The analogous arguments as for Proposition~\ref{proposition: phi sigma tau j bijection}, obtained by switching indices, gives the following. 
\begin{proposition}\label{proposition: kappa sigma beta i bijection}
Let $i\in A_\alpha$.
The map $\kappa_{\alpha,\beta,i}$ defined in \eqref{eq: def kappa alpha beta i} is a bijection. Moreover, the restriction of $\kappa_{\alpha,\beta,i}$ to $(C_{\sigma,j}\times C_\tau)\cap X^\vee$ for any maximal-dimensional $C_{\sigma} \times C_{\tau} \in \Sigma(\MM_2 \times \NN_1)$ and $j \in A_{\sigma}$ is $\Z_{\geq 0}$-linear. Moreover, the natural extension $\kappa_{\alpha,\beta,i,\R}$ restricted to $(C_{\sigma,j} \times C_\tau) \cap X^\vee_\R$ is $\R_{\geq 0}$-linear. 
\end{proposition} 

We define the polyptych lattice $\mathcal{F}$ analogously, as well. 

\begin{definition}\label{definition: cone extension calF}
Following the same setup of Definition~\ref{definition: cone extension calE}, 
consider the set of lattices $\{F_{\alpha,\beta,i} \, \mid \, \alpha \in \pi(\NN_1), \beta \in \pi(\MM_2), i \in A_\alpha\}$. For a pair of such indices $(\alpha_1,\beta_1,i_1)$ and $(\alpha_2,\beta_2,i_2)$, consider the map
\begin{equation*}\label{eq: mutations calF} 
\mu_{(\alpha_2,\beta_2,i_2),(\alpha_1,\beta_1,i_1)} := \kappa_{\alpha_2,\beta_2,i_2} \circ \kappa^{-1}_{\alpha_1,\beta_1,i_1}: F_{\alpha_1,\beta_1,i_1} \to F_{\alpha_2,\beta_2,i_2}
\end{equation*} 
which is a bijection and piecewise linear by Proposition~\ref{proposition: kappa sigma beta i bijection}. We define a polyptych lattice $\mathcal{F}$ by this collection $\{F_{\alpha,\beta,i} \, \mid \, \alpha \in \pi(\NN_1), \beta \in \pi(\MM_2), i \in A_\alpha\}$ of lattices with mutation maps $\{\mu_{(\alpha_2,\beta_2,i_2),(\alpha_1,\beta_1,i_1)}\}$ as defined above. 

\end{definition} 

The analogous proof as for Proposition~\ref{proposition: calE well defined} yields well-definedness.

\begin{proposition}\label{proposition: calF well defined}
Following the notation of Definition~\ref{definition: cone extension calF}, the data $\mathcal{F} = (\{F_{\alpha,\beta,i} \, \mid \, \alpha \in \pi(\NN_1), \beta \in \pi(\MM_2), i \in A_\alpha\}, \{\mu_{(\alpha_2,\beta_2,i_2),(\alpha_1,\beta_1,i_1)}: F_{\alpha_1,\beta_1,i_1} \to F_{\alpha_2,\beta_2,i_2}\}_{(\alpha_1,\beta_1,i_1),(\alpha_2,\beta_2,i_2)})$ in Definition~\ref{definition: cone extension calF} defines a finite polyptych lattice over $\Z$ of rank $r_1+r_2-1$. Moreover, the set $X^\vee$ is in bijection with the set of elements of $\mathcal{F}$. 

\end{proposition}

It will be useful in the later arguments to have chart addition formulas for $\mathcal{E}$ and $\mathcal{F}$. We start with $\mathcal{E}$. Recall that, in our notation below, we have $C_\sigma \subset \MM_2$ and $C_\tau \subset \NN_1$, which means $\sigma \in \pi(\NN_2)$ and $\tau \in \pi(\MM_1)$ are indices of coordinate charts of $\NN_2$ and $\MM_1$ respectively.

\begin{lemma}\label{lemma: chart addition formula for calE} 
Let $\mathcal{E}$ be the polyptych lattice of Definition~\ref{definition: cone extension calE}, with the set of elements identified with the set $X$ of \eqref{eq: def X} and with coordinate charts $\kappa_{\sigma,\tau,j}$ given in \eqref{eq: def kappa sigma tau j}. Suppose $(u,v),(u',v') \in X \subset \MM_1 \times \NN_2$ (viewed as elements of $\mathcal{E}$). Then the chart addition of $(u,v)$ and $(u',v')$ with respect to the chart $\kappa_{\sigma,\tau,j}$ is given by 
\begin{equation}\label{eq: chart addition formula sigma tau j case geq 0}
    (u,v) +_{\sigma,\tau,j} (u',v') =
(u+_\tau u', v+_\sigma v' +_\sigma (\Psi_{\PP_1}(u +_\tau u') - \w_2(v +_\sigma v')(\mathbb{g}_2))n^2_j). 
\end{equation}

\end{lemma}

\begin{proof} 
Recall that $\kappa_{\sigma,\tau,j}$ is the composition of the 
inclusion $X \hookrightarrow \MM_1 \times \NN_2$ with the strict duality map 
$(u,v) \mapsto (\w_2(v),\v_1(u))$, followed by $\mathrm{res}_j \circ L_{\sigma \times \tau}: 
\Sp(\MM_2 \times \NN_1) \to \Hom(((\PP_2(j) \cap C_\sigma) \times C_\tau) \cap (\MM_2 \times \NN_1), \Z)$. 
To compute the chart addition we must take $\kappa_{\sigma,\tau,j}(u,v) + \kappa_{\sigma,\tau,j}(u',v')$ 
in $\Hom(((\PP_2(j) \cap C_\sigma) \times C_\tau) \cap (\MM_2 \times \NN_1), \Z)$ and apply $\kappa_{\sigma,\tau,j}^{-1}$. 
It is useful to compute this inverse in two stages, allowing us to use the argument in the proof of Proposition~\ref{proposition: phi sigma tau j bijection} that $\kappa_{\sigma,\tau,j}$ is surjective.

As a first step, observe that Lemma~\ref{lemma: M alpha addition formula} implies (where we take $\MM = \MM_1 \times \NN_2$ and $\v = (\v_1, \w_2)$) that $(u +_\tau u', v +_\sigma v') \in \MM_1 \times \NN_2$ maps, under $L_{\sigma\times \tau} \circ (\w_2,\v_1)$, to an element of $\Hom((C_\sigma \times C_\tau)  \cap (\MM_2 \times \NN_1),\Z)$ which in turn maps under $\mathrm{res}_j$ to $\kappa_{\sigma,\tau,j}(u,v) + \kappa_{\sigma,\tau,j}(u',v')$.  With this established, we may now use the argument in the proof of Proposition~\ref{proposition: phi sigma tau j bijection} to change $(u +_\tau u', v +_\sigma v')$ into an element of $X$ that also maps to the same homomorphism $\kappa_{\sigma,\tau,j}(u,v) + \kappa_{\sigma,\tau,j}(u',v')$ in $\Hom(((\PP_2(j) \cap C_\sigma) \times C_\tau) \cap \MM_2 \times \NN_1,\Z)$. 

Take $A:=\Psi_{\PP_1}(u +_\tau u') - \w_2(v +_\sigma v')(\mathbb{g}_2)$ and let us show that $A\ge 0$. Recall that $(u,v),(u',v')$ are in $X$, hence $\Psi_{\PP_1}(u) = \w_2(v)(\mathbb{g}_2)$ and $\Psi_{\PP_1}(u') = \w_2(v')(\mathbb{g}_2)$. Moreover, since $\mathbb{g}_2$ is in the lineality space of $\MM_2$, it is contained in every cone of mutation fan of $\MM_2$ and consequently $\w_2(v+_\sigma v')(\mathbb{g}_2) = \w_2(v)(\mathbb{g}_2) + \w_2(v')(\mathbb{g}_2)$. Putting these together, in order to prove the desired inequality it would suffice to prove $\Psi_{\PP_1}(u+_\tau u') - \Psi_{\PP_1}(u) - \Psi_{\PP_1}(u') \geq 0$. Now recall that $\Psi_{\PP_1} := \min_{j \in [s_1]} \{\w_1(n^1_j)\}$ where each $\w_1(n^1_j)$ is a point, which means it satisfies $\w_1(n^1_j)(u+_\tau u') \geq \w_1(n^1_j)(u) + \w_1(n^1_j)(u')$. We can compute 
\begin{equation}
    \begin{split} 
\Psi_{\PP_1}(u +_\tau u') & = \min_{j \in [s_1]} \{\w_1(n^1_j)(u +_\tau u')\} \\
& \geq \min_{j \in [s_1]} \{\w_1(n^1_j)(u) + \w_1(n^1_j)(u')\} \\
& \geq \min_{j \in [s_1]} \{\w_1(n^1_j)(u)\} + \min_{j \in [s_1]} \{\w_1(n^1_j)(u')\} \\
& = \Psi_{\PP_1}(u) + \Psi_{\PP_1}(u') 
    \end{split} 
\end{equation}
which is what we wanted to show.

Using the argument in Proposition~\ref{proposition: phi sigma tau j bijection} we have that $(u +_\tau u', v +_\sigma v' +_\sigma A\, n^2_j)$ lies in $X$ and also maps to $\kappa_{\sigma,\tau,j}(u,v) + \kappa_{\sigma,\tau,j}(u',v')$.
Since $\kappa_{\sigma,\tau,j}$ is injective by Proposition~\ref{proposition: phi sigma tau j bijection}, we conclude that $(u +_\tau u', v +_\sigma v' +_\sigma A\, n^2_j)$ is the chart sum $(u,v) +_{\sigma,\tau,j} (u',v')$, which is the formula 
in~\eqref{eq: chart addition formula sigma tau j case geq 0}. 
\end{proof} 
 
We will also need the chart addition formula in $Y$. The proof is a computation using the bijection $\phi: X \to Y$ and is omitted. 

\begin{lemma}\label{lemma: chart addition in Y}
Let the notation be as in Lemma~\ref{lemma: chart addition formula for calE}. Using the  bijection $\phi: X \to Y$, let the set of elements of $\mathcal{E}$ be represented by $Y$. Then the chart addition of $(u,v)$ and $(u',v')$ in $Y$, with respect to $\kappa_{\sigma,\tau,j}$, is given by 
%obtained by $\phi^{-1}(\phi(u,v)+_{\sigma,\tau,j} \phi(u',v'))$. For $(u,v),(u',v') \in Y$ we have 
\begin{equation}\label{eq: chart addition in Y}
(u,v)+_{\sigma,\tau,j} (u',v') = (u+_\tau u' - \Psi_{\PP_1}(u+_\tau u')\cdot \mathbb{g}_1, v+_\sigma v' +_\sigma \Psi_{\PP_1}(u+_\tau u')\cdot n^2_j) \in Y.
\end{equation} 
\end{lemma}

Similar formulas hold for $\mathcal{F}$, as we state below. The proof is similar to that of Lemma~\ref{lemma: chart addition formula for calE} and Lemma~\ref{lemma: chart addition in Y}. 

\begin{lemma}\label{lemma: chart addition formula for calF} 
Let $\mathcal{F}$ be the polyptych lattice of Definition~\ref{definition: cone extension calF}, with the set of elements identified with the set $X^\vee$ of \eqref{eq: def X vee} and with coordinate charts $\kappa_{\alpha,\beta,i}$ given in \eqref{eq: def kappa alpha beta i}. Suppose $(y,x),(y',x') \in X^\vee \subset \MM_2 \times \NN_1$ (viewed as elements of $\mathcal{F}$). Then the chart addition of $(y,x)$ and $(y',x')$ in $X^\vee$ with respect to the chart $\kappa_{\alpha,\beta,i}$ is given by
\begin{equation}\label{eq: chart addition formula alpha beta i case geq 0}
    (y,x) +_{\alpha,\beta,i} (y',x') =
(y+_\alpha y', x+_\beta x' +_\beta (\Psi_{\PP_2}(y +_\alpha y') - \w_1(x +_\beta x')(\mathbb{g}_1))n^1_i). 
\end{equation}

\end{lemma} 

\begin{lemma}\label{lemma: chart addition in Y vee}
Let the notation be as in Lemma~\ref{lemma: chart addition formula for calF}. Using the  bijection $\phi^\vee: X^\vee \to Y^\vee$, let the set of elements of $\mathcal{F}$ be represented by $Y^\vee$. Then the chart addition of $(y,x)$ and $(y',x')$ in $Y^\vee$, with respect to $\kappa_{\alpha,\beta,i}$, is given by 
\begin{equation}\label{eq: chart addition in Y vee}
(y,x)+_{\alpha,\beta,i} (y',x') = (y+_\alpha y' - \Psi_{\PP_2}(y+_\alpha y')\cdot \mathbb{g}_2, x+_\beta x' +_\beta \Psi_{\PP_2}(y+_\alpha y')\cdot n^1_i) \in Y^\vee.
\end{equation} 
\end{lemma}

With the above in place, we can compute the mutations fans of $\mathcal{E}$ and $\mathcal{F}$.

\begin{proposition}\label{proposition: mutation fans of E and F}
The maximal cones of the mutation fan $\Sigma(\mathcal{E})$ of $\mathcal{E}$ are precisely the cones $(C_{\alpha,i} \times C_\beta) \cap X_\R$ (where we use the identification of $X$ with the set of elements of $\mathcal{E}$). Similarly, the maximal cones of the mutation fan $\Sigma(\mathcal{F})$ of $\mathcal{F}$ are the cones $(C_{\sigma,j} \times C_\tau) \cap X^\vee_\R$. 
\end{proposition}

\begin{proof} 
 We will show that the cones $(C_{\alpha,i} \times C_\beta) \cap X_\R$ are the maximal cones of the mutation fan $\Sigma(\mathcal{E})$. From Proposition~\ref{proposition: phi sigma tau j bijection} we know that all mutation maps are linear on these cones.  We need to show that these cones are precisely the common domains of linearity for all mutations. To see this, fix two distinct cones $(C_{\alpha,i} \times C_\beta) \cap X_\R$ and $(C_{\alpha',i'} \times C_{\beta'}) \cap X_\R$ in $\mathcal{E}_\R$. Now we need to find two distinct charts $\kappa_{\sigma,\tau,j}$ and $\kappa_{\sigma',\tau',j'}$ with the property that the two corresponding chart additions are distinct on the two fixed cones, i.e., we must find $(u,v) \in (C_{\alpha,i} \times C_\beta) \cap X_\R$ and $(u',v') \in (C_{\alpha',i'} \times C_{\beta'}) \cap X_\R$ such that 
\begin{equation}\label{eq: non equal}
(u,v) +_{\sigma,\tau,j} (u',v') \neq (u,v) +_{\sigma',\tau',j'} (u',v').
\end{equation}
By Lemma~\ref{lemma: chart addition formula for calE} we have formulas for the chart additions. 
By assumption, $(\alpha,\beta,i ) \neq (\alpha',\beta',i')$. To show the claim, we take cases. 

First suppose that $\alpha \neq \alpha'$. Then since the $\tau$'s index the coordinate charts of $\MM_1$, and the $C_\alpha$'s are the maximal cones of the mutation fan of $\MM_1$, the assumption $\alpha \neq \alpha'$ implies that there exist $\tau, \tau'$ such that $\mu_{\tau' \tau} \vert_{C_\alpha}$ and $\mu_{\tau' \tau} \vert_{C_{\alpha'}}$, which are linear on $\pi_\tau(C_\alpha)$ and $\pi_\tau(C_{\alpha'})$ respectively, have distinct canonical linear extensions (temporarily denoted $T$ and $T'$ respectively, for simplicity) to $(\MM_1)_\tau$. Since $C_{\alpha,i}$ and $C_{\alpha',i'}$ are full-dimensional in $C_{\alpha}$ and $C_{\alpha'}$ respectively, if $T \neq T'$ then there exists $\bar{u} \in \pi_\tau(C_{\alpha,j})$ and $\bar{u}' \in \pi_\tau(C_{\alpha',j'})$ such that $\mu_{\tau'\tau}(\bar{u}+\bar{u}') \neq \mu_{\tau'\tau}(\bar{u})+\mu_{\tau'\tau}(\bar{u}')$. For $u := \pi_\tau^{-1}(\bar{u})$ and $u' = \pi_\tau^{-1}(\bar{u}')$, this means that $u+_\tau u' \neq u+_{\tau'} u'$. Choosing $v,v' \in (\NN_2)_\R$ such that $\w_2(v)(\mathbb{g}_2)=\Psi_{\PP_1}(u)$ and $\w_2(v')(\mathbb{g}_2) = \Psi_{\PP_1}(u')$ (this is possible since $\mathbb{g}_2$ is in the lineality space, and $\v_2(\mathbb{g}_2)$ is non-zero by assumption (P-2) of Gorenstein cones), we have that $(u,v) \in X_\R$ and $(u',v') \in X_\R$. Choose $\sigma$ so that $v \in C_\sigma$ and similarly choose $\sigma'$ so that $v' \in C_{\sigma'}$. Let $j \in A_\sigma, j' \in A_{\sigma'}$. From the first coordinate of the formula for chart additions in Lemma~\ref{lemma: chart addition formula for calE} we can see that for $(\sigma,\tau,j), (\sigma',\tau',j')$ as chosen, we have that 
$(u,v)+_{\sigma,\tau,j} (u',v') \neq (u,v) +_{\sigma',\tau',j'} (u',v')$ as claimed. Thus we have shown that we can make choices for \eqref{eq: non equal} to hold, in this case. 

Next we consider the case when $\alpha=\alpha'$.  
Suppose that $\beta \neq \beta'$. First note that since any $(u,v) \in C_{\alpha,j} \times C_\beta, (u',v') \in C_{\alpha,j'} \times C_\beta$ satisfies $u,u' \in C_\alpha$, and $C_\alpha$ is a cone of $\Sigma(\MM_1)$, all chart additions in $\MM_1$ will agree on $u$ and $u'$; thus we may assume without loss of generality that we choose $\tau=\tau'$. Since we assume $\beta \neq \beta'$, we know by similar reasoning as in the above paragraph there exist $\sigma,\sigma' \in \pi(\NN_2)$ and $v \in C_\beta, v' \in C_{\beta'}$ such that $v+_\sigma v' \neq v+_{\sigma'} v'$. Also by similar reasoning as in the above paragraph, we may choose $u \in C_{\alpha,i}, u' \in C_{\alpha'=\alpha,i'}$ such that $(u,v), (u',v') \in X_\R$ and also that $u+_\tau u'$ satisfies $\Psi_{\PP_1}(u+_\tau u') - \w_2(v+_\sigma v')(\mathbb{g}_2) = 0$. Now from Lemma~\ref{lemma: chart addition formula for calE} it follows that, for these choices, we have $(u,v)+_{\sigma,\tau,j} (u',v') \neq (u,v)+_{\sigma',\tau',j'} (u',v')$ as desired.

Finally, suppose $\alpha=\alpha'$ and $\beta=\beta'$ and $i \neq i'$. 
Next note that since $n^1_i \neq n^1_{i'}$ we know that $\w_1(n^1_i)\neq \w_1(n^1_{i'})$ and thus the equation $\w_1(n^1_i)=\w_1(n^1_{i'})$ defines a codimension-$1$ subset of any chart. In particular there exists an open subset $U$ in the interior of $C_{\alpha,i}$ on which $\w_1(n^1_i)=0, \w_1(n^1_t)>0$ for all $t\neq i$. An analogous open subset $U'$ can be found in $\mathrm{int}(C_{\alpha,i'})$. It follows that for $u \in U, u' \in U'$ we have $\Psi_{\PP_1}(u +_\tau u') > \Psi_{\PP_1}(u) + \Psi_{\PP_1}(u')$ for any $\tau$. By translating $u,u'$ by $\mathbb{g}_1$, which preserves the property of lying in $U$ and $U'$ respectively, we may arrange that $\Psi_{\PP_1}(u)=\Psi_{\PP_1}(u')=0$. 
Since $\PP_2$ is strongly chart-convex we know from Lemma~\ref{lemma: strongly chart convex} that $s_2 \geq 2$, so there exist $j, j' \in [s_2]$ with $j \neq j'$. By the assumption (P-1) in Definition~\ref{definition: Gorenstein PL cone} on $\PP_2$ we also know $n^2_j \neq n^2_{j'}$. Since the $C_\sigma$ cover $(\MM_2)_\R$, there exist $\sigma, \sigma'$ such that $j \in A_\sigma, j' \in A_{\sigma'}$. Choose $v=v'=0$. Since $\Psi_{\PP_1}(u)=\Psi_{\PP_1}(u')=0$, we know $(u,0),(u',0) \in X$. By hypothesis on $X$ we know $A := \Psi_{\PP_1}(u+_\tau u') - \Psi_{\PP_1}(u)-\Psi_{\PP_1}(u') = \Psi_{\PP_1}(u+_\tau u') - \w_2(v+_\sigma v')(\mathbb{g}_2) = \Psi_{\PP_1}(u+_\tau u')$, and we saw above that  $A > 0$. 
Then by Lemma~\ref{lemma: chart addition formula for calE} we have 
$$
(u,v)+_{\sigma,\tau,j}(u',v') = (u+_\tau u', A n^2_j)
$$
whereas 
$$
(u,v)+_{\sigma',\tau,j'} (u',v') = (u+_\tau u', A n^2_{j'}).
$$
Since $n^2_j \neq n^2_{j'}$, there exist choices for which~\eqref{eq: non equal} holds, as desired. 

The argument is analogous for the $(C_{\sigma,j} \times C_\tau) \cap X^\vee_\R$ and $\Sigma(\mathcal{F})$. This concludes the proof. 
\end{proof}

\subsection{The PLs $\mathcal{E}$ and $\mathcal{F}$ form a strict dual pair}\label{subsec: E and F strict dual}

Next, we wish to show that the two polyptych lattices $\mathcal{E}$ and $\mathcal{F}$, constructed above, form a strict dual pair over $\Z$. To see this, we use the identification of the set of elements of $\mathcal{E}$ with $X$ which comes from its construction. In the case of the set of elements of $\mathcal{F}$, the construction above naturally identifies it with $X^\vee$, but we additionally use the bijection $\phi^\vee: X^\vee \to Y^\vee$ defined in \eqref{eq: phi vee} to identify it with $Y^\vee$. Under these identifications, in order to show that $\mathcal{E}, \mathcal{F}$ form a strict dual pair, we must first prove that there is an identification of the elements of $\mathcal{E}$ with the space of points of $\mathcal{F}$. This is the content of Theorem~\ref{theorem: E and F strict duals} below.

Suppose $(y,x) \in Y^\vee$. Then we can construct a function on $X$ associated to $(y,x)$ by defining, for $(u,v) \in X$, 
\begin{equation}\label{eq: def f y x} 
  f_{(y,x)}: X \to \Z, \quad   f_{(y,x)}(u,v) := \langle (u,v),(y,x)\rangle 
\end{equation}
where $\langle \cdot, \cdot \rangle$ is the pairing in \eqref{eq: explicit pairing}. With this notation, we can state the following.

\begin{proposition}\label{proposition: dual map F to Sp of calE}
Following the notation above, the association $\Theta: (y,x) \in Y^\vee \mapsto f_{(y,x)}$ is a bijection from $Y^\vee$ to $\Sp(\mathcal{E})$. Moreover, if $(y,x) \in (\PP_2(j_0) \cap C_{\sigma_0}) \times C_{\tau_0}$, then $\Theta(y,x)$ lies in $\Sp(\mathcal{E},(\sigma_0,\tau_0,j_0))$, and $\Theta$ extends naturally to $\Theta_\R: Y^\vee_\R \to \Sp_\R(\mathcal{E})$. 
\end{proposition} 

Put another way, the proposition above states that $Y^\vee$ (and hence $\mathcal{F}$) can be identified with the space of points $\Sp(\mathcal{E})$ of $\mathcal{E}$. This is the critical technical step in showing that $\mathcal{E}$ and $\mathcal{F}$ form a dual pair.

\begin{proof}[Proof of Proposition~\ref{proposition: dual map F to Sp of calE}]
A priori, the association $(y,x) \mapsto f_{(y,x)}$ maps to the space of functions $\mathrm{Fun}(X,\Z)$. Hence we must 
first show that $f_{(y,x)}$ is an element of $\Sp(\mathcal{E})$. To see this, we need that for any $(u,v), (u',v') \in X$, we must have
$$
f_{(y,x)}(u,v) + f_{(y,x)}(u',v') = \min_{\sigma,\tau,j} \{f_{(y,x)}((u,v) +_{\sigma,\tau,j} (u',v'))\}
$$
where the RHS takes the minimum over all coordinate charts $\kappa_{\sigma,\tau,j}$ of $X$ (or equivalently $\mathcal{E}$) and $+_{\sigma,\tau,j}$ denotes the chart addition in the $(\sigma,\tau,j)$-th chart. Equivalently, since we are using the pairing given in \eqref{eq: explicit pairing}, we need to prove 
\begin{equation}\label{eq: Yvee is a point}
\langle (u,v),(y,x)\rangle + \langle (u',v'),(y,x)\rangle = 
\min_{\sigma,\tau,j} \{ \langle (u,v) +_{\sigma,\tau,j} (u',v'), (y,x) \rangle \}
\end{equation}
for any $(u,v), (u',v') \in X$. It would suffice now to prove that for any triple $(\sigma,\tau,j)$, the quantity $\langle (u,v) +_{\sigma,\tau,j} (u',v'), (y,x) \rangle$ is greater than or equal to the LHS of \eqref{eq: Yvee is a point}, and, that there exists a triple $(\sigma_0, \tau_0, j_0)$ where equality is achieved. To implement this strategy, we first compute 
the LHS of \eqref{eq: Yvee is a point} as 
\begin{equation}\label{eq: LHS}
\begin{split}
\langle (u,v),(y,x)\rangle + \langle (u',v'),(y,x)\rangle &= \v_1(u)(x)+\w_2(v)(y)+\v_1(u')(x)+\w_2(v')(y)\\
&= \w_1(x)(u) + \w_1(x)(u') + \v_2(y)(v) + \v_2(y)(v').
\end{split}
\end{equation}
Now suppose $(\sigma,\tau,j)$ is any choice of coordinate chart for $\mathcal{E}$. For this triple, we may compute the quantity $\langle (u,v) +_{\sigma,\tau,j} (u',v'), (y,x) \rangle$ appearing in the RHS of \eqref{eq: Yvee is a point} as follows. 
Recall that we know from Proposition~\ref{prop-conescover} that there exist $\sigma_0, \tau_0, j_0$ such that $(y,x) \in (\PP_2(j_0) \cap C_{\sigma_0}) \times C_{\tau_0}$. Also recall from the assumption of strict duality that $C_{\sigma_0} = (\v_2)^{-1}(\Sp_\R(\NN_2,\sigma_0))$ and $C_{\tau_0} = (\w_1)^{-1}(\Sp_\R(\MM_1,\tau_0))$. This means that the minimum in the definition of points for $\w_1(x)$ and $\v_2(y)$ is achieved in the coordinate charts $\tau_0$ and $\sigma_0$ respectively. 
Finally, since $y \in \PP_2(j_0)$, we know $\w_2(n^2_j)(y) \geq 0$ for any $j$ and is equal to $0$ when $j=j_0$.
Together with these facts and Lemma~\ref{lemma: chart addition formula for calE} we obtain: 
\begin{equation}\label{eq: RHS}
    \begin{split} 
\langle (u,v)+_{\sigma,\tau,j} (u',v'), (y,x) \rangle &= \langle (u +_\tau u', v+_\sigma v' +_\sigma (\Psi_{\PP_1}(u+_\tau u') - \w_2(v+_\sigma v')(\mathbb{g}_2))(n^2_j)), (y,x)\rangle \\
& = \v_1(u+_\tau u')(x) + \w_2(v+_\sigma v' +_\sigma (\Psi_{\PP_1}(u+_\tau u') - \w_2(v+_\sigma v')(\mathbb{g}_2))(n^2_j)))(y) \\
& = \w_1(x)(u+_\tau u') + \v_2(y)(v+_\sigma v' +_\sigma (\Psi_{\PP_1}(u+_\tau u') - \w_2(v+_\sigma v')(\mathbb{g}_2))(n^2_j))) \\
& \geq \w_1(x)(u) + \w_1(x)(u') + \v_2(y)(v) + \v_2(y)(v') \\
& \quad + (\Psi_{\PP_1}(u+_\tau u') - \w_2(v+_\sigma v')(\mathbb{g}_2))\v_2(y)(n^2_j) \\
& = \w_1(x)(u) + \w_1(x)(u') + \v_2(y)(v) + \v_2(y)(v') \\
& \quad + (\Psi_{\PP_1}(u+_\tau u') - \w_2(v+_\sigma v')(\mathbb{g}_2))\w_2(n^2_j)(y) \\
& \geq \w_1(x)(u) + \w_1(x)(u') + \v_2(y)(v) + \v_2(y)(v')
    \end{split} 
\end{equation}
for any $(\sigma,\tau,j)$, where the first inequality holds because $\w_1(x)$ and $\v_2(y)$ are points. Moreover, when $\sigma=\sigma_0,\tau=\tau_0,j=j_0$, we see that 
$$
\langle (u,v)+_{\sigma_0,\tau_0,j_0} (u',v'), (y,x) \rangle = \w_1(x)(u) + \w_1(x)(u') + \v_2(y)(v) + \v_2(y)(v'). 
$$
This proves that the minimum in the RHS is achieved by the LHS and thus \eqref{eq: Yvee is a point} holds. We conclude that the association $(y,x) \mapsto f_{(y,x)}$ is a well-defined map $\Theta: Y^\vee \to \Sp(\mathcal{E})$. Moreover, the argument just given additionally shows that if $(y,x) \in (\PP_2(j_0) \cap C_{\sigma_0}) \times C_{\tau_0}$, then $\Theta(y,x)$ is linear on the $(\sigma_0,\tau_0,j_0)$-th chart. Homogeneity with respect to multiplication by real constants is immediate from the fact that the pairing comes from strict dual pairings on the original data. This shows that $\Theta(y,x)$ lies in $\Sp(\mathcal{E},(\sigma_0,\tau_0,j_0))$. This proves the second claim in the proposition. 

We need to show that the association $\Theta: (y,x) \mapsto f_{(y,x)}$ from $Y^\vee \to \Sp(\mathcal{E})$ is a bijection. The proof of injectivity is essentially the same as the proof of injectivity of $\kappa_{\sigma,\tau,j}$ in Proposition~\ref{proposition: phi sigma tau j bijection}, where we swap the roles of $X$ and $X^\vee$, and also use Lemma~\ref{lem-pairing} to see that the pairing between $X^\vee$ and $Y$ may be identified with the pairing between $Y^\vee$ and $X$. To see that $\Theta$ is surjective, first observe that $\Theta$ can be uniquely extended to a continuous map $\Theta_\R: Y^\vee_\R \to \Sp_\R(\mathcal{E})$ by hypotheses on strict duals.   Both the target and domain of $\Theta_\R$ are homeomorphic to a real vector space of the same dimension $r_1+r_2-1$. The arguments above extend naturally to $\R$ coefficients, proving that $\Theta_\R$ is injective. Moreover, $\Theta_\R$ respects dilations by construction. Thus $\Theta_\R$ induces an injective map of spheres of the same dimension. If $\Theta_\R$ were not onto, then by the Borsuk-Ulam theorem, $\Theta_\R$ cannot be injective, but this is a contradiction. Thus $\Theta_\R$ must be surjective. From this, we wish to conclude that $\Theta$ is also surjective. To see this, suppose that $p$ is an (integral) point in $\Sp(\mathcal{E})$. We wish to show there is an element $(y,x)$ of $Y^\vee$ with $\Theta(y,x) = p$. First observe that since a point in $\Sp(\mathcal{E})$ is also a point in $\Sp_\R(\mathcal{E})$, we know there exist $(y',x') \in Y^\vee_\R$ with $\Theta_\R(y',x')=p$. Next, by Proposition~\ref{prop-conescover}(4) we know that $(y',x')$ is contained in some $(\PP_2(j_0) \cap C_{\sigma_0})\times C_{\tau_0}$. We just saw above that this means $(y',x')$ is linear on the $(\sigma_0,\tau_0,j_0)$-th coordinate-chart of $\mathcal{E}$. Since $(y',x')$ maps to $p$ by assumption, we know that $(y',x')$ takes integral values on all (lattice) elements in the $(\sigma_0,\tau_0,j_0)$-th coordinate chart $E_{\sigma_0,\tau_0,j_0}$ of $\mathcal{E}$. Moreover, $(y',x')$ is linear on this chart. The only way this can happen is if $(y',x')$ is itself an integral element on this chart, i.e. lies in $\Hom(E_{\sigma_0,\tau_0,j_0},\Z)$. But this then implies $(y',x') \in \Sp(\mathcal{E})$, since being an integral element in one coordinate chart implies that it takes integral values on all coordinate charts. Thus $\Theta$ is surjective, as desired. 
\end{proof}

An entirely analogous narrative holds by switching indices. We will be brief. Let $(u,v) \in Y$. Define 
\begin{equation}\label{eq: def f vee u v}
f^\vee_{(u,v)}: X^\vee \to \Z, \quad f^\vee_{(u,v)}(y,x) := \langle (u,v),(y,x)\rangle. 
\end{equation} 

\begin{proposition}\label{proposition: dual map E to Sp calF}
Following the notation above, the association $(u,v) \in Y \mapsto f^\vee_{(u,v)}$ is a bijection from $Y$ to $\Sp(\mathcal{F})$. 
\end{proposition}

Thus $\mathcal{E}$ can be identified with the space of points of $\mathcal{F}$. 
We can now state the main result of this section. Define the map $\v: \mathcal{E} \to \Sp(\mathcal{F})$ (respectively $\w: \mathcal{F} \to \Sp(\mathcal{E})$) by the composition of the bijection $\phi: X \to Y$ of Lemma~\ref{lem-2sets}(2) with the bijection $Y \rightarrow \Sp(\mathcal{F})$ of Proposition~\ref{proposition: dual map E to Sp calF} (respectively the composition of $\phi^\vee: X^\vee \to Y^\vee$ in Lemma~\ref{lemma: phivee psivee bijections}(2) with the map $Y^{\vee} \rightarrow \Sp(\mathcal{E})$ in Proposition~\ref{proposition: dual map F to Sp of calE}).

\begin{theorem}\label{theorem: E and F strict duals}
Following the notation above, the data $(\mathcal{E},\mathcal{F}, \v,\w)$ is a strict dual pair of polyptych lattices over $\Z$.  
\end{theorem}

\begin{proof} 
We check the axioms in Definition~\ref{def_dual}. The maps $\v$ and $\w$ extend naturally to $\R$ coefficients since they are induced by strict dual pairings (defined over $\R$) on the original data $(\MM_i, \NN_i)$ for $i=1,2$. We know they are bijections by Proposition~\ref{proposition: dual map F to Sp of calE} and Proposition~\ref{proposition: dual map E to Sp calF}, and $\v(u,v)(y,x)=\w(y,x)(u,v)$ follows from \eqref{eq: def f y x} and~\eqref{eq: def f vee u v} and Lemma~\ref{lem-pairing}.

It remains to show that the maximal cones of $\Sigma(\mathcal{F})$ are precisely the inverse images $\w^{-1}(\Sp_\R(\mathcal{E},(\sigma,\tau,j))$ as $(\sigma,\tau,j)$ varies, and similarly with the roles of $\mathcal{E}$ and $\mathcal{F}$ reversed. 
We will give the argument for $\w: \mathcal{F}_\R \to \Sp_\R(\mathcal{E})$. We saw in Proposition~\ref{proposition: dual map F to Sp of calE} that the association $\w=\Theta \circ \phi^\vee$ takes $(C_{\sigma,j} \times C_\tau) \cap X^\vee_\R \to \Sp_\R(\mathcal{E},(\sigma,\tau,j))$. 
By Proposition~\ref{prop-conescover} it now suffices to prove that if $(y,x) \in Y^\vee_\R$ maps to $\Sp_\R(\mathcal{E},(\sigma,\tau,j))$ under $\Theta$, then $(y,x) \in (\PP_2(j) \cap C_\sigma) \times C_\tau$. We will prove the contrapositive, i.e., if $(y,x) \not \in (\PP_2(j) \cap C_\sigma) \times C_\tau$, then $(y,x)$ is \emph{not} linear on the $(\sigma,\tau,j)$-th chart. In other words, we must show that there exist choices such that 
\begin{equation}\label{eq: strict duality first part}
\langle (y,x), \kappa_{\sigma,\tau,j}^{-1}(\kappa_{\sigma,\tau,j}(u,v) + \kappa_{\sigma,\tau,j}(u',v')) \rangle = \langle (y,x),(u,v)\rangle + \langle (y,x),(u',v')\rangle.
\end{equation}
does \emph{not} hold. 
We will take cases. Namely, the condition $(y,x) \not \in (\PP_2(j) \cap C_\sigma) \times C_\tau$ occurs if any of the following occur: (1) $x \not \in C_\tau$, (2) $y \not \in C_\sigma$, or (3) $y \in C_\sigma$ but $y \not \in \PP_2(j)$. If we can show in each of these cases that $(y,x)$ is not linear on the $(\sigma,\tau,j)$-chart, then we are done. 

Consider the case when $x \not \in C_\tau$. We wish to find $(u,v),(u',v') \in X_\R$ such that the LHS and RHS of~\eqref{eq: strict duality first part} are not equal. First, since $x \not \in C_\tau$, we know $\w_1(x)$ is not linear with respect to $+_\tau$, and hence there exist $u,u' \in (\MM_1)_\R$ such that $\w_1(x)(u+_\tau u') \neq \w_1(x)(u) + \w_1(x)(u')$. In fact, by properties of points, we must have $\w_1(x)(u+_\tau u') > \w_1(x)(u) + \w_1(x)(u')$. We may add multiples of $\mathbb{g}_1$ to $u$ and $u'$ without affecting this inequality; by~Lemma~\ref{lemma: psi of a sum} this implies that we may assume without loss of generality that $\Psi_{\PP_1}(u)\geq 0$ and $\Psi_{\PP_1}(u') \geq 0$. Given such $u$ and $u'$, we now choose $v$ and $v'$ in a specific manner. First, observe that since $y \in Y^\vee$, we know that $\Psi_{\PP_2}(y) := \min_{k \in [s_2]} \{\w_2(n^2_k)(y)\} =0$. Let $j_0$ be the index achieving the minimum, so $\w_2(n^2_{j_0})(y)=0$. 
Now choose $v = \Psi_{\PP_1}(u)\cdot n^2_{j_0}$ and $v' = \Psi_{\PP_1}(u') \cdot n^2_{j_0}$. By construction, since $\w_2(\Psi_{\PP_1}(u)\cdot n^2_{j_0})(\mathbb{g}_2)=\Psi_{\PP_1}(u) \w_2(n^2_{j_0})(\mathbb{g}_2) = \Psi_{\PP_1}(u)$, and similarly for $v'$, we see that $(u,v) \in X_\R$ and $(u',v') \in X_\R$. Notice that 
$$
\v_2(y)(v) = \v_2(y)(\Psi_{\PP_1}(u) \cdot n^2_{j_0}) = \w_2(\Psi_{\PP_1}(u) \cdot n^2_{j_0})(y) = \Psi_{\PP_1}(u) \w_2(n^2_{j_0})(y) = 0. 
$$
 A similar argument shows $\v_2(y)(v')=0$. From this and~\eqref{eq: explicit pairing} it follows that the RHS of~\eqref{eq: strict duality first part} for this choice of $(u,v),(u',v')$ is equal to 
$$
\w_1(x)(u) + \w_1(x)(u').
$$
It remains to compute the LHS of~\eqref{eq: strict duality first part} and show that it is not equal to $\w_1(x)(u) + \w_1(x)(u')$. 
By~Lemma~\ref{lemma: chart addition formula for calE}, the LHS of~\eqref{eq: strict duality first part} can be computed as 
\begin{equation*}
    \begin{split} 
\textup{LHS} &= \w_1(x)(u+_\tau u') + \v_2(y)(\Psi_{\PP_1}(u) \cdot n^2_{j_0} +_\sigma \Psi_{\PP_1}(u')\cdot n^2_{j_0} +_\sigma (\Psi_{\PP_1}(u+_\tau u') - \Psi_{\PP_1}(u) - \Psi_{\PP_1}(u'))n^2_{j_0}) \\
& = \w_1(x)(u+_\tau u') + \v_2(y)(\Psi_{\PP_1}(u +_\tau u')n^2_{j_0})\\
& = \w_1(x)(u+_\tau u') + \Psi_{\PP_1}(u+_\tau u')\v_2(y)(n^2_{j_0}) \\
& = \w_1(x)(u+_\tau u') + \Psi_{\PP_1}(u+_\tau u')\w_2(n^2_{j_0})(y) \\
& = \w_1(x)(u+_\tau u') \quad \textup{ since $\w_2(n^2_{j_0})(y)=0$} \\ 
& > \w_1(x)(u) +\w_1(x)(u') \quad \textup{ by assumption on $u,u'$}
    \end{split} 
\end{equation*}
so we conclude that the LHS is strictly greater than the RHS and hence $(y,x)$ is not linear with respect to the $(\sigma,\tau,j)$-chart addition, as desired. 

Next we consider the case $y \not \in C_{\sigma}$. Then we know there exist $v,v' \in (\NN_2)_\R$ with $\v_2(y)(v+_\sigma v') > \v_2(y)(v) + \v_2(y)(v')$. Choose $u = \w_2(v)(\mathbb{g}_2) \cdot \mathbb{g}_1$ and $u' = \w_2(v')(\mathbb{g}_2) \cdot\mathbb{g}_1$. Then since $\w_1(n^1_k)(\mathbb{g}_1)=1$ for all $i$ we know $\Psi_{\PP_1}(u) = \w_2(v)(\mathbb{g}_2)$ and $\Psi_{\PP_1}(u')=\w_2(v')(\mathbb{g}_2)$ so $(u,v),(u',v') \in X_\R$. Note that in this setting, $\Psi_{\PP_1}(u+_\tau u') - \w_2(v+_\sigma v')(\mathbb{g}_2) = \w_2(v)(\mathbb{g}_2)+\w_2(v')(\mathbb{g}_2) - \w_2(v+_\sigma v')(\mathbb{g}_2) = 0$ because $\Psi_{\PP_1}(\mathbb{g}_1)=1$. Thus the LHS of~\eqref{eq: strict duality first part} may be computed to be 
\begin{equation*}
    \begin{split} 
\mathrm{LHS} & = \langle (y,x), (u+_\tau u', v+_\sigma v')\rangle \\
& = \w_1(x)(u+_\tau u') +\v_2(y)(v+_\sigma v') \\
& = \w_1(x)(u)+\w_1(x)(u') + \v_2(y)(v+_\sigma v')
    \end{split} 
\end{equation*}
where the last equality is because $u, u'$ are both in the lineality space. 
Comparing with~\eqref{eq: LHS} we see that the assumption $\v_2(y)(v+_\sigma v') > \v_2(y)(v) + \v_2(y)(v')$ implies that the LHS is strictly greater than the RHS. Hence $(y,x)$ is not linear with respect to the $(\sigma,\tau,j)$-chart addition. 

Finally consider the case when $(y,x) \in Y^\vee_\R, x \in C_\tau, y\in C_\sigma$, but $y \not \in \PP_2(j)$. Since by assumption $\Psi_{\PP_2}(y)=0$, this implies $\w_2(n^2_j)(y) > 0$.  By the assumption of strong chart-convexity and because $r_1 \geq 2$, the pullback of $\Psi_{\PP_1}$ to the $\tau$-th chart of $(\MM_1)_\R$ is a non-linear convex function. In particular there exist $u,u' \in (\MM_1)_\R$ with the property $\Psi_{\PP_1}(u+_\tau u') > \Psi_{\PP_1}(u) + \Psi_{\PP_1}(u')$. Now following arguments similar to the proof of Proposition~\ref{proposition: mutation fans of E and F}, there exist $(u,v),(u',v') \in X_\R$ with 
$A := \Psi_{\PP_1}(u+_\tau u') - \w_2(v +_\sigma v')(\mathbb{g}_2) > 0$. 
Now the LHS of~\eqref{eq: strict duality first part} becomes 
\begin{equation*}
    \begin{split} 
\textup{ LHS} & = \w_1(x)(u+_\tau u') +\v_2(y)(v+_\sigma v' +_\sigma A n^2_j) \\
& = \w_1(x)(u)+\w_1(x)(u')+\v_2(y)(v)+\v_2(y)(v') + \v_2(y)(An^2_j) \\
& = \w_1(x)(u)+\w_1(x)(u')+\v_2(y)(v)+\v_2(y)(v') + A \cdot  \v_2(y)(n^2_j) \\ 
& > \w_1(x)(u)+\w_1(x)(u')+\v_2(y)(v)+\v_2(y)(v') 
    \end{split} 
\end{equation*}
where the last inequality is because $\v_2(y)(n^2_j)=\w_2(n^2_j)(y) > 0$ and also $A > 0$. Thus the LHS is strictly greater than the RHS. This completes the argument for all cases.

The case of $\v: \mathcal{E}_\R \to \Sp_\R(\mathcal{F})$ follows from the same argument, with indices swapped.  This completes the proof. 
\end{proof}

\section{Detropicalizations of the Gorenstein PL cone extensions $\mathcal{E}$ and $\mathcal{F}$}\label{sec: E and F detrop}

The goal of this section is to prove that the PL lattices $\mathcal{E}$ and $\mathcal{F}$ constructed in Section~\ref{sec: cone extensions} are, under some hypotheses on the original data, also detropicalizable. In particular, in this section we continue to assume all hypotheses on the objects in question, as described at the beginning of Section~\ref{subsec: prelim notation and def}. 

We begin by recalling the key definitions, which we state in a more general setting. Below, the notation $S_\MM$ refers to the canonical semialgebra of $\MM$. For details and further properties, we refer to \cite{EscobarHaradaManon-PL}. 

\begin{definition}\cite[Definition 6.3]{EscobarHaradaManon-PL}\label{definition: recall detrop}
Let $\MM$ be a finite polyptych lattice of rank $r$ over $\Z$ with associated canonical idempotent semialgebra $S_\MM$. Let $\Aa_\MM$ be a Noetherian $\mathbb{K}$-algebra which is an integral domain and $\fv: \Aa_\MM \to S_\MM$ a valuation with values in $S_\MM$. We say that the pair $(\Aa_\MM, \fv)$ is a \textbf{detropicalization of $\MM$} if every element of $\MM$ is in the image of $\fv$, and the Krull dimension of $\Aa_\MM$ equals the rank $r$ of $\MM$. 
\end{definition} 

\begin{definition}\cite[Definition 6.5]{EscobarHaradaManon-PL}\label{definition: recall CAB}
Let $\MM$ be a finite polyptych lattice over $F$. Let $(\Aa_\MM,\fv: \Aa_\MM \to S_\MM)$ be a detropicalization of $\MM$. We say that a $\K$-vector space basis $\B$ of $\Aa_\MM$ is a \textbf{convex adapted basis} for $\fv: \Aa_\MM \to S_\MM$ if 
\begin{enumerate} 
\item[(1)] $\fv(\sum \lambda_i \bb_i) = \bigoplus_i \fv(\bb_i)$, for any finite collection $\lambda_i \in \K^*$ and $\bb_i \in \B$, and
\item[(2)] $\fv(\bb) \in \MM \subset S_\MM$ for all $\bb \in \B$.
\end{enumerate}
\end{definition}
 
 We record here the detailed discussion of a detropicalization for $\mathcal{E}$. The analogous statements for $\mathcal{F}$ are obtained by swapping indices. Assume that $\MM_1$ and $\NN_2$ have detropicalizations $(\Aa_{\MM_1}, \nu_{\MM_1})$ and $(\Aa_{\NN_2}, \nu_{\NN_2})$ respectively, and are also equipped with convex adapted bases $\mathbb{B}_{\MM_1}$ and $\mathbb{B}_{\NN_2}$. In particular, we assume that $\Aa_{\MM_1}, \Aa_{\NN_2}$ are $\K$-algebras and integral domains. Under these hypotheses we show below (Theorem~\ref{theorem: main}) that a detropicalization of $\mathcal{E}$ can be constructed concretely from $\Aa_{\MM_1}$ and $\Aa_{\NN_2}$ (cf. \eqref{eq: E detrop formula}), and we can additionally give a concrete convex adapted basis for $\mathcal{E}$ in terms of those for $\MM_1$ and $\NN_2$.

We first recall that the detropicalization of a direct product can be given by the tensor product of the detropicalizations \cite[Proposition 6.19]{EscobarHaradaManon-PL}. Thus we take as the detropicalization of $\MM_1 \times \NN_2$ the following:  
$$
\Aa_{\MM_1 \times \NN_2} := \Aa_{\MM_1} \otimes \Aa_{\NN_2}. 
$$
Let $\mathbb{B}_{\MM_1} = \{\hat{\mathbb{b}}_{u}^{\MM_1}\}_{u \in \MM_1}$ denote a convex adapted basis for $\Aa_{\MM_1}$, where (following \cite{EscobarHaradaManon-PL}) the basis is naturally indexed, via $\nu_{\MM_1}$, by the elements $u$ of $\MM_1$. Similarly let $\mathbb{B}_{\NN_2}$ denote the convex adapted basis for $\Aa_{\NN_2}$, indexed by elements of $\NN_2$. We will assume that we have chosen $\hat{\mathbb{b}}^{\MM_1}_{0_{\MM_1}} = 1$ and $\hat{\mathbb{b}}^{\NN_2}_{0_{\NN_2}} = 1$, where the $1$ denotes the identity element in the respective algebras.  Again applying \cite[Proposition 6.19]{EscobarHaradaManon-PL}, we know that $\{\hat{\mathbb{b}}_{u}^{\MM_1} \otimes \hat{\mathbb{b}}_{v}^{\NN_2}\}$ form a convex adapted basis for $\Aa_{\MM_1 \times \NN_2}$. For simplicity, for $u \in \MM_1$ and $v \in \NN_2$ we introduce the notation 
\begin{equation}\label{eq: def hat b u v}
\hat{\mathbb{b}}_{(u,v)} := \hat{\mathbb{b}}_{u}^{\MM_1} \otimes \hat{\mathbb{b}}_{v}^{\NN_2} \in \Aa_{\MM_1 \times \NN_2} := \Aa_{\MM_1} \otimes \Aa_{\NN_2}.
\end{equation}
We also define a subalgebra $\Aa_{{\PP}_1 \times \NN_2}$ of $\Aa_{\MM_1 \times \NN_2}$ as follows: 
\begin{equation}
\Aa_{{\PP}_1 \times \NN_2} := \left\{ \sum_{(u,v) \in (\PP_1 \cap \MM_1) \times \NN_2} c_{(u,v)} \hat{\mathbb{b}}_{(u,v)}  \, \mid \, u \in {\PP}_1 \cap \MM_1, c_{(u,v)} \in \K, \textup{ only finitely many $c_{(u,v)} \neq 0$}  \right\}.
\end{equation}
We denote by $\hat{\mathbb{B}}_{\PP_1 \times \NN_2}$ the basis $\{\hat{\mathbb{b}}_{(u,v)}\}_{u \in \PP_1 \cap \MM_1, v \in \NN_2}$ of $\Aa_{\PP_1 \times \NN_2}$. 
 Next we note that $\PP_1$ (and hence $\PP_1\times (\NN_2)_\R$) is point-convex by \cite[Lemma 5.3]{EscobarHaradaManon-PL} (the proof in \cite{EscobarHaradaManon-PL} that PL polytopes are point-convex generalizes to the case of PL polyhedra, i.e., when the intersection of PL half-spaces is not necessarily compact). Applying \cite[Lemma 7.3]{EscobarHaradaManon-PL}, we obtain that $\Aa_{{\PP}_1 \times \NN_2}$ is a subalgebra. 
The basis element corresponding to 
\begin{equation}\label{eq: def tau}
 (\mathbb{g}_1, 0_{\NN_2}) \in \MM_1 \times \NN_2
\end{equation}
will play a crucial role in our arguments below, so we introduce the notation 
\begin{equation}\label{eq: definition tau cal E}
\zeta_{\mathcal{E}} :=\hat{\mathbb{b}}_{(\mathbb{g}_1, 0_{\NN_2})} = \hat{\mathbb{b}}^{\MM_1}_{\mathbb{g}_1} \otimes 1 \in \Aa_{\MM_1}\otimes \Aa_{\NN_2}. 
\end{equation}
Since $\mathbb{g}_1$ lies in $\mathrm{int}(\PP_1)$ by assumption, it follows that $\zeta_{\mathcal{E}}$ lies in the subalgebra $\Aa_{{\PP}_1 \times \NN_2}$. 

We will need the following technical lemma.

\begin{lemma}\label{lemma: tau E is a unit}
We may assume without loss of generality that the inverse in $\Aa_{\MM_1 \times \NN_2}$ of the element $\zeta_{\mathcal{E}}$ is $\hat{\mathbb{b}}_{(-\mathbb{g}_1,0_{\NN_2})}$. 
\end{lemma}

Given the above lemma, we sometimes denote $\hat{\mathbb{b}}_{(-\mathbb{g}_1,0_{\NN_2})}$ as $\zeta_{\mathcal{E}}^{-1}$. Note that $\zeta_{\mathcal{E}}$ is a unit in $\Aa_{\MM_1 \times \NN_2}$ but not in $\Aa_{\PP_1 \times \NN_2}$, since $-\mathbb{g}_1 \not \in \PP_1$.

\begin{proof} 
Recall from \cite[Lemma 7.3]{EscobarHaradaManon-PL} that in general, for $\MM$ a polyptych lattice and $\Aa_\MM$ a detropicalization with convex adapted basis $\{\mathbb{b}_u\}$, with $\nu(\mathbb{b}_u)=u$ for $u \in \MM$, we have 
\begin{equation}\label{eq: mult basis elements} 
\supp(\mathbb{b}_{u} \cdot \mathbb{b}_{u'}) = \ptconv(\{u+_\tau u')\}_{\tau \in \pi(\MM)} \cap \MM.
\end{equation} 
Next recall that since $\mathbb{g}_1$ is contained in the lineality space of $\MM_1$, its negative $-\mathbb{g}_1$ is also in $\MM_1$ and we may consider the basis element $\hat{\mathbb{b}}_{(-\mathbb{g}_1,0_{\NN_2})}$. By \eqref{eq: mult basis elements} we obtain that $\zeta_{\mathcal{E}} \cdot \hat{\mathbb{b}}_{(-\mathbb{g}_1,0_{\NN_2})} = \hat{\mathbb{b}}_{(\mathbb{g}_1, 0_{\NN_2})} \cdot \hat{\mathbb{b}}_{(-\mathbb{g}_1,0_{\NN_2})}$ has support contained in $\ptconv(\{(0_{\MM_1}, 0_{\NN_2})\})$ since $-\mathbb{g}_1+\mathbb{g}_1 = 0_{\MM_1}$ (addition is well-defined in the lineality space). Since $\Aa_{\PP_1 \times \NN_2}$ is a domain we conclude $\zeta_{\mathcal{E}} \cdot \hat{\mathbb{b}}_{(-\mathbb{g}_1,0_{\NN_2})}$ is a non-zero constant multiple of $\hat{\mathbb{b}}_{(0_{\MM_1},0_{\NN_2})} = 1$. By taking appropriate scalar multiples, we may therefore assume without loss of generality that $\zeta_{\mathcal{E}}$ has inverse $\hat{\mathbb{b}}_{(-\mathbb{g}_1,0_{\NN_2})}$. 
\end{proof}

\begin{lemma}\label{lemma: tau calE arithmetic} 
For any $k \in \Z$ and any $(u,v) \in \MM_1 \times \NN_2$, the product $\zeta_{\mathcal{E}}^k \cdot \hat{\mathbb{b}}_{(u,v)}$ is a non-zero scalar multiple of $\hat{\mathbb{b}}_{(u+k\mathbb{g}_1,v)}$. 
\end{lemma} 

\begin{proof} 
Note that $\zeta_{\mathcal{E}}^k$ is a well-defined element of $\Aa_{\MM_1 \times \NN_2}$ for any $k \in \Z$, by Lemma~\ref{lemma: tau E is a unit}. Now the claim follows from 
a similar argument as in the proof of Lemma~\ref{lemma: tau E is a unit}.
\end{proof} 

Before proceeding, we highlight that the next step uses the data of the PL cone $\PP_2$, \emph{not} $\PP_1$, which is what we had repeatedly referenced in the paragraphs above. Recall that $\PP_2$ is defined by PL half-spaces given by $\{n^2_j\}_{j=1}^{s_2}$. For $j \in [s_2]$, consider the convex basis element $\hat{\mathbb{b}}_{(0_{\MM_1},n^2_j)} \in \Aa_{{\PP}_1 \times \NN_2}$. 
We define 
\begin{equation}\label{eq: def h} 
    \hh := \zeta_{\mathcal{E}} - \sum_{j \in [s_2]} \hat{\mathbb{b}}_{(0_{\MM_1},n^2_j)} \in \Aa_{{\PP}_1\times \NN_2}.
\end{equation}
Note that $\mathbb{h}$ is non-zero and not a unit in $\Aa_{\PP_1 \times \NN_2}$ by \cite[Proposition 7.6]{CookEscobarHaradaManon2024}. 

We now define the quotient algebra
\begin{equation}\label{eq: E detrop formula} 
\Aa_{\mathcal{E}} := \Aa_{{\PP}_1 \times \NN_2} \big/ \langle \hh \rangle.
\end{equation}

The main result of this section is Theorem~\ref{theorem: main}, the 
proof of which requires multiple steps. For the argument, it is helpful to first construct a basis for $\mathcal{A}_{\mathcal{E}}$ with the properties we want. With this in mind, we begin with Lemma~\ref{lemma: B cal E is a basis} below, in which we find a $\K$-vector space basis for $\Aa_{\mathcal{E}}$.  Since $\Aa_{\mathcal{E}}$ is a quotient of $\Aa_{{\PP}_1 \times \NN_2}$, we have the natural projection $\Aa_{{\PP}_1 \times \NN_2} \to \Aa_\mathcal{E}$. For $(u,v) \in (\PP_1 \cap \MM_1) \times \NN_2$ we introduce the notation 
$$
\mathbb{b}_{(u,v)} := \textup{ the image of } \hat{\mathbb{b}}_{(u,v)} \, \textup{ in the quotient ring } \, \Aa_{\mathcal{E}}. 
$$
Now we define 
\begin{equation}\label{eq: CAB of detrop of E}
\mathbb{B}_{\mathcal{E}} := \{ \mathbb{b}_{(u,v)} \, \mid \, (u,v) \in Y \subset \MM_1 \times \NN_2 \} \subset \Aa_{\mathcal{E}}
\end{equation}
where $Y$ is defined in \eqref{eq: def Y}. 
We have the following. 

\begin{lemma}\label{lemma: B cal E is a basis}
The set $\mathbb{B}_{\mathcal{E}}$ of \eqref{eq: CAB of detrop of E} is a $\K$-vector space basis of $\Aa_{\mathcal{E}}$. 
\end{lemma}

\begin{proof} 
We begin by showing that $\mathbb{B}_{\mathcal{E}}$ spans $\Aa_{\mathcal{E}}$. Since $\{\hat{\mathbb{b}}_{(u,v)} \, \mid \, (u,v) \in ({\PP}_1 \cap \MM_1)  \times \NN_2\}$ is a basis for $\Aa_{{\PP}_1 \times \NN_2}$, it suffices to show that for any $\hat{\mathbb{b}}_{(u,v)}$, we may find a linear combination in the quotient so that 
\begin{equation*}\label{eq: B_E spans}
\mathbb{b}_{(u,v)} = \sum c_i \mathbb{b}_{(u_i,v_i)}
\end{equation*}
where the sum is finite, $c_i \in \K$, and $(u_i,v_i) \in Y$. We take cases. If $(u,v) \in Y$ then $\mathbb{b}_{(u,v)}$ is itself an element of $\mathbb{B}_\mathcal{E}$ so the claim follows. The non-trivial case is if $(u,v) \in (\mathrm{int}({\PP}_1) \cap \MM_1) \times \NN_2$, i.e. $u$ is in the interior of the cone ${\PP}_1$, which means  $\Psi_{\PP_1}(u)>0$. 
We may rewrite $\hat{\mathbb{b}}_{(u,v)}$ using the distinguished element $\zeta_{\mathcal{E}}$ of \eqref{eq: definition tau cal E} as follows: 
\begin{equation}\label{eq: rewrite with tau mathcal E} 
    \begin{split}
\hat{\mathbb{b}}_{(u,v)} & = 
\zeta_{\mathcal{E}}^{\Psi_{\PP_1}(u)} \cdot 
\zeta_{\mathcal{E}}^{- \Psi_{\PP_1}(u)}
\cdot  \hat{\mathbb{b}}_{(u,v)} \\
& = 
c \cdot \zeta_{\mathcal{E}}^{\Psi_{\PP_1}(u)} \cdot 
\hat{\mathbb{b}}_{(u - \Psi_{\PP_1}(u) \mathbb{g}_1,v)}\\
\end{split} 
\end{equation} 
where we have used Lemma~\ref{lemma: tau calE arithmetic}, and $c \neq 0, c \in \K$. From Lemma~\ref{lem-2sets}(1) we know $(u - \Psi_{\PP_1}(u) \mathbb{g}_1, v) \in Y$. Next, from \eqref{eq: def h} and \eqref{eq: E detrop formula} we see that in the quotient $\Aa_{\mathcal{E}}$, \eqref{eq: rewrite with tau mathcal E} becomes 
\begin{equation}\label{eq: b_uv in quotient}
\mathbb{b}_{(u,v)} = c \cdot 
\left(\sum_{j \in [s_2]} \mathbb{b}_{(0_{\MM_1},n^2_j)}\right)^{\Psi_{\PP_1}(u)} \cdot 
\mathbb{b}_{(u - \Psi_{\PP_1}(u) \cdot \mathbb{g}_1, v)}.
\end{equation}
We now claim that the RHS of \eqref{eq: b_uv in quotient} lies in the span of $\mathbb{B}_{\mathcal{E}}$. To see this, we work upstairs in $\Aa_{\PP_1 \times \NN_2}$. Note that $\hat{\mathbb{b}}_{(0_{\MM_1},v)} := \hat{\mathbb{b}}^{\MM_1}_{0_{\MM_1}} \otimes \hat{\mathbb{b}}^{\NN_2}_{v} = 1\otimes \hat{\mathbb{b}}^{\NN_2}_v \in \Aa_{\MM_1} \otimes \Aa_{\NN_2}$ has a $1$ in the $\Aa_{\MM_1}$ factor. The product of such elements must also have a $1$ in the left factor, which implies that the same is true of all the summands appearing in $\left(\sum_{j \in [s_2]} \hat{\mathbb{b}}_{(0,n^2_j)}\right)^{\Psi_{\PP_1}(u)}$. Now multiplying each summand by $\hat{\mathbb{b}}_{(u - \Psi_{\PP_1}(u) \cdot \mathbb{g}_1, v)}$, by~\eqref{eq: mult basis elements} and the fact that addition with $0_{\MM_1}$ is well-defined in $\MM_1$, it follows that the left factor of any basis element appearing in such a product is of the form $\hat{\mathbb{b}}_{(u-\Psi_{\PP_1}(u)\mathbb{g}_1, v')}$ for some $v' \in \NN_2$. Since $(u-\Psi_{\PP_1}(u)\mathbb{g}_1,v')$ is in $Y$, this implies $\mathbb{b}_{(u-\Psi_{\PP_1}(u)\mathbb{g}_1,v')}$ is in $\mathbb{B}_{\mathcal{E}}$. This shows that the RHS of \eqref{eq: b_uv in quotient} lies in the span of $\mathbb{B}_{\mathcal{E}}$, as claimed. 

We now claim that $\mathbb{B}_{\mathcal{E}}$ is linearly independent. Suppose $\sum_{(u,v)\in Y} c_{u,v}\mathbb{b}_{(u,v)} = 0$ in $\Aa_{\mathcal{E}}$ where the LHS is a finite sum. This means that there exists $f \in \Aa_{{\PP}_1 \times \NN_2}$ such that $\sum_{(u,v) \in Y} c_{u,v} \hat{\mathbb{b}}_{(u,v)} = f \cdot \mathbb{h}$ in $\Aa_{\PP_1\times \NN_2}$. It would suffice to show that if this equation holds, then $f = 0$, since this would imply that all coefficients $c_{u,v}=0$. To see this, it suffices to show that if $f$ is not zero, then the unique expression of $f \cdot \mathbb{h}$ in the basis $\hat{\mathbb{B}}_{\PP_1 \times \NN_2}$ must contain non-zero expressions $c_{u,v}\mathbb{b}_{(u,v)}$ for $(u,v) \not \in Y$.

So suppose $f = \sum c'_{u,v} \hat{\mathbb{b}}_{(u,v)}$, where $(u,v)$ ranges over $(\PP_1 \cap \MM_1) \times \NN_2$, and at most finite many $c'_{u,v}$ are non-zero. We assume $f \neq 0$ so at least one $c'_{u,v} \neq 0$. Since $u \in \PP_1$, we also know $\Psi_{\PP_1}(u) \geq 0$. Now decompose $f$ as $f = f_n + \cdots + f_0$ where $f_k = \sum_{\Psi_{\PP_1}(u) = k} c'_{u,v} \hat{\mathbb{b}}_{(u,v)}$, i.e. we decompose $f$ according to the value of the $u$ component under $\Psi_{\PP_1}$. For each $\hat{\mathbb{b}}_{(u,v)}$ appearing in $f_k$, we know from ~Lemma~\ref{lemma: tau calE arithmetic} that $\zeta_{\mathcal{E}}^k \hat{\mathbb{b}}_{(u-k\cdot \mathbb{g}_1, v)}$ is a non-zero scalar multiple of $\hat{\mathbb{b}}_{(u,v)}$ and, by Lemma~\ref{lem-2sets}(1), that $(u-k\cdot \mathbb{g}_1,v)$ is in $Y$. Factoring out the common $\zeta_{\mathcal{E}}^k$ from all such terms, it follows that we can express $f$ as
$$
f = \zeta_{\mathcal{E}}^n g_n + \zeta_{\mathcal{E}}^{n-1} g_{n-1} + \cdots 
+ g_0
$$
for some $n \in \Z_{\geq 0}$,
where each $g_k$ is a linear combination of basis elements $\hat{\mathbb{b}}_{(u',v')}$ where $(u',v') \in Y$.  Without loss of generality we may assume $g_n \neq 0$. Recall that $\hh = \zeta_{\mathcal{E}} - h_0$ for $h_0 := \sum_{j \in [s_2]} \hat{\mathbb{b}}_{(0_{\MM_1},n^2_j))}$. Therefore 
$$
f \hh = (\zeta_{\mathcal{E}}^{n+1} g_n + \zeta_{\mathcal{E}}^n g_{n-1} + \cdots + \zeta_{\mathcal{E}} g_0)  - 
(\zeta_{\mathcal{E}}^n g_n h_0 + \zeta_{\mathcal{E}}^{n-1} g_{n-1} h_0 + \cdots + g_0 h_0). 
$$
It follows from Lemma~\ref{lemma: tau calE arithmetic} and Lemma~\ref{lemma: psi of a sum} that if $\hat{\mathbb{b}}_{(u'',v'')}$ appears in the expression $\zeta_{\mathcal{E}}^{n+1} g_n$ then $\Psi_{\PP_1}(u'') = n+1>0$.  
On the other hand, for the basis elements $\hat{\mathbb{b}}_{(u'',v'')}$ appearing in the other summands $\zeta_{\mathcal{E}}^k g_k$ for $k<n$, we have $\Psi_{\PP_1}(u'') < n+1$. Moreover, since $h_0$ only contains basis elements with a $0_{\MM_1}$ in the first coordinate, the $\hat{\mathbb{b}}_{(u'',v'' )}$ appearing in $\zeta_{\mathcal{E}}^n g_n h_0 + \zeta_{\mathcal{E}}^{n-1} g_{n-1} h_0 + \cdots + g_0 h_0$ all satisfy $\Psi_{\PP_1}(u'') \leq n$ by assumption on $f$ and the choice of $n$ above. In particular, there is no basis element appearing in any of the terms $\zeta_{\mathcal{E}}^n g_{n-1} + \cdots + \zeta_{\mathcal{E}} g_0$ or in $\zeta_{\mathcal{E}}^n g_n h_0 + \zeta_{\mathcal{E}}^{n-1} g_{n-1} h_0 + \cdots + g_0 h_0$ that cancels any basis element in $\zeta_{\mathcal{E}}^{n+1} g_n$. Finally note that $\zeta_{\mathcal{E}}^{n+1}g_n \neq 0$ and there is no internal cancellation among basis elements within $\zeta_{\mathcal{E}}^{n+1}g_n$.  This implies that $f \mathbb{h}$ contains at least one term $\hat{\mathbb{b}}_{(u'',v'')}$ not contained in $Y$, as was to be shown.

\end{proof}

Now we begin the work required to build an appropriate valuation. (Although in Section~\ref{sec: cone extensions} we viewed the $X$ (respectively $X^\vee$) as the set of elements of $\mathcal{E}$ (respectively $\mathcal{F}$) and $Y$ (respectively $Y^\vee$) as the space of points $\Sp(\mathcal{F})$ (respectively $\Sp(\mathcal{E})$), in this section it is more convenient to do the opposite. Lemma~\ref{lem-pairing} shows that the pairing between $X$ and $Y^\vee$ is equivalent to that between $Y$ and $X^\vee$, so either choice is valid.) We first recall some general facts. Suppose given a polyptych lattice $\MM$, a point $p \in \Sp(\MM)$, a valuation $\nu: \Aa \to S_\MM$ (in the sense of \cite[Definition 6.1]{EscobarHaradaManon-PL}) and $\mathbb{B}$ a convex adapted basis for $\nu$. Then $p$ induces an idempotent semialgebra homomorphism $S_\MM \to (\Z \cup \{\infty\},\min,+)$ (also denoted $p$), the composition $p \circ \nu: \Aa \to \Z \cup \{\infty\}$ is a valuation (see e.g. \cite[Lemma 6.9]{EscobarHaradaManon-PL}), and $\mathbb{B}$ is a convex adapted basis for $p \circ \nu$. Also recall that if $(\Aa_{\MM_1},\nu_{\MM_1}), (\Aa_{\NN_2},\nu_{\NN_2})$ are detropicalizations of $\MM_1, \NN_2$ respectively, then there is a natural choice of valuation $\hat{\nu}: \Aa_{\MM_1 \times \NN_2} \to S_{\MM_1 \times \NN_2}$ given in \cite[Proposition 6.19]{EscobarHaradaManon-PL} making $(\Aa_{\MM_1 \times \NN_2},\hat{\nu})$ a detropicalization of $\MM_1 \times \NN_2$. With this in mind, let $(y,x) \in X^{\vee} \subset \MM_2 \times \NN_1$. By strict duality, $(y,x)$ defines a point in $\Sp(\MM_1 \times \NN_2)$ and hence defines (by composing with $\nu$) a $\Z$-valued valuation; we denote this composition as $\hat{\nu}_{(y,x)}: \Aa_{\MM_1 \times \NN_2} \to \Z \cup \{\infty\}$. The restriction of a valuation to a subalgebra is still a valuation, so we may consider $\hat{\nu}_{(y,x)}$ to be defined on $\Aa_{{\PP}_1 \times \NN_2}$ (which we also denote by $\hat{\nu}_{(y,x)}$). Now for any $a \in \Z$ we may define 
\begin{equation*}\label{eq: def upstairs filtration} 
\hat{F}^{(y,x)}_{\geq a} := \{ \hat{f} \in \Aa_{\PP_1 \times \NN_2} \, \mid \, \hat{\nu}_{(y,x)}(\hat{f}) \geq a \} \subseteq \Aa_{{\PP}_1 \times \NN_2}. 
\end{equation*} 
Varying over $a \in \Z$, this yields a decreasing filtration of $\Aa_{{\PP}_1 \times \NN_2}$ indexed by $\Z$. Let 
\begin{equation}\label{eq: definition quotient filtration F y x geq a}
F^{(y,x)}_{\geq a} :=  \textup{ the image of } \, \,  \hat{F}^{(y,x)}_{\geq a} \, \textup{under the quotient map to} \,  \Aa_{\mathcal{E}}. 
\end{equation}
This induces a decreasing filtration on $\Aa_{\mathcal{E}}$, and we let $\nu_{(y,x)}$ denote the associated quasivaluation defined for $f \in \Aa_{\mathcal{E}}$ by 
\begin{equation}\label{eq: definition quotient valuation nu y x}
\nu_{(y,x)}(f) := \max\left\{a: f \in F^{(y,x)}_{\geq a}\right\}.
\end{equation} 

The following is a basic exercise in definitions and is useful to remember. 

\begin{lemma}\label{lemma: equivalent conditions for adapted basis}
Let $\B$ be a vector space basis for a $\mathbb{K}$-algebra $\Aa$. Let $\nu: \Aa \to \Gamma \cup \{\infty\}$ be a quasivaluation to a totally ordered discrete group. Let $\mathcal{F} = \{F^\nu_{\geq a}\}_{a\in \Gamma}$ be the associated filtration, $F^{\nu}_{\geq a} := \{f \in \Aa : \nu(f) \geq a\}$.  Then the  following are equivalent. 
\begin{enumerate} 
\item $\mathbb{B}$ has the property that $\nu(\sum \lambda_i \mathbb{b}_i) = \min\{\nu(\mathbb{b}_i): \lambda_i \neq 0\}$, for any finite collection $\lambda_i \in \K^*$ and $\bb_i \in \B$,  
\item $\mathbb{B} \cap F^\nu_{\geq a}$ is a basis for $F^\nu_{\geq a}$ for all $a \in \Gamma$. 
\end{enumerate} 
\end{lemma}

\begin{definition}\label{definition: adapted basis of quasivaluation}
In the setting of Lemma~\ref{lemma: equivalent conditions for adapted basis}, if $\mathbb{B}$ is a basis satisfying the equivalent conditions in Lemma~\ref{lemma: equivalent conditions for adapted basis}, we call it an \textbf{adapted basis} for the quasivaluation. 
\end{definition}

The following is a general lemma about detropicalizations that will be useful.  
Recall first that, by the definition of strict duality Definition~\ref{def_dual}, there is a bijective correspondence between the set of charts $\pi(\MM)$ on a polyptych lattice $\MM$, and the set of maximal cones of the mutation fan $\Sigma(\NN)$ of its strict dual $\NN$. For $\alpha \in \pi(\MM)$, let $C_\alpha$ denote the corresponding maximal cone in $\Sigma(\NN)$.

\begin{lemma}\label{lemma: chart additions have to appear}
Let $(\MM,\NN,\v,\w)$ be a strict dual pair of finite polyptych lattices over $\Z$ of rank $r$. Let $(\Aa_\MM, \fv: \Aa_\MM \to S_\MM)$ be a detropicalization of $\MM$ and $\B$ a convex adapted basis for $\fv: \Aa_\MM \to S_\MM$. Let $m, m' \in \MM$ and let $\bb_m, \bb_{m'} \in \B$ denote their associated convex adapted basis elements. Then for every $\alpha \in \pi(\MM)$, the basis element 
$\bb_{m+_\alpha m'}$ has to appear (i.e., with non-zero coefficient) in the product $\bb_m  \bb_{m'}$ in $\Aa_\MM$. 
\end{lemma} 

\begin{proof} 
Let $C_\alpha$ be the maximal cone in $\Sigma(\NN)$ corresponding to the chart $\alpha$. 
We know from Lemma~\ref{lemma: res cone surj} that the map $\MM \to \Sp(\NN) \to \Hom(C_\alpha \cap \NN,\Z)$ is a bijection.  Composing with evaluation against a choice of basis in  $C_\alpha$ yields an injection $\MM \to \Z^{\mathrm{rank} (\MM)}$ by Lemma~\ref{lemma: injectivity}. For every basis element $n$, $m+_\alpha m'$ has to achieve the minimum value $\w(n)(m)+\w(n)(m')$ because $n \in C_\alpha$. Note $\w(n) \circ \fv$ evaluates to $\w(n)(\fv(\bb_m)) = \w(n)(m)$ for any $m \in \MM$.  We saw above that since $\w(n)$ is a point, the composition $\w(n) \circ \fv$ is a valuation and $\mathbb{B}$ is still a convex adapted basis for $\w(n)\circ \fv$. By what we have seen, $\bb_{m+_\alpha m'}$ achieves the minimum for $\w(n)\circ \fv$ in the expression for $\mathbb{b}_m \bb_{m'}$. By the injectivity in~Lemma~\ref{lemma: injectivity}, $\bb_{m+_\alpha m'}$ is in fact the only element that can take these values, so $\bb_{m+_\alpha m'}$ must appear.
\end{proof}

Now we return to the analysis of the algebra $\Aa_{\mathcal{E}}$ defined above. We define 
\begin{equation}\label{eq: def B calE a}
\mathbb{B}_{\mathcal{E},(y,x),a} := \{\mathbb{b}_{(u,v)} \in \mathbb{B}_{\mathcal{E}}: \langle (y,x),(u,v)\rangle = \hat{\nu}_{(y,x)}(\hat{\mathbb{b}}_{(u,v)}) \geq a\}.
\end{equation}

\begin{lemma}\label{lemma: B calE is CAB for nu y x}
Let $\nu_{(y,x)}: \Aa_{\mathcal{E}} \to \Z \cup\{\infty\}$ be the quasivaluation defined by \eqref{eq: definition quotient valuation nu y x} by the filtration spaces $F^{(y,x)}_{\geq a}$ defined in~\eqref{eq: definition quotient filtration F y x geq a} for $a \in \Z$. Then:
\begin{enumerate} 
\item for any $a \in \Z$, the set $\mathbb{B}_{\mathcal{E},(y,x),a}$ in~\eqref{eq: def B calE a} is a basis for $F^{(y,x)}_{\geq a}$, 
\item for $(u,v) \in Y$ we have $\nu_{(y,x)}(\mathbb{b}_{(u,v)}) = \langle (y,x),(u,v)\rangle$, and 
\item $\mathbb{B}_{\mathcal{E}}$ is an adapted basis for $\nu_{(y,x)}$.  
\end{enumerate} 
\end{lemma}

\begin{proof} 
We start with (1). The linear independence of $\mathbb{B}_{\mathcal{E},(y,x),a}$ is immediate since we saw in Lemma~\ref{lemma: B cal E is a basis} that $\mathbb{B}_{\mathcal{E}}$ is linearly independent. It remains to show that $\mathbb{B}_{\mathcal{E},(y,x),a}$ spans $F^{(y,x)}_{\geq a}$. To do this, let $f \in F^{(y,x)}_{\geq a}$.  Choose a lift $\hat{f} \in \hat{F}_{\geq a}^{(y,x)} \subset \Aa_{{\PP}_1 \times \NN_2}$ of $f$. Since $\hat{\mathbb{B}}_{\PP_1 \times \NN_2}$ is an adapted basis for $\hat{\nu}_{(y,x)}$ (it is a subset of an adapted basis of $\Aa_{\MM_1 \times \NN_2}$ which spans the subalgebra $\Aa_{\PP_1 \times \NN_2}$), it follows that $\hat{F}_{\geq a}^{(y,x)}$ is spanned by $\{\hat{\mathbb{b}}_{(u,v)}: (u,v) \in (\PP_1 \cap \MM_1)\times \NN_2, \hat{\nu}_{(y,x)}(\hat{\mathbb{b}}_{(u,v)}) = \langle (y,x),(u,v)\rangle \geq a\}$. Thus we may express $\hat{f}$ as follows: 
\begin{equation*}\label{eq: hat f in hat b_uv}
\hat{f} = \sum c_{(u,v)} \hat{\mathbb{b}}_{(u,v)} 
\end{equation*}
where the sum is finite, $c_{(u,v)} \in \K$, and for each $\hat{\mathbb{b}}_{(u,v)}$ that appears, we have $(u,v) \in ({\PP}_1 \cap \MM_1) \times \NN_2$ and $\langle (y,x),(u,v)\rangle \geq a$. 
By taking the quotient to $\Aa_{\mathcal{E}}$, we conclude that
$$
f = \sum_{(u,v) \in (\PP_1\cap\MM_1) \times \NN_2} c_{(u,v)} \mathbb{b}_{(u,v)} 
$$
where each $\mathbb{b}_{(u,v)}$ has the property $\langle (y,x),(u,v)\rangle \geq a$. 

To finish the argument, it would suffice to show that for any $\mathbb{b}_{(u,v)}$ with $(u,v) \not \in Y$ but $\langle (y,x),(u,v)\rangle \geq a$, we may rewrite $\mathbb{b}_{(u,v)}$ as an element in the span of $\mathbb{B}_{\mathcal{E},(y,x),a}$.  As before, we work with the preimage $\hat{\mathbb{b}}_{(u,v)}$. First recall that $(u,v) \not \in Y$ means that $\Psi_{\PP_1}(u) > 0$. By Lemma~\ref{lemma: tau calE arithmetic} we know 
$\hat{\mathbb{b}}_{(u,v)}$ is a non-zero scalar multiple of $ \zeta_{\mathcal{E}}^{\Psi_{\PP_1}(u)} \cdot \hat{\mathbb{b}}_{(u-\Psi_{\PP_1}(u) \mathbb{g}_1,v)}$ where $\hat{\mathbb{b}}_{(u-\Psi_{\PP_1}(u) \mathbb{g}_1,v)} \in \mathbb{B}_{\mathcal{E}}$.  Thus, as in the proof of Lemma~\ref{lemma: B cal E is a basis}, when we take the quotient to $\Aa_{\mathcal{E}}$ we may write 
\begin{equation*}\label{eq: b uv in F geq a}
\mathbb{b}_{(u,v)} = c \cdot 
\left(\sum_{j \in [s_2]} \mathbb{b}_{(0,n^2_j))}\right)^{\Psi_{\PP_1}(u)} \cdot 
\mathbb{b}_{(u - \Psi_{\PP_1}(u) \cdot \mathbb{g}_1, v)}
\end{equation*}
for $c$ a non-zero constant. 
We showed in that proof that the RHS lies in the span of $\mathbb{B}_{\mathcal{E}}$. Therefore, what remains for us to show is that the RHS lies in the span of $\mathbb{B}_{\mathcal{E},(y,x),a}$. To see this, we may work upstairs; we wish to show that the basis elements $\hat{\mathbb{b}}_{(u',v')}$ that appear in the expansion of 
\begin{equation}\label{eq: upstairs}
\left(\sum_{j \in [s_2]} \hat{\mathbb{b}}_{(0,n^2_j)}\right)^{\Psi_{\PP_1(u)}} \cdot 
\hat{\mathbb{b}}_{(u - \Psi_{\PP_1}(u) \cdot \mathbb{g}_1, v)}
\end{equation}
have the property that $\langle (y,x),(u',v')\rangle = \hat{\nu}_{(y,x)}(\hat{\mathbb{b}}_{(u',v')}) \geq a$. Since $\hat{\mathbb{B}}_{\PP_1 \times \NN_2}$ is an adapted basis for $\hat{\nu}_{(y,x)}$, by~Lemma~\ref{lemma: equivalent conditions for adapted basis} it suffices to show that $\hat{\nu}_{(y,x)}$ evaluates to be $\geq a$ on~\eqref{eq: upstairs}. Recall that $\hat{\mathbb{b}}_{(u,v)}$ is assumed to satisfy $\langle (y,x),(u,v)\rangle \geq a$. By properties of valuations and the definition of $\hat{\nu}_{(y,x)}$ (and using that $\zeta_{\mathcal{E}} := \hat{\mathbb{b}}_{(\mathbb{g}_1, 0_{\NN_2})}$), we also know $\hat{\nu}_{(y,x)}(\zeta_{\mathcal{E}}^{\Psi_{\PP_1}(u)}\cdot \hat{\mathbb{b}}_{(u-\Psi_{\PP_1}(u)\mathbb{g}_1,v)}) = \hat{\nu}_{(y,x)}(\hat{\mathbb{b}}_{(u,v)}) =\langle (y,x),(u,v)\rangle$. Now we claim that 
\begin{equation}\label{eq: comparing nu value}
\hat{\nu}_{(y,x)}(\zeta_{\mathcal{E}}) = \hat{\nu}_{(y,x)}\left(\sum_{j \in [s_2]} \hat{\mathbb{b}}_{(0,n^2_j)}\right).
\end{equation}
 Indeed, may compute the LHS of \eqref{eq: comparing nu value} as 
$$
\hat{\nu}_{(y,x)}(\zeta_{\mathcal{E}}) = \hat{\nu}_{(y,x)}(\hat{\mathbb{b}}_{(\mathbb{g}_1, 0_{\NN_2})}) = \w_1(x)(\mathbb{g}_1). 
$$
Similarly we may compute the RHS of \eqref{eq: comparing nu value} as 
\begin{equation*}
\hat{\nu}_{(y,x)}\left(\sum_{j \in [s_2]} \hat{\mathbb{b}}_{(0,n^2_j)}\right)  = 
\min_{j \in [s_2]} \{ \hat{\nu}_{(y,x)}(\hat{\mathbb{b}}_{(0,n^2_j)}) \} 
 = 
\min_{j \in [s_2]} \{ \langle (y,x), (0, n^2_j)) \rangle \} 
 = \min_{j \in [s_2]} \{ \w_2(n^2_j)(y) \} 
 = \Psi_{\PP_2}(y)
\end{equation*}
by the definition \eqref{eq: def supp tildeP_i} of $\Psi_{\PP_2}$. Since $(y,x) \in X^\vee$, we have $\w_1(x)(\mathbb{g}_1) = \Psi_{\PP_2}(y)$. But this implies that the RHS and LHS of \eqref{eq: comparing nu value} are equal. This in turn implies that $\hat{\nu}_{(y,x)}$ evaluates on~\eqref{eq: upstairs} to be $\langle (y,x),(u,v)\rangle$, which is by assumption $\geq a$. This completes the proof of (1).

We now prove (2). Since $\hat{\nu}_{(y,x)}(\hat{\mathbb{b}}_{(u,v)}) = \langle (y,x),(u,v)\rangle$, it follows from the definition of the quotient quasivaluation that $\nu_{(y,x)}(\mathbb{b}_{(u,v)}) \geq \langle (y,x),(u,v)\rangle$. Now suppose for a contradiction that $\nu_{(y,x)}(\mathbb{b}_{(u,v)}) = c$ for some $c >  \langle (y,x),(u,v)\rangle$. Then $\mathbb{b}_{(u,v)} \in F^{(y,x)}_{\geq c}$. We showed in (1) that $F^{(y,x)}_{\geq c}$ is spanned by the set of $\mathbb{b}_{(u',v')}$ with the property $\langle (y,x),(u',v')\rangle \geq c$. Note that $\mathbb{b}_{(u,v)}$ is not contained in this set. So we have a linear relation 
$$
\mathbb{b}_{(u,v)} =\sum_{(u',v')} c_{u',v'} \mathbb{b}_{(u',v')}
$$
where the finite sum is over $(u',v') \in Y$ with the property $\langle (y,x),(u',v')\rangle \geq c$. This is a non-trivial linear relation in $\mathbb{B}_{\mathcal{E}}$, but by Lemma~\ref{lemma: B cal E is a basis} we know $\mathbb{B}_{\mathcal{E}}$ is a basis, so this is a contradiction. Thus $\nu_{(y,x)}(\mathbb{b}_{(u,v)}) = \langle (y,x),(u,v)\rangle$, as desired.

Lastly, we prove (3). By (2), the set $\mathbb{B}_{\mathcal{E},(y,x),a}$ is equal to the set $\mathbb{B}_{\mathcal{E}} \cap F^{(y,x)}_{\geq a}$. By (1), this set is a basis of $F^{(y,x)}_{\geq a}$. Then by Definition~\ref{definition: adapted basis of quasivaluation} it follows that $\mathbb{B}_{\mathcal{E}}$ is an adapted basis for $\nu_{(y,x)}$, as claimed. 
\end{proof}

Recall that the polyptych lattices $\mathcal{E},\mathcal{F}$ are of rank $r_1+r_2-1$. 
Our next step will be to build a valuation $\nu: \Aa_{\mathcal{E}} \to \Z^{r_1+r_2-1} \cup \{\infty\}$.  We will then use $\nu$ to construct the valuation $\nu_{\mathcal{E}}: \Aa_{\mathcal{E}} \to S_{\mathcal{E}}$. 

 Fix $(C_{\sigma,j} \times C_\tau) \cap X^\vee$ a maximal cone of $\Sigma(\mathcal{F})$. For simplicity of notation, define $$r := r_1+r_2-1$$ for the remainder of this section. Let $\{(y_1,x_1), \cdots, (y_r,x_r)\} \subset (C_{\sigma,j} \times C_\tau) \cap X^\vee$ be a $\Q$-basis of the $\Q$-span of $(C_{\sigma,j} \times C_\tau) \cap X^\vee$. Interpreting $(y_i,x_i)$ as a point on $\MM_1 \times \NN_2$, we may define the associated valuations $\hat{\nu}_{(y_i,x_i)}: \Aa_{{\PP}_1 \times \NN_2} \to \Z \cup \{\infty\}$ as we did above. 
We then define the concatenated high-rank valuation 
\begin{equation}\label{eq: upstairs high rank} 
\hat{\nu} := \hat{\nu}_{(y_1,x_1)} \circledast \cdots \circledast \hat{\nu}_{(y_r,x_r)}: \Aa_{{\PP}_1 \times \NN_2} \to \Z^r \cup \{\infty\}
\end{equation} 
using the method given in \cite[Appendix A]{EscobarHaradaManon-PL}.

The following lemmas are not proven in \cite[Appendix A]{EscobarHaradaManon-PL} but are needed in the arguments to follow. The proofs are straightforward. 

\begin{lemma}\label{lemma: adapted basis for concat val}
In the general setting of \cite[Appendix A]{EscobarHaradaManon-PL},  if $\B$ is simultaneously an adapted basis for each of $\nu_1,\cdots,\nu_r$ (where $\nu_i: \Aa \to \Z$ for each $i$), then $\B$ is also an adapted basis for $\nu_1 \circledast \nu_2 \circledast \cdots \circledast \nu_r$. 
\end{lemma} 

\begin{lemma}\label{lemma: concat val on CAB}
In the setting of the above lemma, for $\mathbb{b} \in \B$ we have $\nu_1\circledast \cdots \circledast \nu_r(\mathbb{b}) = (\nu_1(\mathbb{b}), \cdots, \nu_r(\mathbb{b}))$.  
\end{lemma} 

\begin{lemma}\label{lemma: upstairs concat is a valuation}
In the setting of the above lemma, the quasivaluation $\nu_1 \circledast \nu_2 \circledast \cdots \circledast \nu_r$ is a valuation. 
\end{lemma}

As we saw above, for each $(y_i,x_i)$ we can use pushforward filtrations to define a quasivaluation, denoted $\nu_{(y_i,x_i)}$ on $\Aa_{\mathcal{E}}$. This suggests we can define a quasivaluation on 
$\Aa_{\mathcal{E}}$ in two different ways: we can pushforward and then concatenate, or, concatenate and then pushforward.  
The following lemma states that these two operations commute. This follows from examining the definitions, so we omit the proof. 

\begin{lemma}\label{lemma: two definitions agree}
Let 
$\nu_{pf}: \Aa_{\mathcal{E}} \to \Z^r \cup \{\infty\}$ be the quasivaluation obtained from the pushforward filtration in $\Aa_{\mathcal{E}}$ obtained by pushing forward the filtration obtained from the valuation $\hat{\nu}$ defined in~\eqref{eq: upstairs high rank}. Let $\nu_c$ denote the concatenated quasivaluation $\nu_{(y_1,x_1)} \circledast \cdots \circledast \nu_{(y_r,x_r)}$. Then: $\nu_{pf} = \nu_c$.  
\end{lemma}

Since Lemma~\ref{lemma: two definitions agree} shows that the two definitions agree, we let $$\nu := \nu_{pf}=\nu_c$$ denote the unique quasivaluation obtained by either of the operations.

\begin{lemma}\label{lemma: bijection B epsilon to Zr}
Let $\nu: \Aa_{\mathcal{E}} \to \Z^r \cup \{\infty\}$ the quasivaluation defined above. Then
$\nu$ induces an injection from $\B_{\mathcal{E}}$ to $\Z^r$.  
\end{lemma} 

\begin{proof} 
From previous arguments, we know we can view $(u,v)$ as an element of $\Hom(C_{\sigma,j} \times C_\tau \cap X^\vee, \Z)$ via the pairing ~\eqref{eq: explicit pairing}.
By Lemma~\ref{lemma: B calE is CAB for nu y x}, Lemma~\ref{lemma: adapted basis for concat val}, Lemma~\ref{lemma: concat val on CAB} and Lemma~\ref{lemma: two definitions agree}, we also know that $\mathbb{B}_{\mathcal{E}}$ is an adapted basis for $\nu$. Consider the map $\Hom((C_{\sigma,j}\times C_\tau) \cap X^\vee, \Z) \to \Z^r$ given by evaluating against $(y_1,x_1),\cdots,(y_r,x_r)$.  From the above it follows that the value of $\nu$ on $\bb_{(u,v)}$ is the image of $(u,v)$ under this map. The result follows from Lemma~\ref{lemma: injectivity}. 
\end{proof}

The next result, on multiplicative properties of $\nu$, is the main technical statement needed to construct the detropicalization.

\begin{proposition}\label{proposition: nu on product}
Following the notation already established, let $\mathbb{b}_{(u,v)}, \mathbb{b}_{(u',v')} \in \B_{\mathcal{E}}$. Then 
\begin{equation}\label{eq: val on A cal E multiplicative}
\nu(\mathbb{b}_{(u,v)}\mathbb{b}_{(u',v')}) = \nu(\mathbb{b}_{(u,v)}) + \nu(\mathbb{b}_{(u',v')}).
\end{equation} 
\end{proposition}

\begin{proof} 
As in the proof of Lemma~\ref{lemma: bijection B epsilon to Zr}, we know $\nu(\mathbb{b}_{(u'',v'')})= (\langle (y_1,x_1),(u'',v'')\rangle, \cdots, \langle (y_r,x_r),(u'',v'')\rangle)$ for any $(u'',v'') \in Y$. We also know $\mathbb{B}_{\mathcal{E}}$ is an adapted basis for $\nu$. 
Thus, in order to compute the LHS of~\eqref{eq: val on A cal E multiplicative}, we need to express $\mathbb{b}_{(u,v)}\mathbb{b}_{(u',v')}$ as a linear combination of basis elements in $\mathbb{B}_{\mathcal{E}}$, and then take the minimum of the $\nu$-values of the elements that appear.

To accomplish this, we work upstairs. Let $S = \ptconv(\{(u,v)+_{\tau',\sigma'}(u',v')\}_{\tau' \in \pi(\MM_1),\sigma'\in \pi(\NN_2)}) \cap \MM_1 \times \NN_2$. We know that $\hat{\mathbb{b}}_{(u,v)} \hat{\mathbb{b}}_{(u',v')}$ projects to $\mathbb{b}_{(u,v)}\mathbb{b}_{(u',v')}$, and from \cite[Lemma 7.3]{EscobarHaradaManon-PL} we also know that 
\begin{equation}\label{eq: prod buv}
    \hat{\mathbb{b}}_{(u,v)} \hat{\mathbb{b}}_{(u',v')} = \sum_{(u'',v'') \in S} c_{(u'',v'')} \hat{\mathbb{b}}_{(u'',v'')}. 
\end{equation}
Moreover, we also know from Lemma~\ref{lemma: chart additions have to appear} that for any $\sigma \in \pi(\NN_2), \tau \in \pi(\MM_1)$, the basis element $\hat{\mathbb{b}}_{(u+_\tau u', v+_\sigma v')}$ must appear (i.e. the corresponding coefficient is non-zero) in the sum~\eqref{eq: prod buv}. The issue with the expression~\eqref{eq: prod buv} is that the $\hat{\mathbb{b}}_{(u'',v'')}$ may not satisfy the condition $(u'',v'') \in Y$. Therefore, the next step is to find a different element in $\Aa_{\PP_1 \times \NN_2}$ which also maps to $\mathbb{b}_{(u,v)}\mathbb{b}_{(u',v')}$ but is a linear combination of basis elements whose subscripts lie in $Y$. We follow the same line of reasoning as in the proof of Lemma~\ref{lemma: B cal E is a basis}, and in particular~\eqref{eq: rewrite with tau mathcal E} and~\eqref{eq: b_uv in quotient}. Specifically, replacing the index $(u,v)$ in~\eqref{eq: b_uv in quotient} with $(u'',v'')$ in $S$, we obtain 
\begin{equation}\label{eq: b_uvprime in quotient}
\mathbb{b}_{(u'',v'')} = c \cdot 
\left(\sum_{j \in [s_2]} \mathbb{b}_{(0_{\MM_1},n^2_j)}\right)^{\Psi_{\PP_1}(u'')} \cdot 
\mathbb{b}_{(u'' - \Psi_{\PP_1}(u'') \cdot \mathbb{g}_1, v'')}.
\end{equation}
Now consider the sum upstairs which corresponds to the RHS of~\eqref{eq: b_uvprime in quotient}.  
We now proceed to show that, among all  $\hat{\mathbb{b}}_{(u''',v''')}$ that appear in the upstairs sum corresponding to ~\eqref{eq: b_uvprime in quotient} for varying $(u'',v'') \in S$, that the minimum $\hat{\nu}$-value achieved among them is equal to $\nu(\mathbb{b}_{(u,v)})+\nu(\mathbb{b}_{(u',v')})$. Moreover, we will show that the basis element that achieves this minimum occurs exactly once, as $(u'',v'')$ vary in $S$, i.e., there is no cancellation. This would prove that the RHS of~\eqref{eq: val on A cal E multiplicative} is equal to the LHS.

We proceed in two steps. We first examine the expression (the upstairs analogue of) ~\eqref{eq: b_uvprime in quotient} for each $(u'',v'')$ separately. We claim that, in this expression (when expanded in basis elements with indices in $Y$), a basis element that achieves the minimum is 
$$
\hat{\mathbb{b}}_{min, (u'',v'')} := \hat{\mathbb{b}}_{(u'' - \Psi_{\PP_1}(u'') \cdot \mathbb{g}_1, v'' +_\sigma \Psi_{\PP_1}(u'')\cdot n^2_j)}.
$$
 This claim follows from Lemma~\ref{lemma: upstairs concat is a valuation}, Lemma~\ref{lemma: adapted basis for concat val}, Lemma~\ref{lemma: concat val on CAB}, and the fact that $(y_i,x_i)$ is in $C_{\sigma,j} \times C_\tau$, so we know $\w_2(n^2_t)(y_i)$ achieves its minimum at $t=j$ and $\w_2(n^2_t)(y_i) \geq \w_2(n^2_j)(y_i)$ for $t \neq j$. 
Recall from~Lemma~\ref{lemma: bijection B epsilon to Zr} that $\nu$ induces an injection from $\mathbb{B}_{\mathcal{E}}$ to $\Z^r$. Equivalently, evaluation against $(y_1,x_1),(y_2,x_2),\cdots,(y_r,x_r)$ is an injection from $Y$ and $\Z^r$. From this it follows that $\hat{\mathbb{b}}_{min, (u'',v'')}$ is the unique basis element appearing in~\eqref{eq: b_uvprime in quotient} that achieves this minimum.

Now consider the case $(u'',v'')= (u+_\tau u', v+_\sigma v')$. Define 
$$
\hat{\mathbb{b}}_{min} := \hat{\mathbb{b}}_{min, (u+_\tau u',v+_\sigma v')} = \hat{\mathbb{b}}_{(u+_\tau u' - \Psi_{\PP_1}(u+_\tau u') \cdot \mathbb{g}_1, v+_\sigma v' +_\sigma \Psi_{\PP_1}(u+_\tau u')\cdot n^2_j)}. 
$$
We claim that, among all $\hat{\mathbb{b}}_{min,(u'',v'')}$ for all possible $(u'',v'') \in S$, $\hat{\mathbb{b}}_{min}$ achieves the minimum possible $\hat{\nu}$-value.  This follows from the fact that $(y_i,x_i)$ are in $C_{\sigma,j} \times C_\tau$ and from the properties of points. Moreover, the $\hat{\nu}$-value is precisely $\hat{\nu}(\mathbb{b}_{(u,v)}) + \hat{\nu}(\mathbb{b}_{(u',v')}) = \nu(\mathbb{b}_{(u,v)})+\nu(\mathbb{b}_{(u',v')})$, as a computation shows.

To complete the proof, it suffices to show that $\hat{\mathbb{b}}_{min}$ cannot be cancelled by any other $\hat{\mathbb{b}}_{min,(u'',v'')}$. Recall that $\tau$ and $\sigma$ are fixed. 
We will make this argument in two steps. We first consider the special case when $v=v'=0$, so the two original basis elements are of the form $\mathbb{b}_{(u,0)}, \mathbb{b}_{(u',0)}$. 
Then we have $\hat{\mathbb{b}}_{(u,0)}\hat{\mathbb{b}}_{(u',0)} = \sum c_{u''} \hat{\mathbb{b}}_{(u'',0)}$ where the sum is over distinct $u''$ that appear in $\ptconv(\{u+_{\tau'} u'\})$. 
In this case, we have $\hat{\mathbb{b}}_{min,(u'',0)} = \hat{\mathbb{b}}_{(u''-k\mathbb{g}_1, k n^2_j)}$ where $k=\Psi_{\PP_1(u'')}$. Moreover, 
we also have $\hat{\mathbb{b}}_{min}=\hat{\mathbb{b}}_{(u+_\tau u'- k\mathbb{g}_1, k n^2_j)}$ where $k=\Psi_{\PP_1}(u+_\tau u')$. Now we need to show that if $\hat{\mathbb{b}}_{(u''-k'\mathbb{g}_1, k'n^2_j)} = \hat{\mathbb{b}}_{(u+_\tau u'- k\mathbb{g}_1, k n^2_j)}$ then $u'' = u+_\tau u'$. 
Indeed we know that $\hat{\mathbb{b}}_{(u''-k'\mathbb{g}_1, k'n^2_j)} = \hat{\mathbb{b}}_{(u+_\tau u'- k\mathbb{g}_1, k n^2_j)}$ means $(u''-k'\mathbb{g}_1, k'n^2_j) = (u+_\tau u'- k\mathbb{g}_1, k n^2_j)$ which in turn implies $k=k'$, but this yields that $u'' = u+_\tau u'$, as desired. Moreover, by Lemma~\ref{lemma: chart additions have to appear} we know that $\hat{\mathbb{b}}_{min}$ must appear. This proves the $v=v'=0$ case.

Finally, we must show that we can reduce to the above special case. Suppose we have $\mathbb{b}_{(u,v)}$ and $\mathbb{b}_{(u',v')}$, where $(u,v),(u',v') \in Y$. Then $\mathbb{b}_{(u,v)}=\mathbb{b}_{(u,0)}\mathbb{b}_{(0,v)}$ because chart addition with $0$ is well-defined in both $\MM_1$ and $\NN_2$ and similarly for $(u',v')$. We can write $\mathbb{b}_{(u,v)}\mathbb{b}_{(u',v')} = \mathbb{b}_{(u,0)}\mathbb{b}_{(u',0)}\mathbb{b}_{(0,v)}\mathbb{b}_{(0,v')}$. Now working upstairs, we already saw above that, after rewriting the product expression of $\hat{\mathbb{b}}_{(u,0)}\hat{\mathbb{b}}_{(u',0)}$ modulo $\mathbb{h}$, there is a unique basis element $\hat{\mathbb{b}}_{(u+_\tau u'- k\mathbb{g}_1, k n^2_j)}$ (where $k=\Psi_{\PP_1}(u+_\tau u')$) that achieves the minimum. Similar arguments also show that in the product of $\hat{\mathbb{b}}_{(0,v)}\hat{\mathbb{b}}_{(0,v')}$, the unique basis element achieving the minimum is $\hat{\mathbb{b}}_{(0,v+_\sigma v')}$. It remains to argue that in the product  $\hat{\mathbb{b}}_{(u+_\tau u'- k\mathbb{g}_1,  k n^2_j)}\hat{\mathbb{b}}_{(0,v+_\sigma v')}$, the unique element achieving the minimum (i.e. all others have a strictly greater value) is $\hat{\mathbb{b}}_{(u+_\tau u' - k\mathbb{g}_1, v+_\sigma v' +_\sigma k n^2_j)}$. As in the previous arguments, this element appears by Lemma~\ref{lemma: chart additions have to appear} and it is unique because the $\{(y_i,x_i)\}$ form a basis of $C_{\sigma,j} \times C_\tau$. This concludes the proof. 

\end{proof}

From Proposition~\ref{proposition: nu on product} we can obtain the following.  

\begin{proposition}\label{lemma: consequences of additivity}
Let the notation and assumptions be as above. Let $\mathcal{F}$ denote the filtration associated to the quasivaluation $\nu: \Aa_{\mathcal{E}} \to \Z^r \cup \{\infty\}$ constructed above. We have 
\begin{enumerate} 
\item $\nu: \Aa_{\mathcal{E}} \to \Z^r \cup \{\infty\}$ has $1$-dimensional leaves. 
\item $\Aa_{\mathcal{E}}$ is an integral domain. 
\item $\nu: \Aa_{\mathcal{E}} \to \Z^r \cup \{\infty\}$ is a valuation. 
\item The function $\nu_1: \Aa_{\mathcal{E}}$ that reads off the first coordinate of $\nu$, i.e. $\nu_1  := \pi_1 \circ \nu$ where $\pi_1: \Z^r \to \Z$ reads the first coordinate, is a $\Z$-valuation on $\Aa_{\mathcal{E}}$. 
\end{enumerate} 
\end{proposition}

\begin{proof}
Let $\mathsf{F}^\nu$ be the filtration associated to $\nu$ on $\Aa_{\mathcal{E}}$. Lemma~\ref{lemma: B calE is CAB for nu y x}(2) and Lemma~\ref{lemma: concat val on CAB} and Lemma~\ref{lemma: bijection B epsilon to Zr} together imply that for each $\bar{a} \in \Z^r$ we have $\mathsf{F}^\nu_{\geq \bar{a}}/\mathsf{F}^\nu_{> \bar{a}}$ is at most $1$-dimensional, spanned by $\mathbb{b}_{(u,v)}$ where $(u,v)$ is the unique element such that $\nu(\mathbb{b}_{(u,v)}) = \bar{a}$ (if such $(u,v)$ exists). This shows Claim (1).  

It follows from (1) and Proposition~\ref{proposition: nu on product} that $gr_{\mathcal{F}}(\Aa_{\mathcal{E}})$ is an integral domain. In general, if $\mathsf{F}$ is a multiplicative filtration on $\Aa$ with $\cap_{a} \mathsf{F}^\nu_a = \{0\}$ and $gr_{\mathsf{F}}(\Aa)$ is an integral domain, then $\Aa$ is an integral domain. Since $\Aa$ has an adapted basis with respect to $\nu$ we know $\cap_a \mathsf{F}^\nu_a = \{0\}$, so this shows (2). Since the associated graded $gr_{\mathsf{F}^\nu}(\Aa_{\mathcal{E}})$ is an integral domain, $\nu$ is  a valuation, which is claim (3). 

It remains to show (4). Given a $\Z^r$-valuation $\nu$ we have $\nu(fg) = \nu(f)+\nu(g)$ (where both $f,g \neq 0$) so it follows immediately that $\nu_1(fg)=\nu_1(f)+\nu_1(g)$. Also we know $\nu(f+g) \geq \min\{\nu(f),\nu(g)\}$ (where $f,g,f+g \neq 0$). Now if $\bar{c} \geq \min\{\bar{a},\bar{b}\}$ in lex order, then $c_1 \geq \min\{a_1,b_1\}$ in $\Z$ because otherwise $\bar{c}=(c_1,\cdots,c_r) < \bar{a}=(a_1,\cdots,a_r)$ and similarly $\bar{c}<\bar{b}$, yielding a contradiction. Hence $\nu_1$ is also a valuation, as desired. 

\end{proof}

Using the high-rank valuation $\nu$ of Lemma~\ref{lemma: consequences of additivity}, we now construct a valuation $\nu_{\mathcal{E}}: \Aa_{\mathcal{E}} \to S_{\mathcal{E}}$ which gives us a detropicalization of $\Aa_{\mathcal{E}}$. In fact, by strict duality and \cite[Proposition 4.10]{EscobarHaradaManon-PL} we know $S_{\mathcal{E}} \cong P_{\mathcal{F}}$ as idempotent semialgebras, so it is equivalent to construct a valuation $\nu_{\mathcal{E}}: \Aa_{\mathcal{E}} \to P_{\mathcal{F}}$. 
The definition is as follows. Given $\mathbb{b}_{(u,v)} \in \mathbb{B}_{\mathcal{E}}$, we may associate to it the element $(u,v)$, which by Proposition~\ref{proposition: dual map E to Sp calF} we may view as an element of $\Sp(\mathcal{F})$. 
With this said, we may define $\nu_{\mathcal{E}}$ as follows. Recall that since $\mathbb{B}_{\mathcal{E}}$ is a basis for $\Aa_{\mathcal{E}}$, any element $h \neq 0$ in $\Aa_{\mathcal{E}}$ can be expressed uniquely as a finite linear combination $h = \sum_{(u,v)} c_{(u,v)} \mathbb{b}_{(u,v)}$, $c_{(u,v)} \in \K$. We may now define 
\begin{equation}\label{eq: def valuation calE}
\nu_{\mathcal{E}}: \Aa_{\mathcal{E}} \to P_{\mathcal{F}}, \quad \mathbb{b}_{(u,v)} \mapsto (u,v) \in \Sp(\mathcal{F}) \subset P_{\mathcal{F}}, \quad 0 \neq h =\sum_{(u,v)} c_{(u,v)} \mathbb{b}_{(u,v)} \mapsto \min_{c_{(u,v)}\neq 0} \{(u,v)\}
\end{equation} 
where the last expression is interpreted as a min-combination of the piecewise-linear functions on $\mathcal{F}$ given by $(u,v) \in Y \cong \Sp(\mathcal{F})$. As we noted above, using the isomorphism $S_{\mathcal{E}} \cong P_{\mathcal{F}}$ we may view~\eqref{eq: def valuation calE} as a map $\nu_{\mathcal{E}}: \Aa_{\mathcal{E}} \to S_{\mathcal{E}}$. We have the following.

\begin{theorem}\label{theorem: main}
Let $(\MM_1,\NN_1),(\MM_2,\NN_2)$ be strict dual pairs of finite polyptych lattices, of ranks $r_1$ and $r_2$ respectively. Let $\PP_1,\PP_2$ be strongly chart-convex Gorenstein PL cones in $\MM_1,\MM_2$ respectively in the sense of Definition~\ref{definition: strongly convex} and Definition~\ref{definition: Gorenstein PL cone}. Assume that $\MM_i$ and $\NN_i$ have detropicalizations $(\Aa_{\MM_i}, \nu_{\MM_i})$ and $(\Aa_{\NN_i}, \nu_{\NN_i})$ for $i=1,2$, and are also equipped with convex adapted bases $\mathbb{B}_{\MM_i}$ and $\mathbb{B}_{\NN_i}$ for $i=1,2$. 
Let $\mathcal{E},\mathcal{F}$ be the strict dual pair of polyptych lattices of rank $r := r_1+r_2-1$ constructed from the data $(\MM_i,\NN_i,\PP_i)$ for $i=1,2$ as in Section~\ref{sec: cone extensions}. Let $\Aa_{\mathcal{E}}$ denote the algebra defined in~\eqref{eq: E detrop formula} and let $\nu_{\mathcal{E}}: \Aa_{\mathcal{E}}  \to S_{\mathcal{E}} \cong P_{\mathcal{F}}$ be as defined in~\eqref{eq: def valuation calE}. Then 
\begin{enumerate} 
\item $\nu_{\mathcal{E}}$ 
is a valuation, and
\item $(\Aa_{\mathcal{E}}, \nu_{\mathcal{E}})$ is a detropicalization of $\mathcal{E}$, and
\item $\mathbb{B}_{\mathcal{E}}$ is a convex adapted basis for $\nu_{\mathcal{E}}$. 
\end{enumerate} 
The analogous statements also hold for $\mathcal{F}$, $\Aa_{\mathcal{F}}$ and $\nu_{\mathcal{F}}$. 
\end{theorem}

\begin{proof} 
We begin with (1). It is clear from the definition that $\nu_{\mathcal{E}}(f+g) \geq \min\{\nu_{\mathcal{E}}(f), \nu_{\mathcal{E}}(g)\}$, and that $\nu_{\mathcal{E}}(cf)=\nu_{\mathcal{E}}(f)$ for $c \neq 0$ a constant, and $f\neq 0$ in $\Aa_{\mathcal{E}}$. Hence the only remaining condition to check is the multiplicativity, i.e., for $f \neq 0, g \neq 0$, we must show 
\begin{equation}\label{eq: detrop multiplicativity}
\nu_{\mathcal{E}}(fg) = \nu_{\mathcal{E}}(f) + \nu_{\mathcal{E}}(g).
\end{equation} 

To see this, note that the image of $\nu_{\mathcal{E}}$ is in $P_{\mathcal{F}}$ by definition, so they are piecewise-linear functions on $\mathcal{F}$, and have finitely many cones of linearity. Such functions agree if they agree on a basis of each such cone. In order to prove~\eqref{eq: detrop multiplicativity}, it therefore suffices to check that the LHS and RHS agree on the set of elements which could be extended to a basis of a cone in $\Sigma(\mathcal{F})$. Let $(y,x) \in X^\vee$ be such an element. Suppose $h\neq 0, \in \Aa_{\mathcal{E}}$ and $h=\sum c_{(u,v)} \mathbb{b}_{(u,v)}$ as above. We observe
\begin{equation}\label{eq: val calE and val y x}
    \begin{split}
        \nu_{\mathcal{E}}(h)(y,x) & = \nu_{\mathcal{E}}(\sum c_{(u,v)} \mathbb{b}_{(u,v)})(y,x) \\
        & = (\min \{(u,v): c_{(u,v)} \neq 0\})(y,x) \\
        &= \min \{\langle (y,x),(u,v) \rangle: c_{(u,v)} \neq 0\}  \\
        &= \min\{ \nu_{(y,x)}(\mathbb{b}_{(u,v)}): c_{(u,v)} \neq 0\} \\
        & = \nu_{(y,x)}(\sum c_{(u,v)} \mathbb{b}_{(u,v)}) \\
        & = \nu_{(y,x)}(h).
    \end{split}
\end{equation}
Thus the equality~\eqref{eq: detrop multiplicativity} evaluated at $(y,x)$ is equivalent by~\eqref{eq: val calE and val y x} to
$$
\nu_{(y,x)}(fg) =\nu_{(y,x)}(f) + \nu_{(y,x)}(g). 
$$
If $(y,x)$ can be completed to a basis of a cone in $\Sigma(\mathcal{F})$ then the above equality follows from Lemma~\ref{lemma: consequences of additivity}(4). This is what we wished to show, so we obtain claim (1).

To prove (2) we need that $\Aa_{\mathcal{E}}$ is Noetherian, an integral domain, $\nu_{\mathcal{E}}$ is surjective on $\mathcal{E}$, and the Krull dimension of $\Aa_{\mathcal{E}}$ is equal to the rank of $\mathcal{E}$. We showed in Lemma~\ref{lemma: consequences of additivity}(2) that $\Aa_{\mathcal{E}}$ is an integral domain, and it is Noetherian since it is a finitely generated $\K$-algebra. 
The valuation $\nu_{\mathcal{E}}$ is surjective onto $\mathcal{E}$ by definition since each $(u,v) \in \mathcal{E} \cong Y$ is the image of the basis element $\mathbb{b}_{(u,v)}$. Finally, we compute the Krull dimension of $\Aa_{\mathcal{E}}$. First, $\Aa_{\PP_1 \times \NN_2}$ has Krull dimension $r_1+r_2$ because $\Aa_{\PP_1 \times \NN_2}[\zeta_{\mathcal{E}}^{-1}] = \Aa_{\MM_1} \otimes \Aa_{\NN_2}$. Then $\Aa_{\mathcal{E}} = \Aa_{\PP_1 \times \NN_2}/\langle \mathbb{h}\rangle$ is obtained by taking a quotient by a non-zero non-unit, so the result follows.  

The last claim (3) follows from the definition of $\nu_{\mathcal{E}}$, and the analogous statements for $\mathcal{F}$ hold by swapping indices in the arguments.   
\end{proof}

Our next result shows that, if the original Gorenstein PL cones $\PP_1$ and $\PP_2$ have dual PL cones which are also Gorenstein, then the cone extensions $\mathcal{E}$ and $\mathcal{F}$ can also be naturally equipped with Gorenstein PL cones. 
We discuss this further in Section~\ref{section: examples}.

\begin{proposition}\label{proposition: iteration} 
In the setting of Theorem~\ref{theorem: main}, suppose in addition that the PL dual cones $\PP_1^\vee \subset (\NN_1)_\R$ and $\PP_2^\vee \subset (\NN_2)_\R$ are also Gorenstein PL cones in the sense of Definition~\ref{definition: Gorenstein PL cone}. Then $X_\R \cap ((\MM_1)_\R \times \PP_2^\vee)$ and $X^\vee_\R \cap ((\MM_2)_\R \times \PP_1^\vee)$ are also strongly convex Gorenstein PL cones in $\mathcal{E}_\R$ and $\mathcal{F}_\R$ respectively.  
\end{proposition} 

\begin{proof} 
By our assumption 
we may describe $\PP_1^\vee$ and $\PP_2^\vee$ as  
\begin{equation}\label{eq: def PP1 vee PP2 vee}
\PP_1^\vee = \cap_{a=1}^{\ell_1} \mathcal{H}_{m^1_a, 0}, \quad 
\PP_2^\vee = \cap_{k=1}^{\ell_2} \mathcal{H}_{m^2_k,0}
\end{equation} 
where $\{m^1_1,\cdots, m^1_{\ell_1}\} \subset \MM_1, \{m^2_1,\cdots,m^2_{\ell_2}\} \subset \MM_2$, and (from the proof of Lemma~\ref{lemma: dual of PL cone is a PL cone}(2) it follows that) we may take the $m^1_a$ to be contained in $\PP_1$ and the $m^2_k$ in $\PP_2$. In particular, note that $\Psi_{\PP_1}(m^1_a) \geq 0$ for all $a$. Furthermore, by hypothesis, the assumptions of~Definition~\ref{definition: Gorenstein PL cone} are satisfied.   
In particular, we assume that in there exists $\mathbb{f}_i$, $i=1,2$, such that $\mathbb{f}_i$ is in the lineality space of $\NN_i$ and that for any $j \in [\ell_i]$ we have $\v_i(m^i_j)(\mathbb{f}_i)=1$. 
Below, we will show that $X^\vee_\R \cap ((\MM_2)_\R \times \PP_1^\vee)$ is a Gorenstein PL cone in $\mathcal{F}_\R$ (where we identify the set of elements of $\mathcal{F}_\R$ with $X^\vee_\R$). The analogous statement within $\mathcal{E}_\R$ for $X_\R \cap ((\MM_1)_\R \times \PP_2^\vee)$ is analogous. 

We begin with some preliminaries. We first claim that if $x \in \PP_1^\vee$ and $\w_1(x)(\mathbb{g}_1)=0$, then $x=0$. Suppose such an $x$ is given. For any $u\in(\MM_1)_\R$, choose $N \in \R$ with $N \ge-\Psi_{\PP_1}(u)$. By Lemma~\ref{lemma: psi of a sum}, we know that $u+N\mathbb g_1\in\PP_1$. Since $x \in\PP_1^\vee$, this means $\w_1(x)(u+N\mathbb{g}_1) \geq 0$. Hence we have 
$$0\le\w_1(x)(u+N\mathbb g_1)=\w_1(x)(u)+N \cdot \w_1(x)(\mathbb g_1)=\w_1(x)(u)$$
where we have used Lemma~\ref{lemma: properties of lineality space}(2).
This implies that $\w_1(x)$, which is a minimum of finitely many linear functions on any chart, is $\geq 0$ on any chart. This means $\w_1(x)$ must be identically zero. Since $\w_1$ is a bijection, we conclude $x=0$ as claimed. 

Next we claim that every $m^1_a$ appearing in 
\eqref{eq: def PP1 vee PP2 vee} lies in $\partial \PP_1$, or equivalently, $\Psi_{\PP_1}(m^1_a)=0$. First, since $\PP^\vee \subset \mathcal{H}_{m^1_a,0}$ it is immediate that $m^1_a \in (\PP_1^\vee)^\vee = \PP_1$ by Lemma~\ref{lemma: dual of PL cone is a PL cone}. Thus $c:=\Psi_{\PP_1}(m^1_a)\ge0$. If $c=0$ then we are done. Suppose $c>0$ and define $m':=m^1_a-c \cdot \mathbb g_1$. Then by Lemma~\ref{lemma: psi of a sum} we have $\Psi_{\PP_1}(m')=0$ and hence $m'\in\PP_1$. By assumption on the original cone $\PP_1$, we know $\mathbb{g}_1 \in \PP_1$. Since $\mathbb{g}_1$ is in the lineality space, it follows that for $x\in\PP_1^\vee$, we have
$$\w_1(x)(m^1_a)=c,\w_1(x)(\mathbb g_1)+\w_1(x)(m')$$
where both summands on the RHS are $\ge0$. Define $G_a :=\{x\in\PP_1^\vee:\w_1(x)(m^1_a)=0\}$ to be the facet in $\PP_1$ corresponding to $m^1_a$. By the above displayed equation and since $c>0$, in order for the RHS to be $0$ we must have $\w_1(x)(\mathbb{g}_1)=0$ and hence $G_a$ is contained in $\{x\in\PP_1^\vee:\w_1(x)(\mathbb g_1)=0\}$, but we have seen above that this is $\{0\}$. We have assumed that the dimension $\dim_\R((\NN_1)_\R) = \mathrm{rank}(\NN_1)=r_1$ is $\geq 2$, so this means $G_a$ is not codimension $1$. This contradicts the Gorenstein hypothesis for $\PP_1^\vee$, so $c>0$ cannot occur and $m^1_a \in \partial \PP_1$ as claimed. 

Now consider the set
$$
\mathcal{S}_{\mathcal{F}} := \{(m^1_a, 0_{\NN_2})\}_{a \in [\ell_1]} \subset Y \cong \Sp(\mathcal{F}).
$$
Using $\mathcal{S}_{\mathcal{F}}$ as the set of inner normal vectors we may define a PL cone. (Here we use the bijection of $Y$ with $\Sp(\mathcal{F})$ established in Proposition~\ref{proposition: dual map E to Sp calF}.) The fact that the set $\mathcal{S}_{\mathcal{F}}$ defines a PL cone is immediate, since it is a finite set of points in $\Sp(\mathcal{F})$.  We claim that this PL cone is equal (as a set) to $X^\vee \cap (\MM_2 \times \PP_1^\vee)$, and that it is Gorenstein. 

We first show that the PL cone in $\mathcal{F}$ defined by $\mathcal{S}_{\mathcal{F}}$ is $X^\vee \cap (\MM_2 \times \PP_1^\vee)$.  For $(y,x) \in X^\vee$ we have by definition of the pairing
\begin{equation*}
\langle (y,x),(m^1_a, 0)\rangle \geq 0  \Leftrightarrow \w_1(x)(m^1_a)\geq 0 
\end{equation*}
Since we consider all possible values of $a$, this precisely means that $x \in \PP_1^\vee$, as desired. Thus $\mathcal{S}_{\mathcal{F}}$ cuts out $X^\vee \cap (\MM_2 \times \PP_1^\vee)$, as claimed. 

Next we prove that there exists an element in the lineality space of $\mathcal{F}$ with the property that it pairs to $1$ against all the defining normal vectors. Consider the element $\mathbb{g}_{\mathcal{F}} := (\w_1(\mathbb{f}_1)(\mathbb{g}_1)\cdot \mathbb{g}_2, \mathbb{f}_1)$. This is in the lineality space since $\mathbb{g}_2, \mathbb{f}_1$ are in the lineality spaces of $\MM_2$ and $\NN_1$ respectively. It is in $X^\vee$ since $\Psi_{\PP_2}(\mathbb{g}_2)=1$. We have 
\begin{equation*}
    \begin{split} 
    \langle \mathbb{g}_{\mathcal{F}}, (m^1_a - \Psi_{\PP_1}(m^1_a)\cdot \mathbb{g}_1, \Psi_{\PP_1}(m^1_a)\cdot n^2_j)\rangle & = \langle (\w_1(\mathbb{f}_1)(\mathbb{g}_1)\cdot \mathbb{g}_2, \mathbb{f}_1), (m^1_a - \Psi_{\PP_1}(m^1_a)\cdot \mathbb{g}_1, \Psi_{\PP_1}(m^1_a)\cdot n^2_j)\rangle \\
    & = \v_2(\w_1(\mathbb{f}_1)(\mathbb{g}_1)\cdot \mathbb{g}_2)(\Psi_{\PP_1}(m^1_a)n^2_j) + \w_1(\mathbb{f}_1)(m^1_a -\Psi_{\PP_1}(m^1_a)\cdot \mathbb{g}_1) \\
    & = \Psi_{\PP_1}(m^1_a)\cdot \w_1(\mathbb{f}_1)(\mathbb{g}_1) \cdot \v_2(\mathbb{g}_2)(n^2_j) + \w_1(\mathbb{f}_1)(m^1_a) - \Psi_{\PP_1}(m^1_1) \cdot \w_1(\mathbb{f}_1)(\mathbb{g}_1) \\
    & = \Psi_{\PP_1}(m^1_a)\cdot \w_1(\mathbb{f}_1)(\mathbb{g}_1) \cdot \w_2(n^2_j)(\mathbb{g}_2) + \w_1(\mathbb{f}_1)(m^1_a) - \Psi_{\PP_1}(m^1_a) \cdot \w_1(\mathbb{f}_1)(\mathbb{g}_1) \\
    & = \Psi_{\PP_1}(m^1_a)\cdot \w_1(\mathbb{f}_1)(\mathbb{g}_1) + \w_1(\mathbb{f}_1)(m^1_a) - \Psi_{\PP_1}(m^1_a) \cdot \w_1(\mathbb{f}_1)(\mathbb{g}_1) \\
    & = \w_1(\mathbb{f}_1)(m^1_a)  \\
    & = 1 \quad \textup{ by assumption on the $m^1_a$ and $\mathbb{f}_1$} 
    \end{split} 
\end{equation*}
where the third-to-last equality is because $\w_2(n^2_j)(\mathbb{g}_2)=1$ by assumption.

Now we check that the relative interiors of the facets defined by the elements $(m^1_a,0_{\NN_2})$ are non-empty and codimension $1$. By definition, the relative interior of the facet corresponding to $(m^1_a,0)$ is given by 
$$
G'_a := \{(y,x)\in X^\vee_\R \, \mid \,  \w_1(x)(m^1_a)=0,\ \w_1(x)(m^1_b)>0\,  \textup{ for } \, b\ne a)\}=X^\vee_\R\cap\big((\MM_2)_\R\times\relint(G_a)\big).
$$
We know from the assumptions on $\PP_1^\vee$ that $\relint(G_a)$ is non-empty, so let $x_0 \in \relint(G_a)$ and choose any $y_0\in(\MM_2)\R$. Using Lemma~\ref{lemma: phivee psivee bijections}(4), we have that $\psi^\vee(y_0,x_0)=(y_0+(\w_1(x_0)(\mathbb g_1)-\Psi_{\PP_2}(y_0))\mathbb g_2,,x_0)$ lies in $X^\vee_\R$, and its $\NN_1$-coordinate is still $x_0$, so it also lies in $G'_a$. In particular $G'_a$ is non-empty. To see that it is codimension $1$, note that $\phi^\vee:X^\vee_\R\to Y^\vee_\R=\partial\PP_2\times(\NN_1)_\R$ is a piecewise-linear homeomorphism which is the identity on the $\NN_1$ factor. Hence
$\phi^\vee(G'_a)=\partial\PP_2\times\relint(G_a).$
We know that $\partial\PP_2$ is a finite union of codimension-$1$ cones in $(\MM_2)_\R$, so $\dim_\R(\partial \PP_2) = r_2-1$, and $\dim(\relint(G_a))=r_1-1$ by the Gorenstein hypothesis on $\PP_1^\vee$. Therefore $\dim G'_a=r_1+r_2-2=\dim(\mathcal{F}_\R)-1$, as claimed. 

Finally, we claim that the two cones defined above, namely 
\begin{equation*}\label{eq: def Q vee and Q} 
\mathcal{Q}^\vee:=X^\vee_\R\cap\big((\MM_2)_\R\times\PP_1^\vee\big)\subset\mathcal{F}_\R,\qquad \mathcal{Q}:=X_\R\cap\big((\MM_1)_\R\times\PP_2^\vee\big)\subset\mathcal{E}_\R, 
\end{equation*} 
are strongly chart-convex. We will give details for $\mathcal{Q}^\vee$.  To prove the claim, fix a chart map $\kappa_{\alpha,\beta,i}$ of $\mathcal{F}$. We must show the chart image $\kappa_{\alpha,\beta,i}(\mathcal{Q}^\vee)$ is storngly convex in the classical sense. This is equivalent to showing that if there exist $(y,x),(y',x')\in\mathcal Q^\vee$ such that $\kappa_{\alpha,\beta,i}(y,x)+\kappa_{\alpha,\beta,i}(y',x')=0$, then $\kappa_{\alpha,\beta,i}(y,x)=\kappa_{\alpha,\beta,i}(y',x')=0$. To see this, recall that $\kappa_{\alpha,\beta,i}(y,x)$ is, by definition, the restriction to $(\PP_1(i)\cap C_\alpha)\times C_\beta$ of the functional
$$(u,v)\longmapsto \w_1(x)(u)+\v_2(y)(v).$$
With this in mind, the statement $\kappa_{\alpha,\beta,i}(y,x)+\kappa_{\alpha,\beta,i}(y',x')=0$ means that 
the linear functional $L_{C_\alpha}(\w_1(x)+\w_1(x')) + L_{C_\beta}(\v_2(y)+\v_2(y'))$ on $C_\alpha \times C_\beta$ is identically zero on $(\PP_1(i)\cap C_\alpha)\times C_\beta$. This is a codimension-$1$ subcone of the full-dimensional cone $C_\alpha\times C_\beta$, and it is contained in the codimesion-$1$ plane $\{\w_1(n^1_i)=0\}\times(\NN_2)_\R$. This implies that $L_{C_\alpha}(\w_1(x)+\w_1(x')) + L_{C_\beta}(\v_2(y)+\v_2(y'))$ is a multiple of the defining equation of $\{\w_1(n^1_i)=0\}\times(\NN_2)_\R$. This means there exists $\lambda \in \R$ such that (i) $L_{C_\beta}(\v_2(y)+\v_2(y')) = 0$ and (ii) $L_{C_\alpha}(\w_1(x)+\w_1(x')) = \lambda \cdot L_{C_\alpha}(\w_1(n^1_i))$. By~Lemma~\ref{lemma: psi alpha properties} this implies $\pi_\beta(y')+\pi_\beta(y)=0$ and $\pi_\alpha(x)+\pi_\alpha(x')=\lambda \pi_\alpha(n^1_i)$. Evaluating equation (ii) at $\mathbb{g}_1$ and using the hypothesis on $\PP_1$ and the defining equation for $X^\vee$ we obtain $\lambda = \w_1(x)(\mathbb{g}_1)+\w_1(x')(\mathbb{g}_1) = \Psi_{\PP_1}(y)+\Psi_{\PP_1}(y')$. 
Since $\mathbb g_1\in\PP_1$ and $x,x'\in\PP_1^\vee$, both terms in the middle are $\ge 0$, which implies $\lambda\ge0$. On the other hand, since $\pi_\beta(y')+\pi_\beta(y)=0$, we know $y+_\beta y'=0$. Hence 
$$\w_2(n^2_j)(y)+\w_2(n^2_j)(y') \le\w_2(n^2_j)(y+_\beta y')=\w_2(n^2_j)(0)=0$$
by properties of points. 
We conclude $\Psi_{\PP_2}(y)+\Psi_{\PP_2}(y')\le\w_2(n^2_j)(y)+\w_2(n^2_j)(y')\le 0$, i.e. $\lambda\le 0$. We have shown both inequalities $\lambda \geq 0$ and $\lambda \leq 0$, so $\lambda=0$. As we saw, both terms $\w_1(x)(\mathbb{g}_1)$ and $\w_1(x')(\mathbb{g}_1)$ are $\geq 0$, so this implies 
$\w_1(x)(\mathbb g_1)=\w_1(x')(\mathbb g_1)=0$ and hence also $\Psi_{\PP_2}(y)=\Psi_{\PP_2}(y')=0$.

We have just seen that $y, y'$ lie in $\partial \PP_2$, which means that $\pi_\beta(y), \pi_\beta(y')$ lie in $\partial \pi_\beta(\PP_2)$. On the other hand, we know from above that $\pi_\beta(y)+\pi_\beta(y')=0$. By strong chart convexity of $\PP_2$, this implies that $\pi_\beta(y)=\pi_\beta(y')=0$, hence $y=y'=0$.  Recall that we know $x,x'\in \PP_1^\vee$ and $\w_1(x')(\mathbb{g}_1) = \w_1(x)(\mathbb{g}_1)=0$. We showed above that this implies that $x,x'=0$. This concludes the proof that $(y,x)=(y',x')=0$, and therefore, $\kappa_{\alpha,\beta,i}(\mathcal{Q}^\vee)$ is strongly convex. Since $(\alpha,\beta,i)$ was arbitrary, we conclude $\mathcal{Q}^\vee$ is strongly convex. The argument for $\mathcal{Q}$ is similar with a change of indices. 
\end{proof}

\section{Examples and Further Questions}\label{section: examples}

The Gorenstein PL cone extension construction in Sections~\ref{sec: cone extensions} and~\ref{sec: E and F detrop} gives rise to a rich array of new examples of polyptych lattices. With no claim to being exhaustive, we mention some interesting classes in this section. We also document some open questions which we intend to pursue in future work.

We begin by noting that a class of examples studied in \cite[Section 8]{EscobarHaradaManon-PL} arise as a special case of the Gorenstein PL cone construction given in this paper. 

\begin{example}\label{example: MM d r} 
Fix $r,d$ positive integers. The pair of polyptych lattices 
     $\MM_{d,r}$ and $\MM_{r,d}$, which were introduced and shown to be a strict dual pair in \cite[Section 8]{EscobarHaradaManon-PL} are an example of a strict dual pair arising from the Gorenstein cone construction of Section~\ref{sec: cone extensions}. We take a moment to explain the connection. Let $\MM_1 \cong \NN_1 \cong \Z^d$, the trivial polyptych lattice of rank $d$. Similarly take $\MM_2 \cong \NN_2 \cong \Z^r$. Since both pairs of polyptych lattices are trivial in this case, the entire polyptych lattice is the lineality space, and there is only one chart addition. Thus, the discussions in Section~\ref{sec: cone extensions} and Section~\ref{sec: E and F detrop} are much simplified in this case. 
     
     Now let $\PP_i$ for $i=1,2$ be the positive orthant in their respective vector spaces, so $\PP_1 \cong \R^d_{\geq 0} = \mathrm{Cone}(\{e_1,\cdots,e_d\})$ where $\{e_1,\cdots,e_d\}$ denotes the standard basis for $(\MM_1)_\R \cong \R^d$. Similarly $\PP_2 \cong \R^r_{\geq 0} = \mathrm{Cone}(\{f_1,\cdots,f_r\})$ where $\{f_1,\cdots,f_r\}$ is the standard basis for $(\MM_2)_\R \cong \R^r$. In particular, the inner normal vectors $n^1_i$ defining $\PP_1$ are precisely the standard basis vectors $e_1,\cdots,e_d$, and similarly the $n^2_j$ are $f_1,\cdots,f_r$ for $\PP_2$. For the sake of this example let $\langle \cdot, \cdot \rangle$ denote the standard inner product on either $\R^d$ or $\R^r$; recall that this inner product is the strict dual pairing between $\MM_i$ and $\NN_i$, $i=1,2$, in this trivial-PL case. It can be checked that both cones satisfy the condition of being Gorenstein PL cones by the choices $\mathbb{g}_1 = \sum_i e_i, \mathbb{g}_2 = \sum_j f_j$. The corresponding support functions are $\Psi_{\PP_1} = \min_i \{\langle e_i, \cdot \rangle\}$ and $\Psi_{\PP_2} = \min_j \{\langle f_j, \cdot \rangle \}$.  
     With this in mind, we can provide a dictionary between the discussion in \cite[Section 8]{EscobarHaradaManon-PL} and that in Section~\ref{sec: cone extensions}. The sets $\mathbb{M}_{d,r}$ \cite[near Lemma 8.1]{EscobarHaradaManon-PL} and $T_{d,r}$ \cite[Section 8.2]{EscobarHaradaManon-PL} are our sets $Y$ and $X^\vee$ respectively. In Section~\ref{sec: cone extensions} we needed the cones $C_{\sigma,j} \times C_\tau$. In our case, the coordinates charts $\sigma,\tau$ do not appear since our initial polyptych lattices are trivial, but for a fixed $j \in [r]$, the cone $C_{\sigma,j} \times C_\tau = C_j$ corresponds to the set $\mathbb{M}_{d,r}(j)$ of \cite[near Lemma 8.1]{EscobarHaradaManon-PL}.
     The space $\mathrm{Hom}(C_j,\Z)$ may be identified with the set labelled $M_{d,r}^{(j)}$ \cite[Section 8.1, first displayed equation]{EscobarHaradaManon-PL}, and a computation shows that the mutation maps obtained in Section~\ref{sec: cone extensions} agree with those given between the $M_{d,r}^{(j)}$ in \cite[Section 8.1]{EscobarHaradaManon-PL}.  Finally, it is straightforward to see that the detropicalization constructed in Section~\ref{sec: E and F detrop} also agrees. Indeed, since $\PP_1$ is the positive orthant, and the detropicalization of the trivial polyptych lattice $\Z^d$ is the Laurent polynomial ring $\C[x_1^{\pm}, \cdots, x_d^{\pm}]$ (and similarly for $\Z^r$ and $\C[t_1^{\pm},\cdots, t_r^{\pm}]$), it follows that the subring $\Aa_{\PP_1}$ is isomorphic to $\C[x_1,\cdots,x_d]$,  the usual polynomial ring.  Consequently we have $\Aa_{\PP_1 \times \NN_2} \cong \C[x_1,\cdots,x_d, t_1^{\pm}, \cdots, t_r^{\pm}]$. The formula given for the detropicalization of $\MM_{d,r}$ \cite[Section 8.5]{EscobarHaradaManon-PL} takes this ring and mods out by the relation $(\prod_i x_i) - (t_1+\cdots+t_r)$. It is straightforward to see that $\prod_i x_i$ is precisely the element $\hat{\mathbb{b}}_{(\mathbb{g}_1,0)}$ in this case since $\mathbb{g}_1=(1,1,1\cdots,1) \in \Z^d$ (so its corresponding monomial is indeed $\prod_i x_i$). Finally, since the inner normal vectors for $\PP_2$ are $f_1,\cdots,f_r$, it follows that $\sum_j \hat{\mathbb{b}}_{(0,n^2_j)}$ is precisely $t_1+\cdots+t_r$. The analogous discussion holds for $\Aa_{\PP_2 \times \NN_1}$.

\end{example}

\begin{example}\label{example: classical GF pairs}
    For any classical Gorenstein-Fano polytope $\Delta$ of dimension $n$, it is not hard to see that the cone over $\Delta$ in $n+1$ dimensions is a Gorenstein PL cone (in the trivial polyptych lattice of rank $n+1$), temporarily denoted $\PP(\Delta)$. (Note that since the polyptych lattice is trivial, the PL cone is in fact a classical Gorenstein cone.) Given two such Gorenstein-Fano polytopes $\Delta, \Delta'$ of dimensions $n, n'$ respectively, we may consider their corresponding Gorenstein PL cones $\PP = \PP(\Delta), \PP' = \PP(\Delta')$ in trivial polyptych lattices of rank $n+1, n'+1$ respectively. We may then denote by $\mathcal{E}(\Delta,\Delta'), \mathcal{F}(\Delta,\Delta')$ the strict dual pair of polyptych lattices obtained by applying our Gorenstein PL cone extension construction. 

    We sketch a simple example below with $n=n'=2$. 

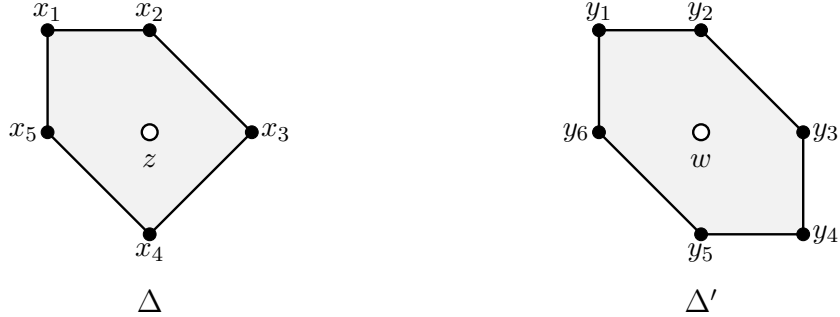
\begin{figure}[ht]\label{fig-GFpolygons}
\centering

\begin{tikzpicture}[
    scale=1.35,
    polygon/.style={
        draw=black,
        line width=0.9pt,
        fill=black!5,
        line join=round
    },
    vertex/.style={
        circle,
        fill=black,
        draw=black,
        inner sep=1.7pt
    },
    origin/.style={
        circle,
        fill=white,
        draw=black,
        line width=0.9pt,
        inner sep=2.1pt
    },
    vlabel/.style={
        font=\large
    }
]

\begin{scope}[xshift=-2.7cm]

    \draw[polygon]
        (-1,1)
        -- (0,1)
        -- (1,0)
        -- (0,-1)
        -- (-1,0)
        -- cycle;

    \node[vertex] at (-1,1) {};
    \node[vertex] at ( 0,1) {};
    \node[vertex] at ( 1,0) {};
    \node[vertex] at ( 0,-1) {};
    \node[vertex] at (-1,0) {};

    \node[vlabel, above]       at (-1,1) {$x_1$};
    \node[vlabel, above]       at ( 0,1) {$x_2$};
    \node[vlabel, right]       at ( 1,0) {$x_3$};
    \node[vlabel, below]       at ( 0,-1) {$x_4$};
    \node[vlabel, left]        at (-1,0) {$x_5$};

    \node[origin] at (0,0) {};
    \node[vlabel, below=5pt] at (0,0) {$z$};

    \node[font=\Large] at (0,-1.65) {$\Delta$};

\end{scope}

\begin{scope}[xshift=2.7cm]

    \draw[polygon]
        (-1,1)
        -- (0,1)
        -- (1,0)
        -- (1,-1)
        -- (0,-1)
        -- (-1,0)
        -- cycle;

    \node[vertex] at (-1,1) {};
    \node[vertex] at ( 0,1) {};
    \node[vertex] at ( 1,0) {};
    \node[vertex] at ( 1,-1) {};
    \node[vertex] at ( 0,-1) {};
    \node[vertex] at (-1,0) {};

    \node[vlabel, above]       at (-1,1) {$y_1$};
    \node[vlabel, above]       at ( 0,1) {$y_2$};
    \node[vlabel, right]       at ( 1,0) {$y_3$};
    \node[vlabel, right] at ( 1,-1) {$y_4$};
    \node[vlabel, below]  at ( 0,-1) {$y_5$};
    \node[vlabel, left]        at (-1,0) {$y_6$};

    \node[origin] at (0,0) {};
    \node[vlabel, below=5pt] at (0,0) {$w$};

    \node[font=\Large] at (0,-1.65) {$\Delta'$};

\end{scope}

\end{tikzpicture}

\end{figure}

The inner normals of the cone $\PP(\Delta)$ are $(1,1,1), (1,0,1), (0,-1,1),(-1,-1,1),(-1,1,1)$; these are the vertices of $\Delta^\vee$ placed at height $1$. Similarly, the inner normals of the cone $\PP(\Delta')$ are $(1, 1, 1), (1,0,1),\\ (0,-1,1), (-1,-1,1), (-1,0, 1), (0, 1, 1)$.
Following the construction given in Section~\ref{sec: cone extensions} and Section~\ref{sec: E and F detrop}, one obtains that a detropicalization of $\mathcal{E}(\Delta,\Delta')$ is given as follows: set 
\[I
=
\bigl(
x_1x_3-x_2z,\,
x_1x_4-x_5z,\,
x_1z-x_2x_5,\,
x_2x_4-z^2,\,
x_3x_5-z^2,\]

\[z - t_1t_2s- t_1s- t_2^{-1}s- t_1^{-1}t_2^{-1}s- t_1^{-1}s- t_2s
\bigr).\]
Then the detropicalization $\Aa_{\mathcal{E}(\Delta,\Delta')}$ from Section~\ref{sec: E and F detrop} is given by 
\[
\Aa_{\mathcal{E}(\Delta,\Delta')} 
\cong
\frac{\K[x_1,\ldots,x_5,z, t_1^\pm, t_2^\pm, s^\pm]}{I}.\]
Similarly, if we set 
\[J
=
\bigl(
y_1y_3-y_2w,\,
y_1y_4-w^2,\,
y_1y_5-y_6w,\,
y_1w-y_2y_6,\\
y_2y_4-y_3w,\,
y_2y_5-w^2,\,
y_3y_5-y_4w,\,
y_3y_6-w^2,\,
y_4y_6-y_5w,\]

\[w - u_1u_2v-u_1v-u_2^{-1}v-u_1^{-1}u_2^{-1}v-u_1^{-1}u_2v
\bigr),\]
then $\Aa_{\mathcal{F}(\Delta,\Delta')}$ is given by 
\[
\Aa_{\mathcal{F}(\Delta,\Delta')}
\cong
\frac{\K[y_1,\ldots,y_6,z,u_1^\pm,u_2^\pm,v^\pm]}{J}.\]

\end{example}

\begin{example}\label{example: cluster}
A different family of examples may be obtained by using cluster data to build the initial inputs to our Gorenstein PL cone extension. We only sketch the construction here and leave detailed explanation to forthcoming work \cite{FMEHMM}.

Let $\MM$ be the polyptych lattice associated to a planar finite type cluster algebra. In forthcoming work \cite{FMEHMM} we show that this lattice is in fact self-dual. We let $\Delta$ denote the point-convex hull of the ray generators of $\Sigma(\MM)$. We let $\PP$ denote the cone over $\Delta$ in the polyptych lattice $\MM\times \Z$, where $\Z$ denotes the trivial polyptych lattice of rank $1$.

We may form the Gorenstein PL cone extension using $\MM$, $\PP$, and the Gorenstein PL cone $\Q_{\geq 0}^2$ in the trivial polyptych lattice $\Z^2$. This data yields a rank $4$ polyptych lattice.   Alternatively, we could form the cone extensions associated to two choices $\MM_1,\MM_2$ of polyptych lattices coming from planar cluster algebras and their polytopes $\Delta_1,\Delta_2$; this would yield examples of rank $5$. Below, we record the detropicalizations for the case of $\A_2$ and $\Z^2$, and then $\A_2$ and $\B_2$. The corresponding polytopes of $\A_2$ and $\B_2$ are displayed in the figure. 

\begin{figure}[ht]\label{fig-clustercanonical}
\centering

\begin{tikzpicture}[
    scale=1.35,
    vertex/.style={
        circle,
        fill=black,
        inner sep=1.9pt
    },
    origin/.style={
        circle,
        draw=black,
        fill=white,
        line width=0.8pt,
        inner sep=1.9pt
    },
    ray/.style={
        dotted,
        line width=0.85pt,
        ->,
        >=stealth
    },
    poly/.style={
        draw=black,
        line width=0.8pt,
        fill=black!7,
        line join=round
    },
    vlabel/.style={
        font=\large
    },
    figlabel/.style={
        font=\Large
    }
]

% ============================================================
% LEFT FIGURE: A_2
%
% x_1 = ( 1, 0)
% x_2 = ( 1, 1)
% x_3 = ( 0, 1)
% x_4 = (-1, 0)
% x_5 = ( 0,-1)
% z   = ( 0, 0)
% ============================================================

\begin{scope}[xshift=-2.7cm]

% Convex hull
\draw[poly]
    (1,0)
    -- (1,1)
    -- (0,1)
    -- (-1,0)
    -- (0,-1)
    -- cycle;

% Dotted rays
\draw[ray] (0,0) -- (1.55,0);
\draw[ray] (0,0) -- (1.45,1.45);
\draw[ray] (0,0) -- (0,1.55);
\draw[ray] (0,0) -- (-1.55,0);
\draw[ray] (0,0) -- (0,-1.55);

% Vertices
\node[vertex] at (1,0) {};
\node[vertex] at (1,1) {};
\node[vertex] at (0,1) {};
\node[vertex] at (-1,0) {};
\node[vertex] at (0,-1) {};

% Vertex labels
\node[vlabel, above right=2pt]        at (1,0)  {$x_1$};
\node[vlabel, right=2pt]   at (1,1)  {$x_2$};
\node[vlabel, above left=2pt]   at (0,1)  {$x_3$};
\node[vlabel, above left=2pt]   at (-1,0) {$x_4$};
\node[vlabel, below right=2pt]  at (0,-1) {$x_5$};

% Origin
\node[origin] at (0,0) {};
\node[vlabel, below left=2pt] at (0,0) {$z$};

% Figure label
\node[figlabel] at (0,-1.95) {$\Delta(\A_2)$};

\end{scope}

% ============================================================
% RIGHT FIGURE: B_2
%
% y_1 = ( 1, 0)
% y_2 = ( 1, 1)
% y_3 = ( 1, 2)
% y_4 = ( 0, 1)
% y_5 = (-1, 0)
% y_6 = ( 0,-1)
% w   = ( 0, 0)
% ============================================================

\begin{scope}[xshift=2.7cm]

% Convex hull
\draw[poly]
    (1,0)
    -- (1,2)
    -- (0,1)
    -- (-1,0)
    -- (0,-1)
    -- cycle;

% Dotted rays
\draw[ray] (0,0) -- (1.55,0);
\draw[ray] (0,0) -- (1.55,1.55);
\draw[ray] (0,0) -- (1.25,2.50);
\draw[ray] (0,0) -- (0,2.45);
\draw[ray] (0,0) -- (-1.55,0);
\draw[ray] (0,0) -- (0,-1.55);

% Points
\node[vertex] at (1,0) {};
\node[vertex] at (1,1) {};
\node[vertex] at (1,2) {};
\node[vertex] at (0,1) {};
\node[vertex] at (-1,0) {};
\node[vertex] at (0,-1) {};

% Vertex labels
\node[vlabel, above right=2pt]        at (1,0)  {$y_1$};
\node[vlabel, right=2pt]        at (1,1)  {$y_2$};
\node[vlabel, right=2pt]  at (1,2)  {$y_3$};
\node[vlabel, above left=2pt]   at (0,1)  {$y_4$};
\node[vlabel, above left=2pt]   at (-1,0) {$y_5$};
\node[vlabel, below right=2pt]  at (0,-1) {$y_6$};

% Origin
\node[origin] at (0,0) {};
\node[vlabel, below left=2pt] at (0,0) {$w$};

% Figure label
\node[figlabel] at (0,-1.95) {$\Delta(\B_2)$};

\end{scope}

\end{tikzpicture}
\caption{The canonical polytopes of $\A_2$ and $\B_2$.}
\end{figure}
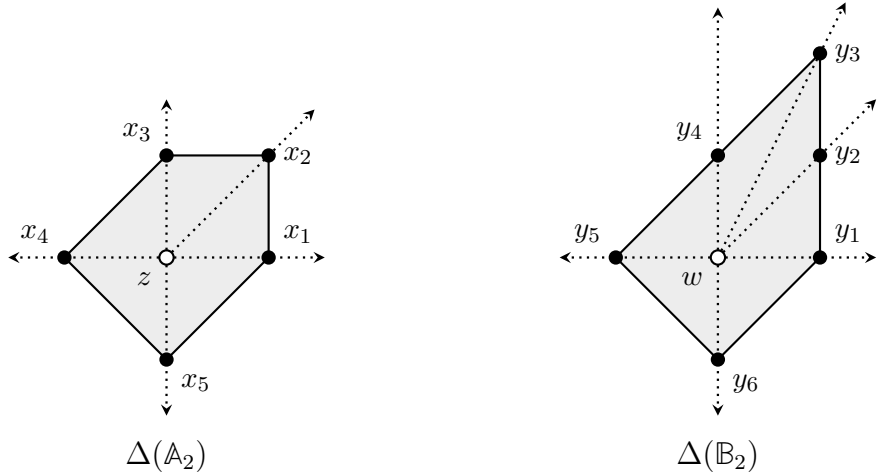

For the equations below, we let $\PP(\A_2)$ denote the Gorenstein PL cone obtained by taking the cone over $\Delta(\A_2)$. Taking the Gorenstein PL cone extension of $\A_2$ with the trivial lattice $\Z^2$, using the Gorenstein PL cones $\PP(\A_2)$ and $\Q_{\geq 0}^2$, and applying the methods of Section~\ref{sec: E and F detrop} yields the following pair of algebras as the detropicalizations of the associated Gorenstein PL cone extensions.\\

\[
\Aa_{\E(P_{can}(\A_2),\Z_{\geq 0}^2)}
\cong
\frac{
\K[x_1,x_2,x_3,x_4,x_5,z^{\pm},Y_1,Y_2]
}{\begin{gathered}
x_1x_3-x_2-1,\\
x_2x_4-x_3-1,\\
x_3x_5-x_4-1,\\
x_4x_1-x_5-1,\\
x_5x_2-x_1-1,\\
Y_1Y_2z^{-1}-(x_1+x_2+x_3+x_4+x_5)
\end{gathered}}.
\]

\vspace{.5in}

\[
\Aa_{\mathcal{F}(\PP(\A_2),\Z_{\geq 0}^2)}
\cong
\frac{
\K[x_1,x_2,x_3,x_4,x_5,z,Y_1,Y_2]
}{\begin{aligned}
x_1x_3 &- x_2z - z^2,\\
x_2x_4 &- x_3z - z^2,\\
x_3x_5 &- x_4z - z^2,\\
x_4x_1 &- x_5z - z^2,\\
x_5x_2 &- x_1z - z^2,\\
z &- Y_1-Y_2
\end{aligned}}.
\]

If we pair $\A_2$ and $\B_2$ using the Gorenstein PL cones $\PP(\A_2)$ and $\PP(\B_2)$, then we get the following pair of algebras. 

\[
\Aa_{\E(P_{can}(\A_2),P_{can}(\B_2))}
\cong
\frac{
\K[x_1,\ldots,x_5,z,y_1,\ldots,y_6,w^{\pm}]
}{\begin{gathered}
x_1x_3-x_2z-z^2,\\
x_2x_4-x_3z-z^2,\\
x_3x_5-x_4z-z^2,\\
x_4x_1-x_5z-z^2,\\
x_5x_2-x_1z-z^2,\\
y_1y_3-y_2-1,\\
y_2y_4-y_3^2-1,\\
y_3y_5-y_4-1,\\
y_4y_6-y_5^2-1,\\
y_5y_1-y_6-1,\\
y_6y_2-y_1^2-1,\\
zw^{-1}-(y_1+\cdots+y_6)
\end{gathered}}.
\]

\vspace{.5in}

\[
\Aa_{\mathcal{F}(P_{can}(\A_2),P_{can}(\B_2))}
\cong
\frac{
\K[x_1,\ldots,x_5,z^{\pm},y_1,\ldots,y_6,w]
}{\begin{gathered}
x_1x_3-x_2-1,\\
x_2x_4-x_3-1,\\
x_3x_5-x_4-1,\\
x_4x_1-x_5-1,\\
x_5x_2-x_1-1,\\
y_1y_3-y_2w-w^2,\\
y_2y_4-y_3^2-w^2,\\
y_3y_5-y_4w-w^2,\\
y_4y_6-y_5^2-w^2,\\
y_5y_1-y_6w-w^2,\\
y_6y_2-y_1^2-w^2,\\
wz^{-1}-(x_1+\cdots+x_5)
\end{gathered}}.
\]\\

\end{example}

\begin{example}\label{example: iterating}
Proposition~\ref{proposition: iteration} shows that when the cones $\PP_1, \PP_2$ used as the initial data for the cone extension satisfy the additional hypothesis that $\PP_1^\vee, \PP_2^\vee$ are also Gorenstein, then we obtain PL cones in the cone extensions that are also Gorenstein. This opens the possibility that we can iterate the extension. This implies that the family of examples mentioned in Example~\ref{example: classical GF pairs} may be extended iteratively, and a larger family can be obtained.  Relatedly, it may be interesting to ask which polyptych lattices can be obtained by this process. 
\end{example}

We conclude this section with several open-ended questions which we leave for future work. 

\begin{question} 
Among the examples mentioned in Example~\ref{example: classical GF pairs} to Example~\ref{example: iterating} which ones are smooth (i.e. the $\Spec$ of the relevant detropicalizations are smooth)? Which ones are unique factorization domains? Are there combinatorial criteria for smoothness or UFD-ness? 
\end{question}

\begin{question} 
We have produced many examples of (strict dual pairs of) polyptych lattices, but we do not know an efficient way to classify them, or more specifically, we do not know combinatorial (or geometric) methods to determine when two examples are isomorphic in a suitable sense. Do the examples obtained from our constructions yield any hints toward a general classification theory? 
\end{question} 

\begin{question} 
In Example~\ref{example: iterating}, we noted that we can start with the initial data of a Gorenstein PL cone whose dual is also Gorenstein, and iterate the process of Gorenstein PL cone extensions to obtain a potentially infinite family of strict dual polyptych lattices. What properties do all members of this family have in common? 
\end{question} 

\begin{question}
In the theory of cluster algebras, a theme that arises is to realize cluster algebras as decategorifications of certain categories arising from representation theory. See e.g. \cite{Keller} and references therein; a sample of recent developments of which we are aware are \cite{NC, Pressland}.  We ask whether polyptych lattices and their associated detropicalizations can be given a similar interpretation. With this in mind, we additionally ask if there is a natural categorical operation that corresponds to our Gorenstein PL cone extension construction. 
\end{question}

%%%%%%%%%%%%%%%%%%%

%%%%%%%%%%%%%%%%%%%%%%
%\bibliographystyle{amsplain}
%\bibliography{biblio}

\end{document}